\documentclass{article}

\PassOptionsToPackage{numbers, compress}{natbib}
\usepackage[final]{neurips_2021}

\usepackage{graphicx}
\usepackage[export]{adjustbox}
\usepackage{amsmath,amssymb,amsthm,enumerate,color,ifthen,graphicx,overpic,bm}		
\usepackage{xargs}
\usepackage{subfigure}
\usepackage{comment}
\usepackage{algorithmic}

\usepackage{booktabs,colortbl}

\newcolumntype{a}{>{\columncolor{green!30}}c}
\newcolumntype{R}{>{\columncolor{red!30}}c}

\usepackage{wrapfig}

\usepackage[colorlinks=true,
linkcolor=blue,
urlcolor=blue,
citecolor=blue]{hyperref}
 \usepackage[textwidth=4cm, textsize=footnotesize]{todonotes}
\usepackage[utf8]{inputenc}

\usepackage{cleveref}

\usepackage[linesnumbered,vlined,ruled,commentsnumbered]{algorithm2e}
\usepackage[shortlabels]{enumitem}
\usepackage{dashundergaps}
\usepackage{aliascnt,bbm}

\usepackage{tabularx}
\usepackage{wrapfig}
\usepackage{enumerate}

\makeatletter
\newcommand{\settitle}{\@maketitle}
\makeatother

\def\nset{{\mathbb{N}}}
\def\rset{\mathbb R}

\def\Zset{\mathsf{Z}}

\def\rmd{\mathrm{d}}
\def\argmin{\operatorname{Argmin}}

\def\max{\mathrm{max}}

\def\1{\mathsf{1}} 

\def\PP{\mathbb{P}} 
\def\PE{\mathbb{E}} 

\newcommand{\F}{\mathcal{F}} 
\def\lawPP{\mathbb P_{\mathrm{PP}}}

\def\VRFEDEM{\texttt{VR-FedEM}}
\def\FEDEM{\texttt{FedEM}}
\def\Q{\mathrm{Quant}}
\newcommand{\QEM}[1]{\ensuremath{\mathsf{Q}_{#1}}}
\newcommand{\pscal}[2]{\left\langle#1,#2\right\rangle}

\newcommand{\eqdef}{\ensuremath{:=}}
\newcommand{\kmax}{k_\mathrm{max}}
\newcommand{\kouter}{k_\mathrm{out}}
\newcommand{\kin}{k_\mathrm{in}}

\newcommand\curr{\mathrm{curr}}
\newcommand\init{\mathrm{init}}
\newcommand{\R}{\mathsf{R}}
\newcommand{\loss}[1]{\ensuremath{\mathcal{L}_{#1}}}
\newcommand{\bars}{\bar{\mathsf{s}}}

\newcommand{\mf}{\mathsf{h}}

\newcommand{\hatS}{\widehat{S}}
\newcommand{\Smem}{\mathsf{S}}

\newcommand{\CE}{\mathrm{CE}}
\newcommand{\opt}{\mathrm{opt}}

\newcommand{\param}{\theta}
\newcommand{\Param}{\Theta}
\newcommand{\map}{\mathsf{T}}
\newcommand{\lyap}{\operatorname{W}}
\newcommand{\batch}{\mathcal{B}}
\newcommand{\lbatch}{\mathsf{b}}
\newcommand{\pas}{\gamma}

\def\pa{\mathsf{a}}

\def\set{\mathcal A}

\newcommand\sequence[3] {\ifthenelse{\equal{#3}{}}{\ensuremath{\{
#1_{#2}\}}}{\ensuremath{\{ #1^{#2}, \eqsp #2 \in #3 \}}}}
\newcommand\sequencedown[3] {\ifthenelse{\equal{#3}{}}{\ensuremath{\{
#1_{#2}\}}}{\ensuremath{\{ #1_{#2}, \eqsp #2 \in #3 \}}}}

\def\eqsp{\;}
\newcommand{\ie}{i.e.}

\newcommand{\coint}[1]{\left[#1\right)}
\newcommand{\ocint}[1]{\left(#1\right]}
\newcommand{\ooint}[1]{\left(#1\right)}
\newcommand{\ccint}[1]{\left[#1\right]}

\newcommandx{\CPE}[3][1=]{{\mathbb E}_{#1}\left[\left. #2 \middle \vert #3 \right. \right]} 
\newcommandx{\CPVar}[3][1=]{\mathrm{Var}^{#3}_{#1}\left\{ #2 \right\}}
\newcommand{\CPP}[3][]
{\ifthenelse{\equal{#1}{}}{{\mathbb P}\left(\left. #2 \, \right| #3 \right)}{{\mathbb P}_{#1}\left(\left. #2 \, \right | #3 \right)}}

\newtheorem{assumption}{A\hspace{-3.1pt}}
\Crefname{assumption}{A\hspace{-3pt}}{A\hspace{-3pt}}
\crefname{assumption}{A}{A}

\newtheorem{theorem}{Theorem}
\newaliascnt{proposition}{theorem}
\newtheorem{proposition}[proposition]{Proposition}
\aliascntresetthe{proposition}
\newaliascnt{lemma}{theorem}
\newtheorem{lemma}[lemma]{Lemma}
\aliascntresetthe{lemma}
\newaliascnt{corollary}{theorem}
\newtheorem{corollary}[corollary]{Corollary}
\aliascntresetthe{corollary}

\newaliascnt{definition}{theorem}

\aliascntresetthe{definition}
\newaliascnt{remark}{theorem}

\aliascntresetthe{remark}

\title{Federated Expectation Maximization with heterogeneity
  mitigation and variance reduction}
 \author{ Aymeric Dieuleveut
  \\ Centre de Math\'{e}matiques Appliqu\'{e}es \\ Ecole
  Polytechnique, France\\ Institut Polytechnique de Paris
  \\ \texttt{aymeric.dieuleveut@polytechnique.edu} \And Gersende Fort
  \\ Institut de Mathématiques de Toulouse \\ Université de Toulouse;
  CNRS \\ UPS, Toulouse, France
  \\ \texttt{gersende.fort@math.univ-toulouse.fr} \\ \And Eric
  Moulines \\ Centre de Math\'{e}matiques Appliqu\'{e}es \\ Ecole
  Polytechnique, France\\ CS Dpt, HSE University, Russian Federation
  \\ \texttt{eric.moulines@polytechnique.edu} \\ \And Genevi\`eve
  Robin \\  Laboratoire de Mathématiques \\ et Modélisation d'Évry  \\ Université d'Évry Val d'Essonne; CNRS\\ Évry-Courcouronnes,
  France\\ \texttt{genevieve.robin@cnrs.fr} }

\begin{document}
\maketitle

\begin{abstract}
  The Expectation Maximization (EM) algorithm is the default algorithm
  for inference in latent variable models. As in any other field of
  machine learning, applications of latent variable models to very
  large datasets makes the use of advanced parallel and distributed
  architectures mandatory. This paper introduces \FEDEM, which is the
  first extension of the EM algorithm to the federated learning
  context. \FEDEM~is a new communication efficient method, which
  handles partial participation of local devices, and is robust to
  heterogeneous distributions of the datasets. To alleviate the
  communication bottleneck, \FEDEM~compresses appropriately defined
  complete data sufficient statistics. We also develop and analyze an
  extension of \FEDEM~to further incorporate a variance reduction
  scheme. In all cases, we derive finite-time complexity bounds for
  smooth non-convex problems.  Numerical results are presented to
  support our theoretical findings, as well as an application to
  federated missing values imputation for biodiversity monitoring.
\end{abstract}

\section{Introduction}
 The Expectation Maximization (EM) algorithm is the most popular
 approach for inference in latent variable models.  The EM algorithm,
 a special instance of the Majorize/Minimize algorithm
 \cite{lange2016mm}, was formalized by \cite{Dempster:em:1977} and is
 without doubt one of the fundamental algorithms in machine
 learning. Applications include among many others finite mixture
 analysis, latent factor models inference, and missing data
 imputation; see
 \cite{xu1996convergence,murphy2002dynamic,maclachlan:2008,book:EM:mixture:2019}
 and the references therein. As in any other field of machine
 learning, training latent variable models on very large datasets make
 the use of advanced parallel and distributed architectures
 mandatory. Federated Learning (FL)
 \cite{konevcny2016federated,yang2019federated}, which exploits the
 computation power of a large number of edge devices to perform
 distributed machine learning, is a powerful framework to achieve this
 goal.

The conventional EM algorithm is not suitable for FL settings. We
propose several new distributed versions of the EM algorithm
supporting compressed communication.  More precisely, our objective is
to minimize a non-convex finite-sum smooth objective function
\begin{equation}\label{eq:problem}
\argmin_{\param\in \Param} F(\param), \qquad F(\param) \eqdef
\frac{1}{n}  \sum_{i=1}^n \loss{i}(\param) + \R(\param) \eqsp, \qquad \quad \Theta
\subseteq \rset^d \eqsp,\vspace{-0.5em}
\end{equation}
where $n$ is the number of workers/devices which are connected to a
central server, and the worker $\#i$ only has access to its local data; finally $\R$ is a penalty term which may be
introduced to promote sparsity, regularity, etc. In latent variable
models, $\loss{i}(\param) = - m^{-1} \sum_{j=1}^{m}\log p
(y_{ij};\param)$, where $\lbrace y_{ij}\rbrace_{j=1}^{m}$ are the $m$
observations available for worker $\# i$, and $ p (y;\param)$ is the
\emph{incomplete} likelihood.  $p(y;\param)$ is defined by
marginalizing the \emph{complete-data} likelihood $p(y,z;\param)$
defined as the joint probability density function of the observation
$y$ and a non-observed latent variable $z \in \Zset$,
i.e. $p(y;\param) = \int_\Zset p(y,z;\param) \mu(\rmd z)$ where
$\Zset$ is the \emph{latent space} and $\mu$ is a measure on $\Zset$.
We focus in this paper on the case where $p(y,z;\theta)$ belongs to a
curved exponential family, given by
\begin{equation}
\label{eq:complete-likelihood}
p (y, z;\param) \eqdef \rho( y,z) \exp \big\{ \pscal{ s(y,z) }{ \phi( \param) } - \psi(\param) \big\} \eqsp;
\end{equation}
where $s(y,z) \in \rset^q$ is the \emph{complete-data sufficient
  statistics}, $\phi: \Param \to \rset^q$ and $\psi : \Param
\rightarrow \rset$, $\rho: {\sf Y} \times \Zset \rightarrow \rset^+$
are vector/scalar functions.

In absence of communication constraints, the EM algorithm is a popular
method to solve \eqref{eq:problem}. It alternates between two steps: in the Expectation ({\sf E})
step, using the current value of the iterate $\param_\curr$, it
computes a majorizing function $\param \mapsto
\QEM{}(\param,\param_\curr)$ given up to an additive constant by
\vspace{-0.5em}
\begin{equation} \label{eq:majorize}
\QEM{}(\param, \param_\curr) \eqdef - \pscal{ \bars( \param_\curr) }{
  \phi( \param ) } + \psi(\param) + \R( \param ) \quad \text{where} \quad \bars(
\param ) \eqdef \frac{1}{n} \sum_{i=1}^n \bars_i( \param ) \eqsp;\vspace{-0.5em}
\end{equation}
 and $\bars_i( \param )$ is the $i$th device conditional expectation
 of the complete-data sufficient statistics:
\begin{equation}
\label{eq:definition-bar-s}
\bars_i( \param ) \eqdef \frac{1}{m}\sum_{j=1}^{m} \bars_{ij}(\param)
\eqsp, \quad \bars_{ij}(\param) \eqdef \int_{\Zset} s(y_{ij},z) p( z
\vert y_{ij}; \param ) \mu( \rmd z ) \eqsp, \quad \eqsp
\vspace{-0.5em}
\end{equation}
where $p(z|y_{ij}; \param) \eqdef p(y_{ij},z;\param)/ p(y_{ij};
\param)$.  As for the {\sf M} step, an updated value of $\param_\curr$
is computed as a minimizer of $\param \mapsto \QEM{}(\param,
\param_\curr)$. The majorizing function is then updated with the new
$\param_\curr$; this process is iterated until convergence.  The EM
algorithm is most useful when for any $\param_\curr \in \Param$, the
function $\param \mapsto \QEM{}(\param, \param_\curr)$ is a convex
function of the parameter $\theta$ which is solvable in $\theta$
either explicitly or with little computational effort.  A major
advantage of the EM algorithm stems from its invariance under
homeomorphisms, contrary to classical first-order methods: the EM
updates are the same for any continuous invertible re-parametrization
\cite{kunstner2021homeomorphic}.

In the FL context, the vanilla EM algorithm is affected by three major
problems: (1) the communication bottleneck, (2) data heterogeneity,
and (3) partial participation (PP) of the workers.

When the number of workers is large, the cost of communication becomes
overwhelming. A classical technique to alleviate this problem is to
use \emph{communication compression}. Most FL algorithms are first
order methods and compression is typically applied to stochastic
gradients.  Yet, these methods are not appropriate to
solve~\eqref{eq:problem} since \textit{(i)} they do not preserve the
desirable homeomorphic invariance property, and \textit{(ii)} the full
EM iteration is not distributed since the M step is performed by the
central server only. This calls for an extension of the EM algorithm
to the FL setting.

Since workers are often user personal devices, the issue of data
heterogeneity naturally arises.
Our model in \Cref{eq:problem,eq:majorize,eq:definition-bar-s} allows the local loss functions to depend on the worker $ i \in \{1, \ldots, n\}$ and the observations $ y_{ij} $ to be independent but not necessarily identically distributed. In addition, our theoretical results deal with specific behaviors for each worker $ i \in \{1, \ldots, n\} $, see e.g., \Cref{hyp:lipschitz,hyp:variance:oracle,hyp:DS:lipschitz}.
In the FL-EM setting, heterogeneity
manifests itself by the non-equality of the \emph{local} conditional
expectations of the complete-data sufficient statistics $\bars_i$'s;
modifications to the algorithms must be performed to ensure convergence at the central server. 

Finally, a subset of users are potentially inactive in
each learning round, being unavailable or unwilling
to participate. Thus, taking into account PP of the
workers and its impact on the convergence of algorithms, is a
major issue.

\begin{itemize}[topsep=0pt,itemsep=1pt,leftmargin=*,noitemsep,wide]
\item \textbf{\FEDEM. }The main contribution of our paper is a new
  method called \FEDEM, supporting communication compression, partial
  participation and data heterogeneity. In this algorithm, the workers
  compute an estimate of the \emph{local complete-data sufficient
    statistics} $\bars_i$ using a minibatch of data, apply an unbiased
  compression operator to a noise compensated version (using a
  technique inspired by
  \cite{horvath2019stochastic,gorbunov2020unified}) and send the
  result to the central server, which performs aggregation and the
  M-step (i.e. the parameter update).
\item \textbf{\VRFEDEM. } We improve \FEDEM~by adding a variance
  reduction method inspired by the \texttt{SPIDER} framework
  \cite{fang:etal:2018} which has recently been extended to the EM
  framework \cite{SPIDER-EM}.  For both \FEDEM\ and \VRFEDEM, the
  central server updates the expectations of the global complete-data
  sufficient statistics through a Stochastic Approximation procedure
  \cite{benveniste:etal:1990,borkar:2008}.  When compared to \FEDEM,
  \VRFEDEM~additionally performs variance reduction for each worker,
  progressively alleviating the variance brought by the random oracles
  which provide approximations of the local complete-data sufficient
  statistics.
\item \textbf{Theoretical analysis.} EM in the curved exponential
  family setting converges to the roots of a function $\mf$ (see
  e.g. \Cref{sec:FEDEM}).  We introduce a unified theoretical
  framework which covers the convergence of \FEDEM\, and
  \VRFEDEM\ algorithms in the non-convex case and establish
  convergence guarantees for finding an $\epsilon$-stationary point
  (see \Cref{theo:dianaem} and~\Cref{theo:DS}). In both cases, we
  provide the number $K_\opt(\epsilon)$ of optimization steps and the
  number $K_\CE(\epsilon)$ of computed conditional expectations
  $\bars_{ij}$'s required to reach $\epsilon$-stationarity. These
  results show that in the Stochastic Approximation steps of
  \VRFEDEM\ , the step sizes are independent of $m$, the number of
  observations per server. Furthermore, the computational cost in
  terms of $\mathcal{K}_{\CE}(\epsilon)$ improves on earlier
  results. In this respect, \VRFEDEM~has the same advantages as
  \texttt{SPIDER} \cite{fang:etal:2018} compared to \texttt{SVRG}
  \cite{johnson:zhang:2013} and \texttt{SAGA}
  \cite{Defazio:bach:2014}, or as \texttt{SPIDER-EM} \cite{SPIDER-EM}
  compared to \texttt{sEM-vr} \cite{chen:etal:2018} and \texttt{FIEM}
  \cite{karimi:etal:2019,fort:gach:moulines:2021}. Lastly, our bounds
  demonstrate the robustness of \FEDEM~and \VRFEDEM~to data
  heterogeneity.
  \item Finally, seen as a root finding algorithm in
    a quantized FL setting, \VRFEDEM\ can be compared to {\tt
      VR-DIANA}~\cite{horvath2019stochastic}: we show that
    \VRFEDEM\ does not require the step sizes to decrease with $m$ and
    provides state of the art iteration complexity to reach a
    precision $\epsilon$.
\end{itemize}

\textbf{Notations.}
For  vectors $a,b $ in  $\rset^q$, $\pscal{a}{b}$ is the Euclidean
 scalar product, and $\|\!\cdot\!\|$ denotes the asso\-ciated
norm. For $r\geq 1$, $\|a\|_r$ is the $\ell_r$-norm of a vector
$a$. The Hadamard product $a \odot b$ denotes the entrywise product of the two vectors
$a,b$. By convention, vectors are column-vectors. For a matrix $A$,
$A^{\top}$ is its transpose and $\|A\|_F$ is its Frobenius norm;
for two matrices $A,B$, $\pscal{A}{B} \eqdef \mathrm{Trace}(B^\top A)$.
For a positive integer $n$, set $[n]^\star \eqdef \{1, \cdots, n\}$
and $[n] \eqdef \{0, \cdots, n\}$. The set of non-negative integers
(resp. positive) is denoted by $\nset$ (resp. $\nset^\star$).
The minimum (resp. maximum) of two real numbers
  $a,b$ is denoted by $a \wedge b$ (resp. $a \vee b$). We will use the
  Bachmann-Landau notation $a(x) = O(b(x))$ to characterize an upper
  bound of the growth rate of $a(x)$ as being~$b(x)$. 

\section{{\tt \FEDEM}:  Expectation Maximization algorithms for federated learning}
\label{sec:FEDEM}
Recall the definition of the negative penalized (normalized)
log-likelihood $F(\param)$ from \eqref{eq:problem}. Along the entire
paper, we make the following assumptions \Cref{hyp:model} to
\Cref{hyp:Tmap},which define the model at hand.
\begin{assumption} \label{hyp:model} The parameter set $\Param \subseteq \rset^d$ is a convex open set.   The functions $\R: \Param \to \rset$, $\phi : \Param \to \rset^q$, $\psi: \Param \to \rset$,  and $\rho(y_{ij},\cdot): \Zset \to \rset_+$, $s(y_{ij},\cdot): \Zset \to \rset^q$ for $i \in [n]^{\star}$ and $j\in [m]^{\star}$ are measurable functions. For any $\param \in \Param$ and $i \in [n]^{\star}$, the log-likelihood is finite: $-\infty<\loss{i}(\param) < \infty$.
\end{assumption}
\begin{assumption} \label{hyp:bars} For all $\param \in \Param$ and
  $i \in [n]^\star$, the conditional expectation $\bars_{i}(\param)$
  is well-defined.
\end{assumption}
\begin{assumption} \label{hyp:Tmap}  For any $s \in \rset^q$, the map $s \mapsto \argmin_{\param \in \Param} \ \left\{  \psi(\param) + \R(\param) -  \pscal{s}{\phi(\param)} \right\}$ exists and is unique; the singleton is denoted by $\{\map(s)\}$.
\end{assumption}

EM defines a sequence $\{\param_k, k \geq 0 \}$ that can be computed
recursively as $ \param_{k+1} = \map \circ \bars (\param_k)$, where
the map $\map$ is defined in \Cref{hyp:Tmap} and $\bars$ is defined in
\eqref{eq:majorize}.  On the other hand, the EM algorithm can be
defined through a mapping in the complete-data sufficient statistics,
referred to as the {\em expectation space}. In this setting, the EM
iteration defines a $\rset^q$-valued sequence $\{\hatS_k, k \geq 0 \}$
given by $ \hatS_{k+1} = \bars \circ \map(\hatS_k)$. Thus, we observe
that the EM algorithm admits two equivalent representations:
\begin{equation}
\label{eq:EM-equiv-param-stat}
    \text{(Parameter space)}~~\param_{k+1} = \map \circ \bars (\param_k); \quad \text{(Expectation space)}~~\hatS_{k+1} = \bars \circ \map(\hatS_k).
\end{equation}
In this paper, we focus on the expectation space representation; see
\cite{kunstner2021homeomorphic} for an interesting discussion on the
connection of EM and mirror descent.  It has been shown in
\citep{Delyon:etal:1999} that if $s_\star$ is a fixed point to the EM
algorithm in the expectation space, then $\param_\star \eqdef
\map(s_\star)$ is a fixed point of the EM algorithm in the parameter
space, i.e., $\param_\star = \map \circ \bars( \param_\star )$; note
that the converse is also true.  Define the functions $\mf_i$ and
$\mf$ from $\rset^q$ to $\rset^q$ by
$ \mf(s) \eqdef \frac{1}{n} \sum_{i=1}^n \mf_i(s) $ with $\  \mf_i(s)
\eqdef \bars_i \circ \map(s) - s \eqsp. $
\begin{equation}
\label{eq:hfield}
\mf(s) \eqdef \frac{1}{n} \sum_{i=1}^n \mf_i(s) \eqsp, \qquad \mf_i(s)
\eqdef \bars_i \circ \map(s) - s \eqsp.
\end{equation}
A key property is that the fixed points of EM in the expectation space
are the roots of the {\em mean field} $s \mapsto \mf(s)$ (see
\eqref{eq:majorize} for the definition of $\bars$). Therefore,
convergence of EM-based algorithms is evaluated in terms of
\textbf{$\epsilon$-stationarity} (see
\cite{ghadimi:lan:2013,SPIDER-EM}): for all $\epsilon>0$, there exists
a (possibly random) termination time $K$ s.t.  $ \PE\Big[ \|
  \mf(\hatS_K) \|^2 \Big] \leq \epsilon \eqsp. $ Another key property
of EM is that it is a monotonic algorithm: each iteration leads to a
decrease of the negative penalized log-likelihood
i.e. $F(\param_{k+1}) \leq F(\param_k)$ or, equivalently in the
expectation space $F \circ \map(\hatS_{k+1}) \leq F \circ
\map(\hatS_k)$ (for sequences $\{\param_k, k \geq 0\}$ and $\{\hatS_k,
k \geq 0\}$ given by
\eqref{eq:EM-equiv-param-stat}). \Cref{hyp:DS:lyap} assumes that the
roots of the mean field $\mf$ are the roots of the gradient of $F
\circ \map$ (see \cite{Delyon:etal:1999} for the same assumption when
studying Stochastic EM). \Cref{hyp:lipschitz} assumes global Lipschitz
properties of the functions $\mf_i$'s.
\begin{assumption}  \label{hyp:DS:lyap} The function $\lyap \eqdef F \circ \map: \rset^q \to \rset$ is
  continuously differentiable on $\rset^q$ and its gradient is globally
  Lipschitz with constant $L_{\dot \lyap}$.  Furthermore, for any $s
  \in \rset^q$,
$
\nabla \lyap(s) = - B(s) \mf(s)
$
where $B(s)$ is a $q \times q$ positive definite matrix. In addition,
there exist $0 < v_{\min} \leq v_\max$ such that for any $s \in
\rset^q$, the spectrum of $B(s)$ is in $\ccint{v_{\min}, v_{\max}}$.
\end{assumption}

\begin{assumption}
  \label{hyp:lipschitz}
  For any $i \in [n]^\star$, there exists $L_i>0$ such that for any $s,s' \in \rset^q$,
 $
\|\mf_i(s) - \mf_i(s') \| = \| (\bars_i \circ \map(s) -s) - (\bars_i
\circ \map(s') - s')\| \leq L_i \| s-s' \| \eqsp.
  $
\end{assumption}

\textbf{A Federated EM algorithm.}
 \begin{wrapfigure}[30]{R}{0.59\textwidth}
  \flushright
 \vspace{-3em}
\begin{minipage}{0.95\linewidth}
			\SetInd{0.25em}{0.3em}
\begin{algorithm}[H]\DontPrintSemicolon
  \caption{\FEDEM~with partial
    participation \label{algo:dianaem-main}} \KwData{ $\kmax \in
    \nset^\star$; for $i \in [n]^\star$, $V_{0,i} \in \rset^q$;
    $\hatS_0 \in \rset^q$; a positive sequence $\{\pas_{k+1}, k \in
         [\kmax-1]\}$; $\alpha>0$; a coefficient $p=\PE_{\mathcal
           A\sim \lawPP}[\mathrm{card}(\mathcal A)]/n$.}  \KwResult{
    The \texttt{\FEDEM-PP} sequence: $\{\hatS_{k}, k \in [\kmax]\}$}
  Set $V_0 = n^{-1} \sum_{i=1}^n V_{0,i}$ \; \For{$k=0, \ldots,
    \kmax-1$}{ Sample $\set_{k+1} \sim \lawPP$
    \label{line:PPworker} \; 
    \For{$i \in \mathcal \set_{k+1} $} {{\em
        (worker $\# i$)}\; 
    Sample $\Smem_{k+1,i}$, an approximation of
      $\bars_i \circ \map(\hatS_k)$ \label{line:SampleSk} \; 
      Set
      $\Delta_{k+1,i} = \Smem_{k+1,i} - V_{k,i} - \hatS_k$ \label{line:diffVk}\;
       Set
      $V_{k+1,i} = V_{k,i} + \alpha \, \Q(\Delta_{k+1,i})$.  \label{line:upV}\; 
      Send
      $\Q(\Delta_{k+1,i})$ to the central server \label{line:compDelta} \;
  }
    \For{$i \notin \mathcal \set_{k+1} $} {
    	{\em (worker $\# i$)}\; 
    	Set
      $V_{k+1,i} = V_{k,i}$ (no update) \;}
   {\em (the central server)}
    \; 
    Set
    $H_{k+1} = V_k + (np)^{-1} \sum_{i\in \set_{k+1}}
    \Q(\Delta_{k+1,i})$ \label{line:reconstruct_H_central}
    \; 
    Set $\hatS_{k+1} = \hatS_k + \pas_{k+1} H_{k+1}$ \label{line:upModel_central}\; 
    Set
    $V_{k+1} = V_k + \alpha n ^{-1} \sum_{i\in \set_{k+1}}
    \Q(\Delta_{k+1,i})$ \label{line:upV_central}
    \; Send $\hatS_{k+1}$ and $\map (\hatS_{k+1})$ to the $n$ workers}
\end{algorithm}
\end{minipage}
\end{wrapfigure}

Our first contribution, the novel algorithm \FEDEM\, is described by \autoref{algo:dianaem-main}. 
The algorithm encompasses partial participation of the workers: at
iteration $\# (k+1)$, only a subset $ \set_{k+1}$ of active workers
participate to the training, see \autoref{line:PPworker}. The averaged
fraction of participating workers is denoted $p$. Each of the active
workers $\#i$ computes an {\em unbiased} approximation $\Smem_{k+1,i} $ (\autoref{line:SampleSk})
of $\bars_i \circ \map(\hatS_k)$; conditionally to the past (see
\Cref{sec:FEDEMwithPP:tribu} for a rigorous definition), these
approximations are independent. The workers then transmit to the
central server a compressed information about the new sufficient
statistics. A naive solution would be to compress and transmit $\Smem_{k+1,i} - \hatS_k$, but  data heterogeneity between servers often prevents these
local differences from vanishing at the optimum, leading to large
compression errors and impairing convergence of the algorithm. Following \cite{mishchenko2019distributed}, a
memory $V_{k,i}$ (initialized to $\mf_i(\hatS_0)$ at $k=0$) is introduced; and the \textit{differences} $\Delta_{k+1,i} \eqdef
\Smem_{k+1,i} - \hat{S}_k - V_{k,i}$ are compressed for $i \in
\set_{k+1}$ (\autoref{line:diffVk} and \autoref{line:compDelta}). These memories are updated locally: $V_{k+1,i} = V_{k,i}
+ \alpha\, \Q(\Delta_{k+1,i})$, at \autoref{line:upV}, with $\alpha > 0$ (typically set to
$1/(1+\omega)$ where $\omega$ is defined in \Cref{hyp:var:quantif}). On
its side, the central server releases an aggregated
estimate $\hatS_{k+1}$ of the complete-data sufficient statistics by
averaging the quantized difference $(np)^{-1} \sum_{i\in \set_{k+1}}
\Q(\Delta_{k+1,i})$ and by adding $V_k$ (\autoref{line:reconstruct_H_central} and \autoref{line:upModel_central}). Then, it  updates $V_{k+1}= V_k + \alpha n^{-1}
\sum_{i=1}^n \Q(\Delta_{k+1,i})$, see \autoref{line:upV_central}.  The final step consists in
solving the {\sf M}-step of the EM algorithm, \ie\ in computing $
\map(\hatS_{k+1})$ (see \Cref{hyp:Tmap}).

We finally state our assumption on the compression process. We consider a large class of \textit{unbiased} compression operators $\Q$ satisfying a variance bound:
\begin{assumption} \label{hyp:var:quantif}   There exists $\omega \!\geq\! 0$ s.t. for any $s \in \rset^q$: $\PE\left[ \Q(s) \right]=s $, and $\PE\left[\| \Q(s) \|^2 \right]
\leq (1+ \omega) \|s \|^2$.
\end{assumption}
Intuitively, the stronger the compression is, the larger $\omega$ will
be. Remark that if no compression is used, or equivalently for all
$ s\in \rset^q $, $ \Q(s) = s $, then \Cref{hyp:var:quantif} is
satisfied with $\omega=0$.  An example of quantization operator
satisfying \Cref{hyp:var:quantif} is the random dithering that can be
described as the random operator $\Q: \rset^q \to \rset^q$,
$\Q(x)= (1/s_{\operatorname{quant}}) \|x\|_{r} \,
\operatorname{sign}(x) \odot \left\lfloor s_{\operatorname{quant}}
  (|x|/\|x\|_r)+\xi \right\rfloor$
where $r \geq 1$ is user-defined, $\xi$ is a uniform random variable
on $[0,1]^q$ and $s_{\operatorname{quant}} \in \nset^\star$ is the
number of levels of roundings; see
\cite{horvath2019stochastic,alistarh2018convergence}. This operator
satisfies \Cref{hyp:var:quantif} with
$\omega= s_{\operatorname{quant}}^{-1} O(q^{1/r} + q^{1/2})$; see
\cite[Example~1]{horvath2019stochastic}.  Another example, namely the
block-$p$-quantization, is provided in the supplemental (see
\Cref{app:quantization}). More generally, this assumption is valid for
many compression operators, for example resulting in
sparsification~\cite[see. e.g.][]{mishchenko2019distributed}.

The convergence analysis is under the following
  assumptions on the oracle $\Smem_{k+1,i}$: for any $i \in
          [n]^\star$, the approximations $\Smem_{k+1,i}$ are unbiased
          and their conditional variances are uniformly bounded in
          $k$. For each $k \in \nset$, denote by $\F_k$ the
          $\sigma$-algebra generated by $\{
            \Smem_{\ell,i}, \set_\ell ; i \in [n]^\star, \ell \in [k]
            \}$ and including the randomness inherited from the quantization
            operator $\Q$ up to iteration $\# k$.
\begin{assumption} \label{hyp:variance:oracle}
For all $k \in \nset$, conditional to $\F_k$, $\{\Smem_{k+1,i}
\}_{i=1}^n$ are independent. Moreover, for any $i \in [n]^\star$,
$\CPE{\Smem_{k+1,i}}{\F_{k}}= \bars_i \circ \map(\hatS_k)$ and there
exists $\sigma_i^2>0$ such that for any $k \geq 0$ $\CPE{\|
  \Smem_{k+1,i} - \bars_i \circ \map(\hatS_k) \|^2}{\F_k} \leq
\sigma_i^2$.
\end{assumption}
\Cref{hyp:variance:oracle} covers both the finite-sum setting
described in the introduction, and the online setting.  In the
finite-sum setting, $\bars_i$ is of the form
$m^{-1} \sum_{j=1}^m \bars_{ij}$. In that case, $\Smem_{k+1,i}$ can be the
sum over a minibatch $\batch_{k+1,i}$ of size $\lbatch$ sampled at
random in $[m]^\star$, with or without replacement and independently
of the history of the algorithm: we have
$\Smem_{k+1,i} = \lbatch^{-1} \sum_{j \in \batch_{k+1,i}} \bars_{ij}
\circ \map(\hatS_{k})$.
In the online setting, the oracles $\Smem_{k+1,i}$ come from an online
processing of streaming informations; in that case $\Smem_{k+1,i}$ can
be computed from a minibatch of independent examples so that the
conditional variance $\sigma_i^2$, which will be inversely
proportional to the size of the minibatch, can be made arbitrarily
small.

\textbf{Reduction of communication complexity for FL.} Reducing the communication cost between workers is a crucial aspect of the FL approach~\cite{kairouz_advances_2019}. In gradient based optimization, four techniques have been used to reduce the amount of communication: (i) increasing the mini\-batch size and reducing the number of iterations, (ii) increasing the number of \textit{local steps} between two communication rounds, (iii) using compression, (iv) sampling clients at each step. Here, we provide a tight analysis of strategies (i),  (iii) and (iv) (sampling client is part of PP). 

Regarding the interest of performing multiple iterations (ii), as
analyzed for example
in~\cite{karimireddy_scaffold_2019,mcmahan_communication-efficient_2017}
for the classical gradient settings, we note that: first, from a
theoretical standpoint, tradeoffs between larger minibatch and more
local iterations are unclear~\cite{2020arXiv200207839W}. Secondly,
\textit{performing local iterations is not possible in the EM
  setting}: one iteration of EM is the combination of two steps E and
M and the M step, which required the use of the map $T$, is only
performed by the central server; this remark is a fundamental
specificity of the EM framework (which is not shared by the gradient
framework).  In applications, we usually do not want $T$ to be
available at each local node. However, our work allows to perform
multiple local iterations of the E step before communicating with the
central server. In \autoref{algo:dianaem-main}, the local statistics
$S_{k+1,i}$ are general enough to cover this case; see the comment
above on ~\Cref{hyp:variance:oracle}.

Finally, as we do not perform local full EM iterations, we do not face
the well-identified \textit{client-drift} challenge (in the presence
of heterogeneity). Yet, we stress that combining compression and
heterogeneity results in other challenges: it is known in the Gradient
Descent setting (see
e.g. \cite{mishchenko2019distributed,philippenko2020artemis}), that
heterogeneity strongly hinders convergence in the presence of
compression. To alleviate the impact of heterogeneity, we introduce
the $ V_{k,i}$'s memory-variables.

\textbf{Convergence analysis, full participation regime.}  In this paragraph, we focus on the \textit{full-participation regime} ($p=1$): for all $ k\in [\kmax]^\star$, $\set_k=[n]^*$. We now present in \Cref{theo:dianaem} our key result, from which complexity expressions are derived. The proof is postponed to \Cref{app:proof-fedem}. 
\begin{theorem} \label{theo:dianaem}
   Assume \Cref{hyp:model} to \Cref{hyp:variance:oracle} and set $L^2
   \eqdef n^{-1} \sum_{i=1}^n L_i^2$, $\sigma^2 \eqdef{n}^{-1}
   \sum_{i=1}^n \sigma_i^2$.  Let $\{\hatS_k, k \in [\kmax] \}$ be
   given by \autoref{algo:dianaem-main}, with $\omega
     >0$, $\alpha \eqdef (1+\omega)^{-1}$ and $\pas_k = \pas
   \in\ocint{0, \gamma_{\max}}$ where
   \vspace{-0.3em}
         \begin{equation}\label{eq:maxLR}
         \gamma_{\max} \eqdef \frac{v_{\min}}{2L_{\dot \lyap}} \wedge
         \frac{\sqrt{n}}{2 \sqrt{2} L (1+\omega) \sqrt{\omega}} \eqsp.
         \end{equation}
            \vspace{-1em}
            
Denote by $K$ the uniform random variable on
           $[\kmax-1]$. Then, taking $V_{0,i} = \mf_i(\hatS_0)$ for
           all $i \in [n]^\star$:
   \begin{equation}\label{eq:theo:fedem}
 {v_{\min}} \left (1- \pas \frac{L_{\dot \lyap}}{v_{\min}} \right)
 \PE\ \left[ \|\mf(\hatS_K)\|^2 \right] \leq \frac{1}{\pas
 \kmax}\left( \lyap(\hatS_0) - \min \lyap \right) + \pas L_{\dot
   \lyap} \frac{1+5\omega}{n} \sigma^2 \eqsp.
  \end{equation}
\end{theorem}
When there is no compression ($\omega =0$ so that
  $\Q(s)=s$), we prove that the introduction of the random variables
  $V_{k,i}$'s play no role whatever $\alpha >0$ and the choice of the
  $V_{0,i}$'s, and we have for any $\pas \in \ooint{0,2 v_{\min} /
    L_{\dot \lyap}}$ (see \eqref{eq:fedem:nocompression} in the
  supplemental)
\begin{equation} \label{eq:dianaem:noomega}
\Big (1- \pas \frac{L_{\dot \lyap}}{2 v_{\min}} \Big) \PE\ \Big[
  \|\mf(\hatS_K)\|^2 \Big] \leq \frac{1}{\pas
  	\kmax}\left(
\lyap(\hatS_0) - \min \lyap \right) + \pas L_{\dot \lyap}
\frac{\sigma^2}{n} \eqsp.
\end{equation}
Optimizing the learning rate $\pas$, we derive the following corollary
(see the proof in \Cref{app:proof-fedem}).
\begin{corollary}[of \Cref{theo:dianaem}]\label{cor:diana-em}
 Choose $\pas \eqdef  \big( \frac{(\lyap(\hatS_0) - \min \lyap )n}{
   \kmax L_{\dot \lyap} (1+5\omega)\sigma^2 }\big )^{1/2} \wedge
 \pas_{\max}$. We get
\begin{align*}
 \PE\ \left[ \|\mf(\hatS_K)\|^2 \right] \leq \frac{4}{ v_{\min}}
 \bigg( \sqrt{ \frac{\big ( \lyap(\hatS_0) - \min \lyap \big)
     L_{\dot \lyap} (1+5\omega)\sigma^2 }{n \kmax} } \vee \frac{\big(
   \lyap(\hatS_0) - \min \lyap \big)}{\pas_{\max} \kmax} \bigg)
 \eqsp.
\end{align*}
\vspace{-0.5em}
\end{corollary}
\Cref{theo:dianaem} and
\Cref{cor:diana-em} do not require any assumption regarding the
distributional heterogeneity of workers. These results remain thus
valid when workers have access to data resulting from different
distributions — a widespread situation in FL frameworks. Crucially,
without assumptions on the heterogeneity of workers,
  the convergence of a ``naive" implementation of compressed
distributed EM  (\ie\ an implementation without the variables
$V_{k,i}$'s) would not converge.

Let us comment the complexity to reach an $\epsilon$-stationary point,
and more precisely how the complexity evaluated in terms of the number
of optimization steps depend on $\omega,n, \sigma^2$ and $\epsilon$.
Since $\mathcal{K}_{\operatorname{Opt}}(\epsilon) = \kmax$, from
\Cref{cor:diana-em} we have that: $ \mathcal{K}_\opt (\epsilon) =
O\Big( \frac{(1+\omega)\sigma^2}{\ n\epsilon^2 } \Big) \vee
O\Big(\frac{1}{ \gamma_{\max} \epsilon } \Big) \eqsp.  $
 
 \textbf{Maximal learning rate and compression.}  The comparison of \Cref{theo:dianaem} with the no
   compression case (see \eqref{eq:dianaem:noomega}) shows that compression impacts $\pas_{\max}$ by a
 factor proportional to $\sqrt{n}/\omega^{3/2} $ as $\omega$ increases
 (similar constraints were observed in the risk optimization
 literature, e.g. in
 \cite{horvath2019stochastic,philippenko2021preserved}). This
 highlights two different regimes depending on the ratio
 $\sqrt{n}/\omega^{3/2} $: if the number of workers $n$ scales at
 least as $\omega^{3}$, the maximal learning rate is not impacted by
 compression; on the other hand, for smaller numbers of workers $n\ll
 \omega^3$, compression can degrade the maximal learning rate. We
 highlight this conclusion with a small example in the case of scalar
 quantization for which $\omega \sim \sqrt{q}/
 s_{\mathrm{quant}}$: for $q = 10^2$ and $s_{\mathrm{quant}}=4$
 (obtaining a compression rate of a factor $16$), the maximal learning
 rate is almost unchanged if $n\geq 16$.

\textbf{Dependency on $\epsilon$. }
The complexity $\mathcal{K}_\opt(\epsilon)$ is decomposed into two
terms scaling respectively as $\sigma^2 \epsilon^{-2}$ and
$\gamma_{\max}^{-1}\epsilon^{-1}$, the first term being dominant when
$\epsilon \to 0$.  This  observation highlights two different
regimes: a \emph{high noise regime} corresponding to
$ \gamma_{\max} (1+\omega) \sigma^2 / (n \epsilon^{-1}) \ge 1 $ where
the complexity is of order $\sigma^2 \epsilon^{-2}$, and a \emph{low
  noise regime} where
$\gamma_{\max} (1+\omega) \sigma^2 / (n \epsilon^{-1}) \le 1 $ and the
complexity is of order $\gamma_{\max}^{-1} \epsilon^{-1}$. An extreme
example of the low noise case is $\sigma^2=0$, occurring
for example in the finite-sum case (i.e., when $\bars_i = m^{-1} \sum_{j=1}^m \bars_{ij}$) with the oracle
$\Smem_{k+1,i}= \bars_i \circ \map(\hatS_k)$.

\textbf{Impact of compression for $\epsilon$-stationarity.} As
mentioned above, the compression simultaneously impacts the maximal
learning rate (as in \eqref{eq:maxLR}) and the complexity $
\mathcal{K}_\opt (\epsilon) $.  Consequently, the impact of the
compression depends on the balance between $ \omega, n, \sigma^2 $ and
$ \epsilon $, and we can distinguish four different ``main''
regimes. In the following tabular, for each of the four situations, we
summarize the \textit{increase in complexity} $ \mathcal{K}_\opt
(\epsilon) $ resulting from compression.  \resizebox{\linewidth}{!}{
  \setlength{\aboverulesep}{0pt} \setlength{\belowrulesep}{0pt}
  \setlength{\extrarowheight}{.75ex}
	\begin{tabular}{c|ccc}
\toprule
&\begin{tabular}{c}
Complexity regime: \\ (Dominating term in $  \mathcal{K}_\opt (\epsilon)  $) 
\end{tabular} &  {$ 
	\frac{(1+\omega)\sigma^2}{\ n\epsilon^2 }  $} &  {$ 
	\frac{1}{ \gamma_{\max} \epsilon }   $} \\
\midrule
	\begin{tabular}{c}
$ \gamma_{\max} $ regime:  \\ (Dominating term in \eqref{eq:maxLR}) 
\end{tabular} & Example situation & \begin{tabular}{c}
High noise $ \sigma^2 $, \\small $ \epsilon $ \end{tabular}  &  \begin{tabular}{c}
Low $ \sigma^2 $ (e.g., large minibatch) \\  larger $ \epsilon $
\end{tabular}\\
  $ \frac{v_{\min}}{2L_{\dot \lyap}} $ &large ratio $ n / \omega^3 $ &    \cellcolor{orange}  $ \times \omega $  & \cellcolor{green!80!black} $ \times 1 $\\
   $ \frac{\sqrt{n}}{2 \sqrt{2} L (1+\omega) \sqrt{\omega}}  $ & low ratio $ n / \omega^3 $ &   \cellcolor{orange}  $ \times \omega $  & \cellcolor{yellow} $ \times \omega^{3/2}/\sqrt{n}  $ \\
   \bottomrule
\end{tabular}
}

Depending on the situation, the complexity can be multiplied by a factor ranging from 1 to $ \omega \vee (\omega^{3/2}/\sqrt{n})$ . Remark that the communication cost of each iteration is typically reduced by compression of a factor at least $ \omega $.  Moreover, the  benefit of compression is  most
significant in the \textit{low noise} regime and when the maximal
learning rate is $v_{\min}/(2 L_{\dot \lyap})$ (e.g., when $ n $ large enough).  We then improve the
communication cost of each iteration without increasing the
optimization complexity, effectively reducing the communication budget ``for free''.

Because of space constraints, the results in the PP regime are
postponed to \Cref{app:mainresultsPP}.

\section{\VRFEDEM: Federated EM algorithm with variance reduction}
A novel algorithm, called \VRFEDEM~and described
by~\autoref{algo:DIANASPIDEREM}, is derived to additionally
incorporate a variance reduction scheme in \FEDEM. It is described in
the finite-sum setting when for all $i \in [n]^\star$, $\bars_i \eqdef
m^{-1} \sum_{j=1}^m \bars_{ij}$: at each iteration $\# (t,k+1)$, the
oracle on $\bars_i \circ \map(\hatS_{t,k})$ will use a minibatch
$\batch_{t,k+1,i}$ of examples sampled at random (with or without
replacement) in $[m]^\star$.

\begin{wrapfigure}[34]{R}{0.60\textwidth}
	\flushright
	 \vspace{-2em}
\begin{minipage}{0.98\linewidth}
\SetInd{0.25em}{0.3em}
\begin{algorithm}[H]\DontPrintSemicolon
\caption{\VRFEDEM\ \label{algo:DIANASPIDEREM}} \KwData{
$\kouter, \kin, \lbatch \in \nset^\star$; for $i \in [n]^\star$,
$V_{1,0,i} \in \rset^q$; $\hatS_\init \in \rset^q$; a positive
sequence $\{\pas_{t,k+1}, t \in [\kouter]^\star, k \in
[\kin-1]\}$; $\alpha >0$}
\KwResult{sequence:
$\{\hatS_{t,k}, t \in [\kouter]^\star, k \in [\kin]\}$}{
$\hatS_{1,0} = \hatS_{1,-1} = \hatS_\init$, $V_{1,0} = n^{-1}
\sum_{i=1}^n V_{1,0,i}$ \label{algo:DS:init:V0} \; \For{$i=1,
\ldots, n$}{$\Smem_{1,0,i} = \frac{1}{m} \sum_{j=1}^m
\bars_{ij} \circ \map(\hatS_\init)$ } \For{$t=1, \ldots,
\kouter$}{ \For{$k=0, \ldots, \kin-1$}{ \For{$i=1, \ldots,
n$ {\em (worker $\# i$, locally)}}{  Sample at random a
batch $\batch_{t,k+1,i}$ of size $\lbatch$ in $[m]^\star$
\; Set $\Smem_{t,k+1,i} = \Smem_{t,k,i} + \lbatch^{-1}
\sum_{j \in \batch_{t,k+1,i}} \left( \bars_{ij} \circ
\map(\hatS_{t,k}) - \bars_{ij} \circ \map(\hatS_{t,k-1})
\right)$ \label{line:DS:localS} \; Set $\Delta_{t,k+1,i} =
\Smem_{t,k+1,i} - \hatS_{t,k} - V_{t,k,i}$  \; Set
$V_{t,k+1,i} = V_{t,k,i} + \alpha\, \Q(\Delta_{t,k+1,i})$.
\label{line:DS:Vi}
    \; Send $\Q(\Delta_{t,k+1,i})$ to the central server \;}
  {\em (the central server)} \label{line:DS:sendcontroller}\; 
  Set $H_{t,k+1} =  V_{t,k} + n^{-1}   \sum_{i=1}^n \Q(\Delta_{t,k+1,i})$ \; 
  Set   $\hatS_{t,k+1} =  \hatS_{t,k} +   \pas_{t,k+1} H_{t,k+1}$  \label{line:DS:SA} \;
   Set $V_{t,k+1} = V_{t,k} + \alpha n^{-1} \sum_{i=1}^n \Q(\Delta_{t,k+1,i})$ \label{line:DS:V} \;
 Send $\hatS_{t,k+1}$ and $\map (\hatS_{t,k+1})$ to the $n$ workers \;}
$\hatS_{t+1,0} = \hatS_{t+1,-1} = \hatS_{t,\kin}$ \;
$V_{t+1,0} = V_{t,\kin}$ \label{algo:DS:init:V1} \;
\For{$i=1, \ldots, n$}
{$\Smem_{t+1,0,i}= \frac{1}{m}\sum_{j=1}^m\bars_{ij} \circ\map(\hatS_{t+1,0})$\label{line:DS:initS} \;
  $V_{t+1,0,i} = V_{t,\kin,i}$ \label{algo:DS:init:V2}}
} }\end{algorithm}
\end{minipage}
\end{wrapfigure}
The algorithm is decomposed into $\kouter$ outer loops (indexed by
$t$), each of them having $\kin$ inner loops (indexed by $k$). At
iteration $\# (k+1)$ of the inner loops, each worker $\# i$ updates a
local statistic $\Smem_{t,k+1,i}$ based on a minibatch
$\batch_{t,k+1,i}$ of its own examples $\{ \bars_{ij}, j \in
\batch_{t,k+1,i} \}$ (see Line~\ref{line:DS:localS}): starting from
$\hatS_{t,0,i} \eqdef m^{-1} \sum_{j=1}^m \bars_{ij} \circ
\map(\hatS_{t,-1})$, $\hatS_{t,k+1,i}$ is defined in such a way that
it approximates $m^{-1} \sum_{j=1}^m \bars_{ij} \circ
\map(\hatS_{t,k})$ (see \Cref{coro:DS:biasS}). Then, the worker $\# i$
sends to the central server a quantization of $\Delta_{t,k+1,i}$ (see
Line~\ref{line:DS:sendcontroller}) which can be seen as an
approximation of $\alpha^{-1}\{ \mf_i(\hatS_{t,k}) -
\mf_i(\hatS_{t,k-1})\}$ upon noting that the variable $V_{t,k+1,i}$
defined by Line~\ref{line:DS:Vi} approximates $\mf_i(\hatS_{t,k})$
(see \Cref{cor:control_var}). The central server learns the mean value
$V_{t,k+1} = n^{-1} \sum_{i=1}^n V_{t,k+1,i}$ (see
Line~\ref{line:DS:V} and \Cref{lem:DS:V}) and, by adding the quantized
quantities, defines a field $H_{t,k+1}$ which approximates $n^{-1}
\sum_{i=1}^n \mf_i(\hatS_{t,k})$ (see
\Cref{prop:DS:Hvariance}). Line~\ref{line:DS:SA} can be seen as a
Stochastic Approximation update, with learning rate $\pas_{t,k+1}$ and
mean field $s \mapsto n^{-1} \sum_{i=1}^n \mf_i(s)$ (see
\eqref{eq:hfield} for the definition of $\mf_i$).

The variance reduction is encoded in the definition of
$\Smem_{t,k+1,i}$, Line~\ref{line:DS:localS}. We have
$\Smem_{t,k+1,i}= \lbatch^{-1} \sum_{j \in \batch_{t,k+1,i}}
\bars_{ij} \circ \map(\hatS_{t,k}) + \Upsilon_{t,k+1,i}$. The first
term is the natural approximation of $\bars_i \circ \map(\hatS_{t,k})$
based on a minibatch $\batch_{t,k+1,i}$. Conditionally to the past,
$\Upsilon_{t,k+1,i}$ is correlated to the first term and biased, but
its bias is canceled at the beginning of each outer loop (see
Line~\ref{line:DS:initS} and \Cref{sec:DS:controlvariate}):
$\Upsilon_{t,k+1,i}$ defines a {\em control variate}. Such a variance
reduction technique was first proposed in the stochastic gradient
setting \cite{nguyen:liu:etal:2017,fang:etal:2018,wang:etal:nips:2019}
and then extended to the EM setting
\cite{SPIDER-EM,fortetal:2021:icassp}.  At the end of each outer loop,
the local approximations $\Smem_{t+1,0,i}$ are initialized to the full
sum $m^{-1} \sum_{j=1}^m \bars_{ij} \circ \map(\hatS_{t,\kin})$ (see
Line~\ref{line:DS:initS}) thus canceling the bias of $\Smem_{\cdot,i}$
(see \Cref{prop:DS:biasS}).

When there is a single worker and no compression is used ($n=1$,
$\omega=0$), \VRFEDEM~reduces to \texttt{SPIDER-EM}, which has been
shown to be rate optimal for smooth, non-convex finite-sum
optimization~\cite{SPIDER-EM}. \Cref{theo:DS} studies the FL setting
($n \geq 1$ and $\omega \geq 0$): it establishes a finite time control
of convergence in expectation for \VRFEDEM\ . Assumptions
\Cref{hyp:lipschitz} and \Cref{hyp:variance:oracle} are replaced with
\Cref{hyp:DS:lipschitz}.
\begin{assumption} \label{hyp:DS:lipschitz} For any $i \in [n]^\star$ and $j \in [m]^\star$, the conditional expectations
  $\bars_{ij}(\param)$ are well defined for any $\param \in \Theta$,
  and there exists $L_{ij}$ such that for any $s,s' \in \rset^q$, $ \|
  (\bars_{ij} \circ \map(s) -s) - (\bars_{ij} \circ \map(s') -s')\|
  \leq L_{ij} \|s-s'\| \eqsp.  $
  \end{assumption} 
\begin{theorem}\label{theo:DS}
  Assume \Cref{hyp:model,hyp:bars,hyp:Tmap}, \Cref{hyp:DS:lyap},
  \Cref{hyp:var:quantif} and \Cref{hyp:DS:lipschitz}. Set
  $L^2 \eqdef n^{-1} m^{-1} \sum_{i=1}^n \sum_{j=1}^m L^2_{ij}$. Let
  $\{\hatS_{t,k}, t \in [\kouter]^\star, k \in [\kin-1] \}$ be given
  by~\autoref{algo:DIANASPIDEREM} run with
  $\alpha \eqdef 1/(1+\omega)$, $V_{1,0,i} \eqdef \mf_i(\hatS_{1,0})$
  for any $i \in [n]^\star$,
  $ \lbatch \eqdef \lceil\frac{\kin}{(1+\omega)^2} \rceil$ and
\begin{align}
\pas_{t,k} = \pas \eqdef \frac{v_{\min}}{L_{\dot \lyap}}  \bigg (1 +  4 \sqrt{2} \frac{v_{\max}}{L_{\dot \lyap}}  \frac{L}{\sqrt{n}} (1+\omega) \Big( \omega + \frac{1+10 \omega}{8}\Big)^{1/2} \bigg )^{-1}. \label{eq:pasmax_DS}
\end{align} 
Let $(\tau, K)$ be the uniform random variable on
$[\kouter]^\star \times [\kin-1]$, independent of
$\{\hatS_{t,k}, t \in [\kouter]^\star, k \in [\kin]\}$. Then, it
holds
  \begin{align}
\PE\left[\|H_{\tau,K+1}\|^2\right] & \leq    \frac{2 \big( \PE\big[ \lyap(\hatS_{1,0})\big] -
	\min \lyap\big) }{{v_{\min}}\pas \kin \kouter} \eqsp,  \label{eq:thm-VR-conv}\\ 
  \PE\left[\|\mf(\hatS_{\tau,K})\|^2 \right] & \leq 2 \Big(1+ \pas^2
  \frac{L^2 (1+\omega)^2}{n}\Big) \PE\left[\|H_{\tau,K+1}\|^2\right] \label{eq:thm-VR-control}
  \eqsp. 
    \end{align}
  \end{theorem} The proof is postponed to~\Cref{sec:proof:DS}.  This
  result is a consequence of the more general~\Cref{prop:DS}.  We make
  the following comments:
\begin{enumerate}[topsep=0pt,itemsep=1pt,leftmargin=*,noitemsep,wide]
\item Eq.~\eqref{eq:thm-VR-conv} provides the convergence of $
  \PE\left[\|H_{\tau,K+1}\|^2\right ] $, and
  Eq.~\eqref{eq:thm-VR-control} ensures that the quantity of interest
  $ \PE[\|\mf(\hatS_{\tau,K})\|^2 ] $ is controlled by $
  \PE[\|H_{\tau,K+1}\|^2 ] $. We observe that $ 2 (1+ \pas^2 \frac{L^2
    (1+\omega)^2}{n}) $ is uniformly bounded w.r.t. $ \omega $ as, by
  \eqref{eq:pasmax_DS}, $ \gamma^2 = O_{\omega\to \infty}(\omega^{-3})
  $.
    \item Up to our knowledge, this is the first result on Federated
      EM, that leverages advanced variance reduction techniques, while
      being robust to distribution heterogeneity (the theorem is valid
      without any assumption on heterogeneity) and while reducing the
      communication cost.
    \item Without compression ($\omega=0$) and in the single-worker
      case ($n=1$), \citet{SPIDER-EM} use $\kin = \lbatch$: we recover
      this result as a particular case. When $n>1$ and $\omega>0$, the
      recommended batch size $\lbatch$ decreases as $1/(1+\omega)^2$.
\end{enumerate}

\textbf{Convergence rate and optimization complexity.} Our step-size $
\pas $ is chosen constant and \textit{independent} of $ \kin, \kouter
$. Indeed, contrary to \Cref{theo:dianaem}, there is no Bias-Variance
trade-off (as typically observed with variance reduced methods), and
the optimal choice of $ \pas $ is the largest one to ensure
convergence. Consequently, since the number of optimization steps is
$\kouter \kin$, we have $ \mathcal{K}_\opt (\epsilon) =
O(\frac{1}{\gamma \epsilon })$.

\textbf{Impact of compression on the learning rate and $ \epsilon $-stationarity.} The compression constant  $ \omega $ does not directly appear in \eqref{eq:thm-VR-conv}, but impacts the value of $ \pas $. Two different regimes appear: 
\begin{enumerate}[topsep=0pt,itemsep=1pt,leftmargin=*,noitemsep,wide]
\item if $ 4 \sqrt{2} \frac{v_{\max}}{L_{\dot \lyap}}
  \frac{L}{\sqrt{n}} (1+\omega) \left( \omega + \frac{1+10
    \omega}{8}\right)^{1/2} \ll 1 $ (i.e. we focus on the large $
  \omega,n $ asymptotics when $ \omega^{3} \ll n$), then $ \pas
  \simeq \frac{v_{\min}}{L_{\dot \lyap}}$ has nearly the same value as
  without compression \cite{SPIDER-EM}. The complexity is then similar
  to the one of \texttt{SPIDER-EM}~\cite{SPIDER-EM}, with a smaller
  communication cost. The gain from compression is maximal in this
  regime.
\item if $ 4 \sqrt{2} \frac{v_{\max}}{L_{\dot \lyap}}
  \frac{L}{\sqrt{n}} (1+\omega) \left( \omega + \frac{1+10
    \omega}{8}\right)^{1/2} \gg 1 $ (i.e. we focus on the large $
  \omega,n $ asymptotics when $ \omega^{3} \gg n$), then $ \pas =
  O\left ( \frac{v_{\min} \sqrt{n}}{v_{\max }L \omega^{3/2}}\right )$
  is strictly smaller than without compression. The optimization
  complexity is then higher to the one of \texttt{SPIDER-EM}\footnote{As a corollary
    of \cite[Theorem 2]{SPIDER-EM}, the optimization complexity of
    \texttt{SPIDER-EM} is $\kouter + \kin \kouter$ that is
    $\epsilon^{-1}$ in order to reach $\epsilon$-stationarity.} (by a
  factor proportional to $ \omega^{3/2}/\sqrt{n}$) with a smaller
  communication cost (typically at least $\omega$ times less bits
  exchanged per iteration). The overall trade-off thus depends on the
  comparison between $ \omega $ and $ n $.
\end{enumerate}
\begin{wrapfigure}[4]{R}{0.450\textwidth}
	\flushright
	\vspace{-2.5em}
	\begin{minipage}{0.98\linewidth}
		\setlength{\aboverulesep}{0pt}
		\setlength{\belowrulesep}{0pt}
		\setlength{\extrarowheight}{.75ex}
		\resizebox{\linewidth}{!}{	\begin{tabular}{l|cc}
				\toprule
				&		Complexity : &  {$ 
					{1} / {( \gamma \epsilon )}   $} \\
				\midrule
				\begin{tabular}{c}
					$ \gamma $ regime:   (Dominating \\ term in \eqref{eq:pasmax_DS}) 
				\end{tabular} & Example situation &  \\
				$ v_{\min}/L_{\dot \lyap}$ &large ratio $ n / \omega^3 $  &\cellcolor{green!80!black} $ \times 1 $\\
				$ {v_{\min} \sqrt{n}} / ({v_{\max }L \omega^{3/2}}) $ & low ratio $ n / \omega^3 $ & \cellcolor{yellow} $ \times \omega^{3/2}/\sqrt{n}  $ \\
				\bottomrule
		\end{tabular}}
	\end{minipage}
\end{wrapfigure}

We summarize these two regimes in this  tabular, focusing on
the large $ n $, large $ \omega $ asymptotic regimes. For the two
regimes, we indicate the \textit{increase in complexity}
$ \mathcal{K}_\opt (\epsilon) $ resulting from compression.

We provide a discussion on  \textit{computed conditional expectations} complexity $ \mathcal K_\CE $ in \Cref{app:KCE}.

\section{Numerical illustrations}
\label{sec:numerical}
In this section, we illustrate the performance of \FEDEM~and \VRFEDEM~
applied to inference in Gaussian Mixture Models (GMM), on a synthetic
data set and on the MNIST data set. We also present an application to
Federated missing data imputation, in the context of citizen science
data analysis for biodiversity monitoring with the analysis of a
subsample of the eBird data set \cite{SULLIVAN20092282,eBird}.

\textbf{Synthetic data. }  The synthetic data are from the following
GMM model: for all $\ell \in[N]^{\star}$ and $g\in\{0,1\}$,
$\PP(Z_\ell = g) = \pi_g$; and conditionally to $Z_\ell = g$, $Y_\ell
\sim \mathcal{N}_2(\mu_g,\Sigma)$. The $2 \times 2$ covariance matrix
$\Sigma$ is known, and the parameters to be fitted are the weights
$(\pi_0,\pi_1)$ and the expectations $(\mu_0,\mu_1)$.  The total
number of examples is $N=10^4$, the number of agents is $n=10^2$, and
the probability of participation of servers is $p=0.75$. \FEDEM\ and
\VRFEDEM\ are run with $\gamma = 10^{-2}$, $\omega=1$ and $\alpha =
10^{-2}$.  For \FEDEM, we consider the finite-sum setting when
$\bars_i = m^{-1} \sum_{j=1}^m \bars_{ij}$ with $m=10^2$; the oracle
$\Smem_{k+1,i}$ is obtained by a sum over a minibatch of $\lbatch =
20$ examples. For \VRFEDEM, we set $\lbatch =5$ and $\kin= 20$.  We
run the two algorithms for $500$ epochs (one epoch corresponds to $N$
conditional expectation evaluations $\bars_{ij}$).
Figure~\ref{fig:gmm_synthet} shows a trajectory of $\|H_k\|^2$ given
by \FEDEM~(and $\| H_{t,k}\|^2$ given by \VRFEDEM), along with the
theoretical value of the mean field $\| \mf(\hatS_k)\|^2$ for $\FEDEM$
(and $\| \mf(\hatS_{t,k})\|^2$ for \VRFEDEM). The results illustrate
the variance reduction, and gives insight on the variability of the
trajectories resulting from the two algorithms.

\textbf{MNIST Data set. }
We perform a similar experiment on the MNIST dataset to illustrate the
behaviour of \FEDEM~and \VRFEDEM~on a GMM inference problem with real
data.  The dataset consists of $N = 7 \times 10^4$ images of
handwritten digits, each with $784$ pixels. We pre-process the dataset
by removing $67$ uninformative pixels (which are always zero across
all images) to obtain $d = 717$ pixels per image.  Second, we apply
principal component analysis to reduce the data dimension. We keep the
$d_{\operatorname{PC}} = 20$ principal components of each
observation. These $N$ preprocessed observations are distributed at
random across $n =10^2$ servers, each containing $m = 700$
observations. We estimate a
$\rset^{d_{\operatorname{PC}}}$-multivariate GMM model with $G=10$
components. Details on the multivariate Gaussian mixture model are
given in the supplementary material (see \autoref{app:numerical}).
Here again, $\bars_i$ is a sum over the $m$ examples available at
server $\# i$; the minibatches are independent and sampled at random
in $[m]^\star$ with replacement; we choose $\lbatch= 20$ and the
step size is constant and set to $\gamma= 10^{-3}$. The same initial
value $\hatS_\init$ is used for all experiments: we set $\hatS_\init
\eqdef \bar{s}(\pi^0, \mu^0,\widehat{\Sigma}^0)$, where $\pi_g^0 =
1/G$ for all $g\in [G]^\star$, the expectations $\mu_g^0$ are sampled
uniformly at random among the available examples, and
$\widehat{\Sigma}^0$ is the empirical covariance matrix of the $N$
examples. \Cref{fig:estimation:weight} shows the sequence of
parameter estimates for the weights and the squared norm of the mean
field $ \| H_k \|^2$ for \FEDEM\ (resp. $\|H_{t,k}\|^2$ for
\VRFEDEM\ ) vs the number of epochs.

 \begin{figure}
	\begin{minipage}{.46\textwidth}
		\centering
		\includegraphics[width=0.49\textwidth]{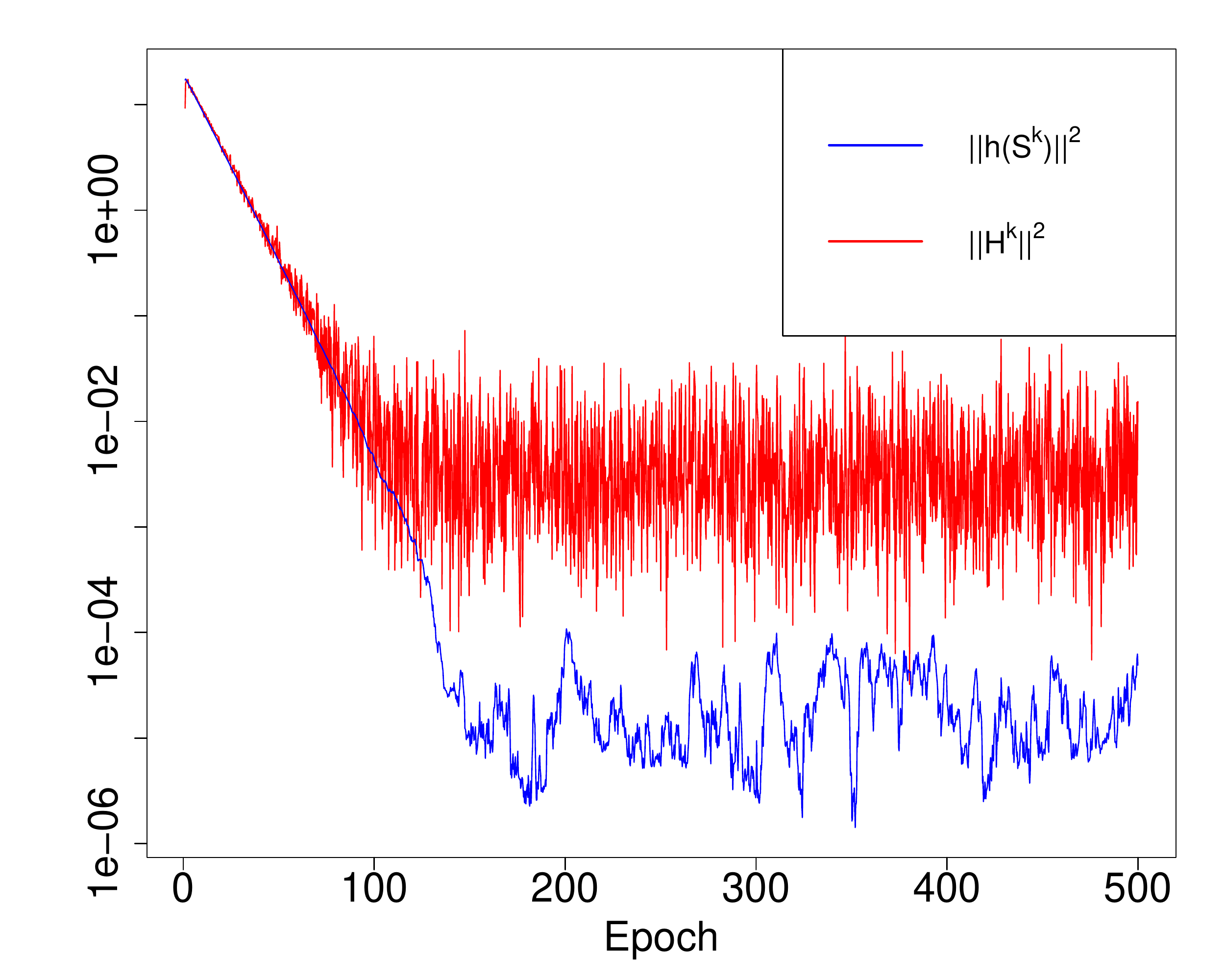} 
		\includegraphics[width=0.49\textwidth]{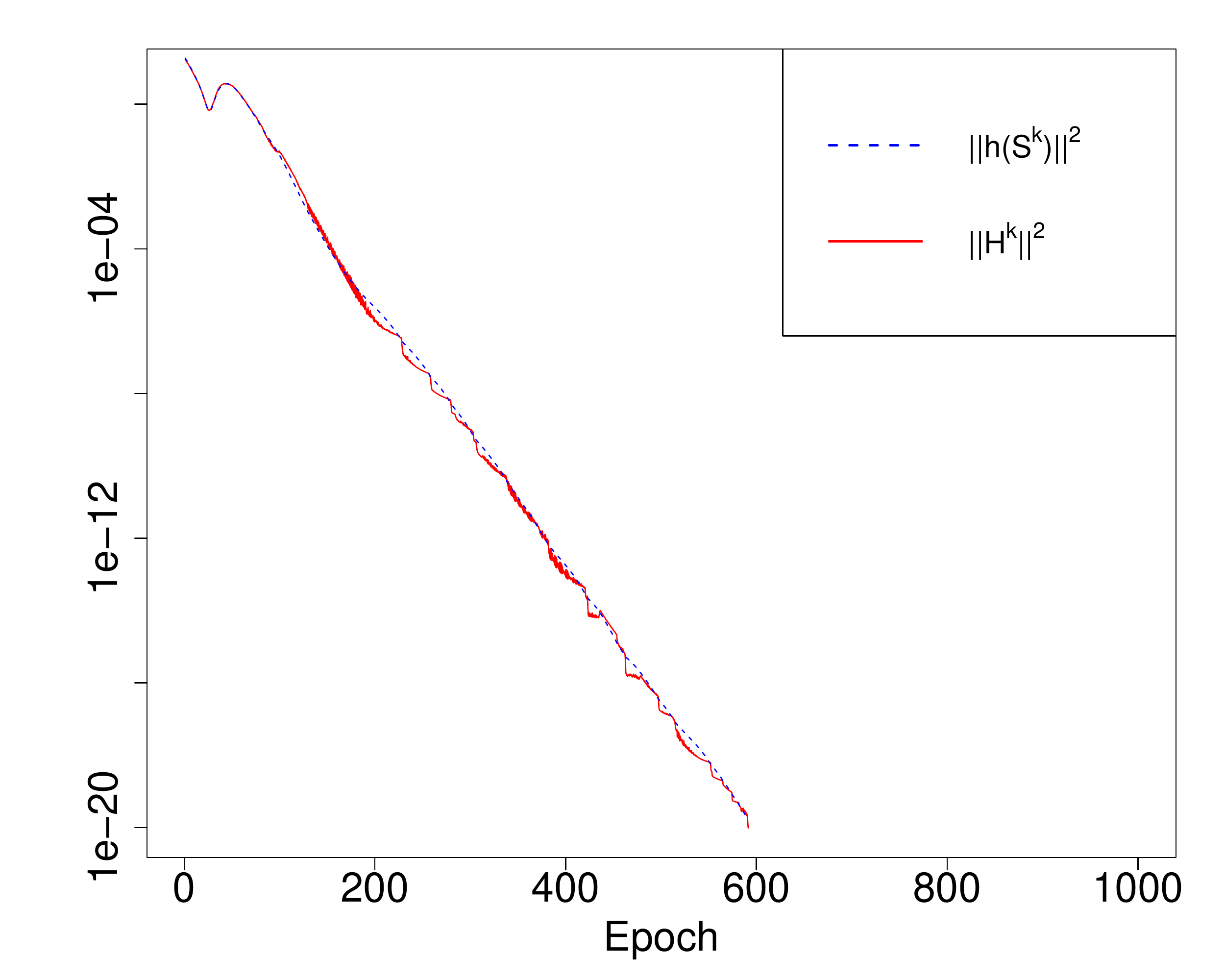}
		\caption{Trajectory of \FEDEM~vs the number of epochs (left; blue line: $\|\mf(\hatS^k)\|^2$; red line: $\|H_{k}\|^2$) and of \VRFEDEM~(right; dashed blue line: $\|\mf(\hatS^k)\|^2$; solid red line: $\|H_{t,k}\|^2$).}
		\label{fig:gmm_synthet}\vspace{-.2cm}
		\label{fig:test1}
	\end{minipage}%
	\hspace{0.12cm}
	\begin{minipage}{.53\textwidth}
		\centering
		\includegraphics[width=0.49\textwidth]{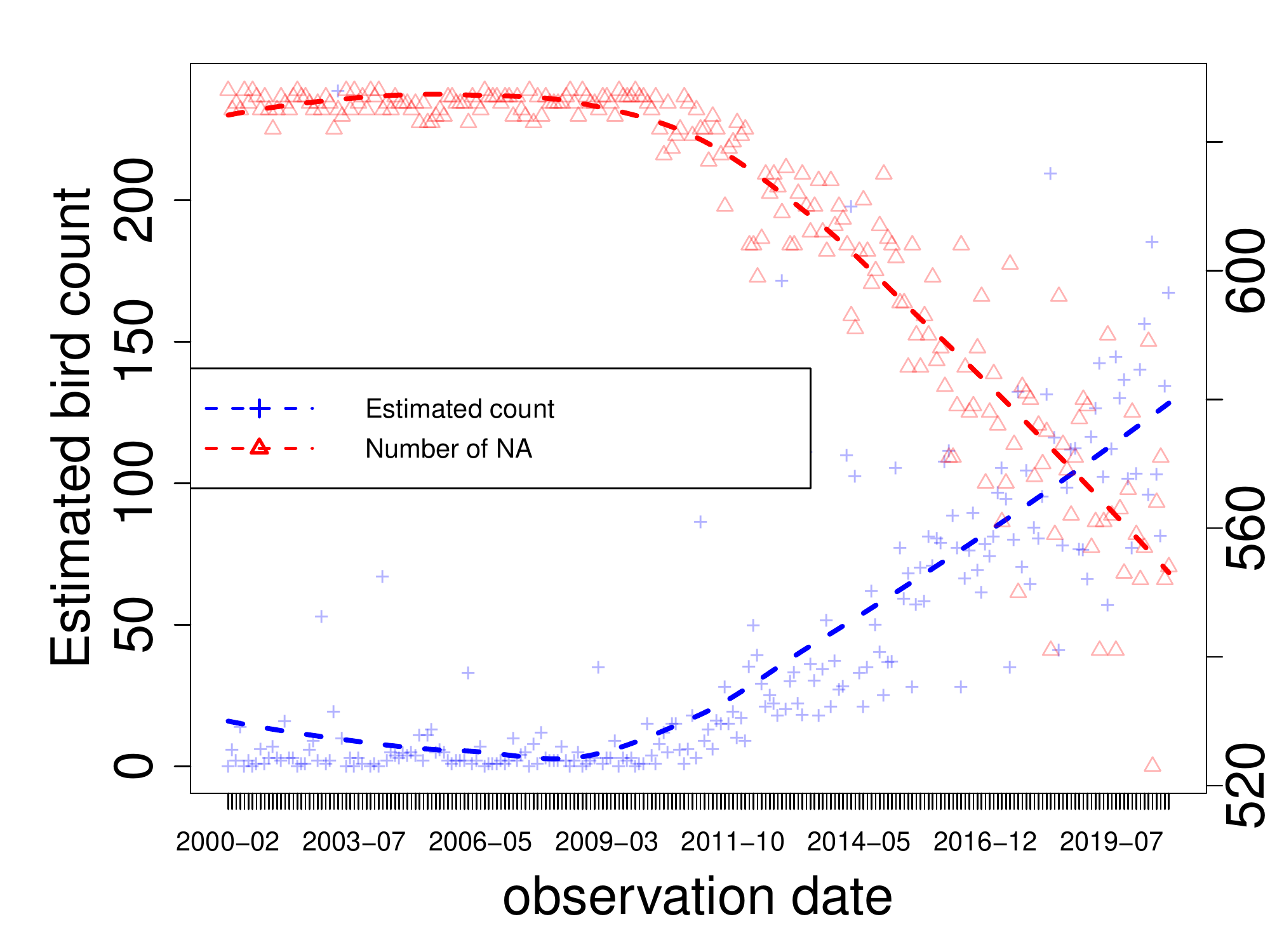}
		\includegraphics[width=0.49\textwidth]{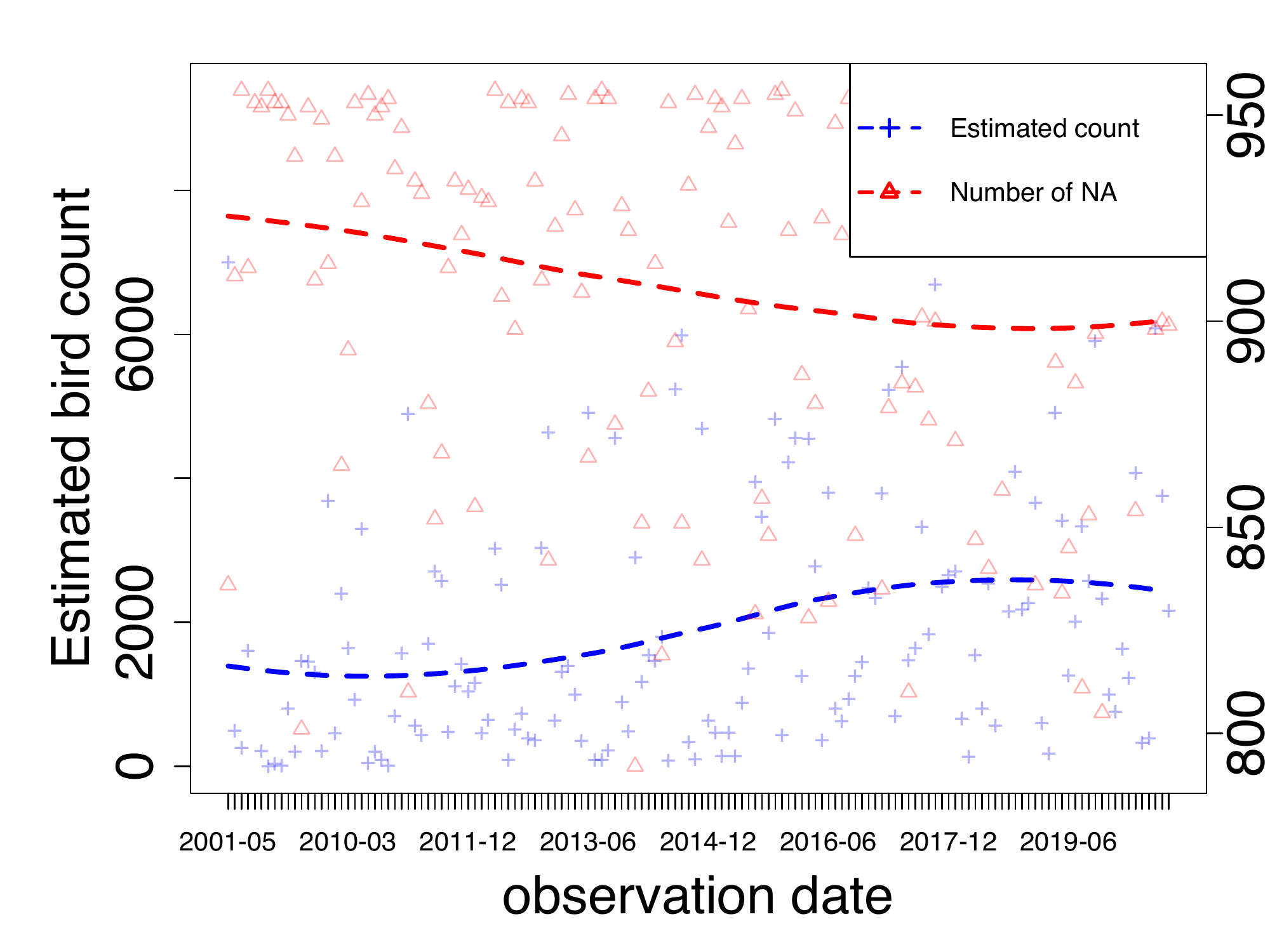}
		\caption{Estimated temporal trends for Common Buzzard (Left) and Mallard (right). Blue crosses: estimated monthly counts; Red triangles: number of missing values. Dotted lines: LOESS regressions for the estimated counts (blue) and the number of missing values (red).}
		\label{fig:ebird}\vspace{-0cm}
	\end{minipage}\\%
	\begin{minipage}{0.99\textwidth}
	\centering
	\includegraphics[width=0.22\textwidth]{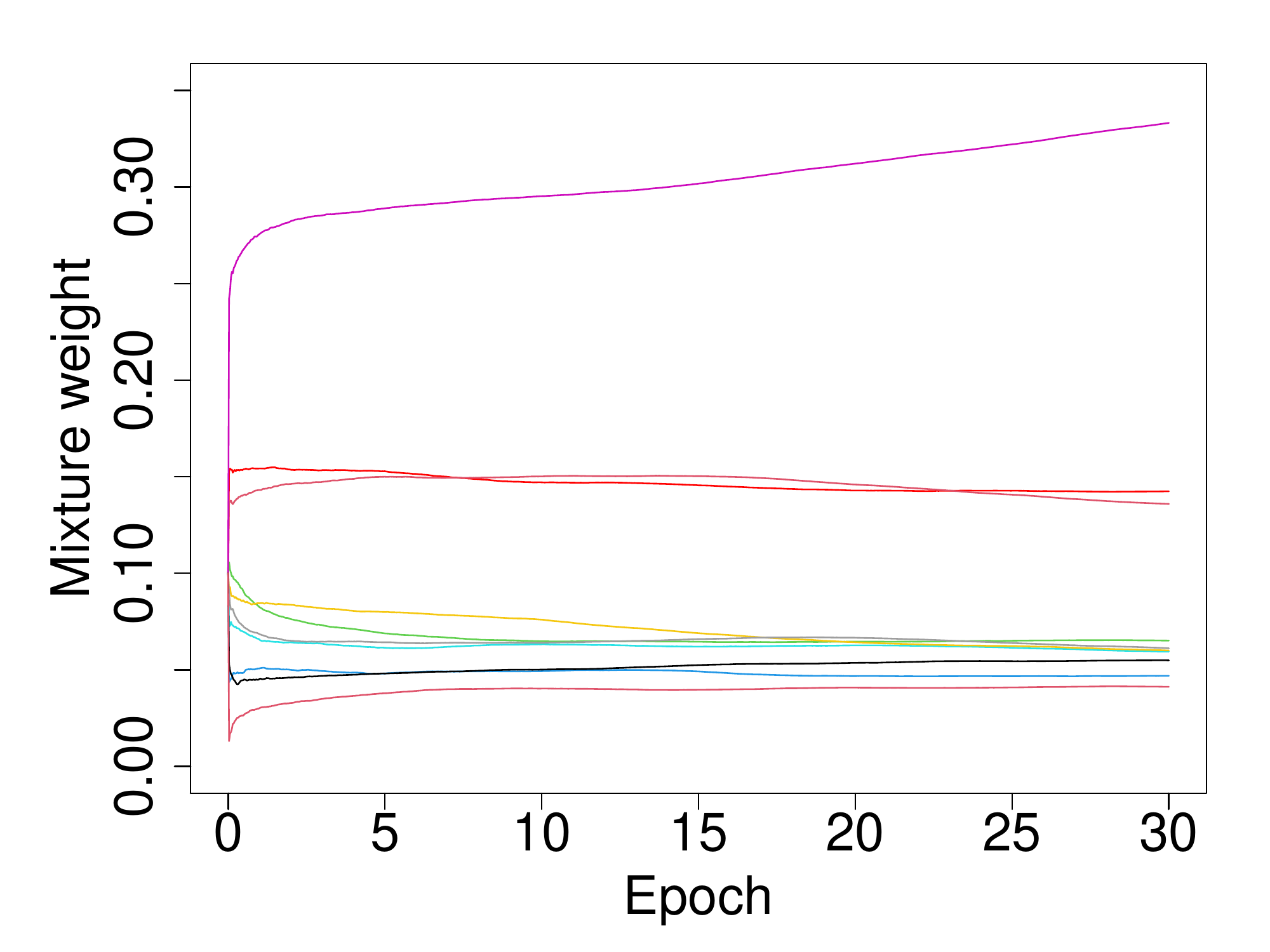}
	\hfill
	\includegraphics[width=0.22\textwidth]{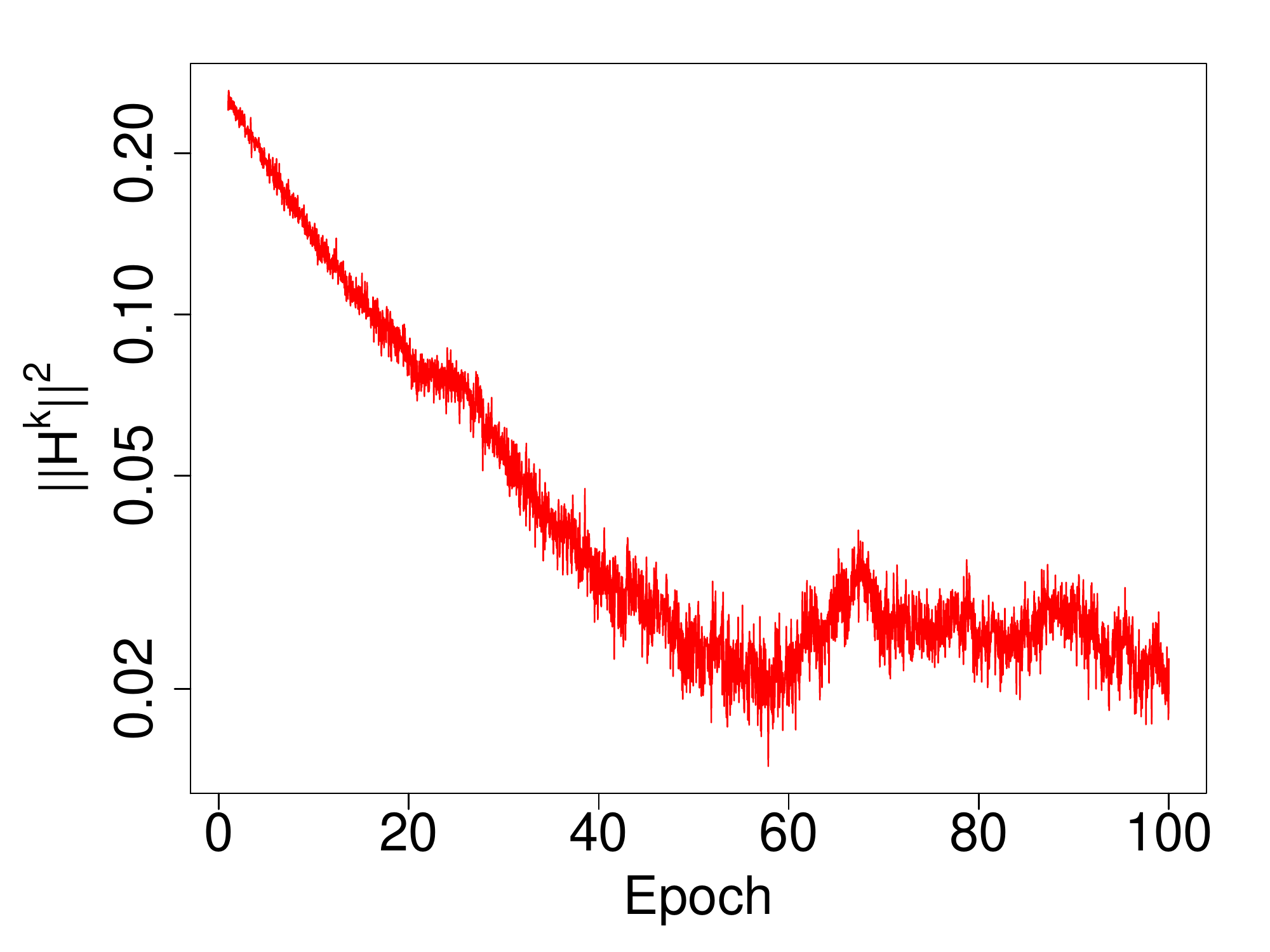}
        \hfill
	\includegraphics[width=0.22\textwidth]{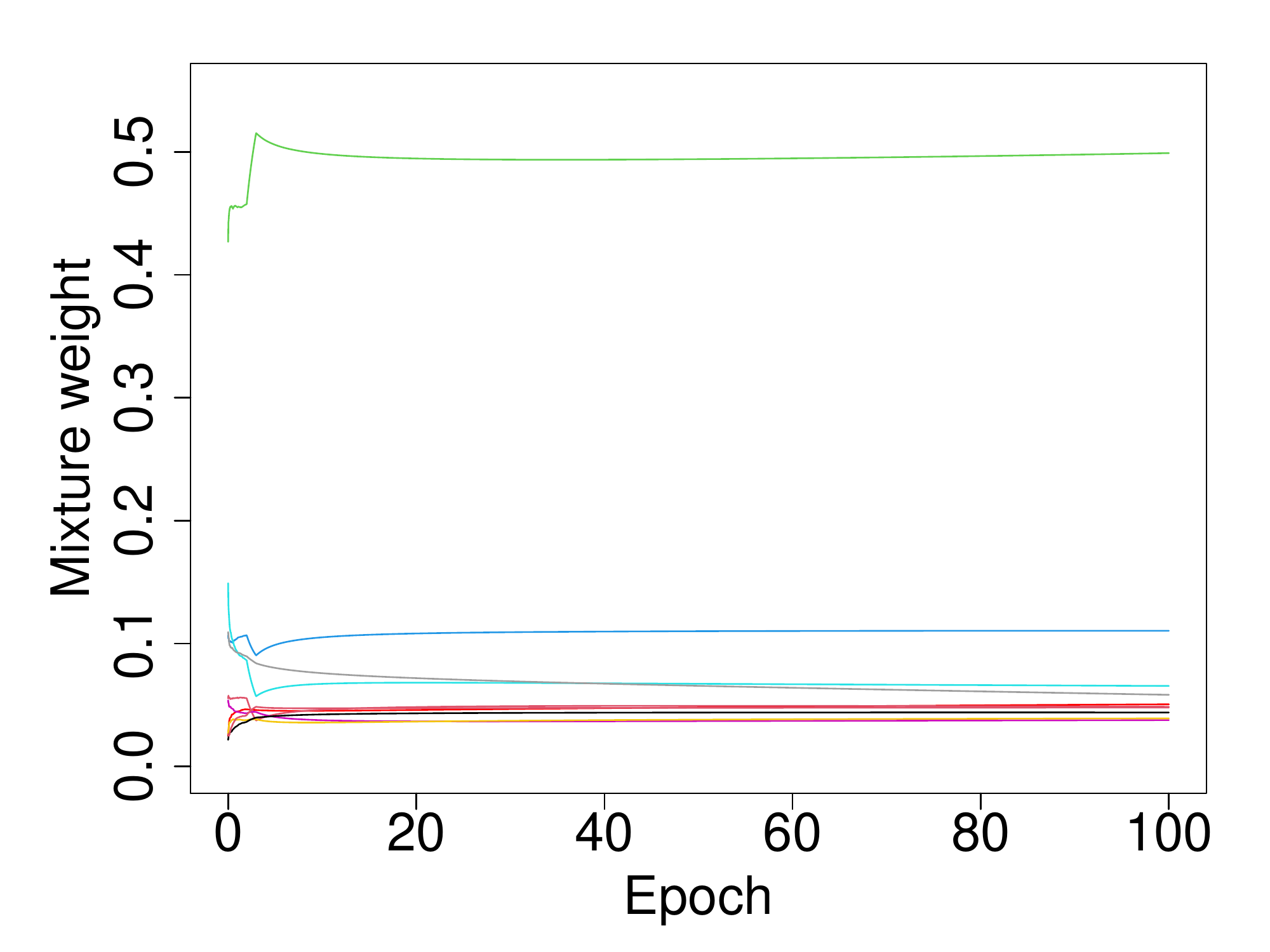}
	\hfill
	\includegraphics[width=0.22\textwidth]{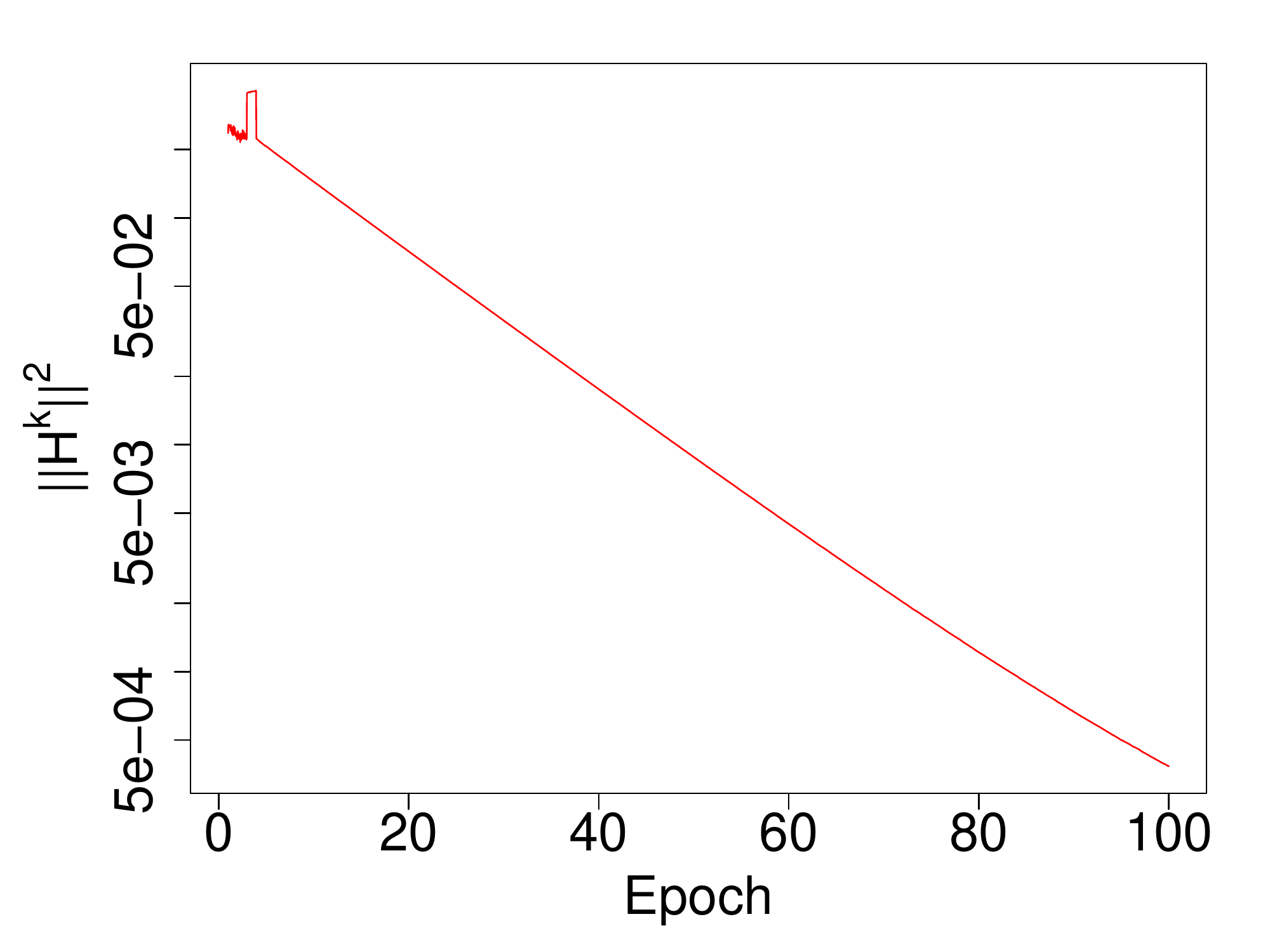}
	\caption{[Left to right] For \FEDEM~: Evolution of the
          estimates of the weights $\pi_\ell$ for $\ell \in [G]^\star$
          vs the number of epochs (first plot) and Evolution of the
          squared norm of the mean field $\|H_k\|^2$ vs the number of
          epochs (second plot). Then, the same things for
          \VRFEDEM\ (third and fourth plots).  \vspace{-0.4cm}}
	\label{fig:estimation:weight}\vspace{-.2cm}
	\label{fig:test2}
\end{minipage}%
\hspace{0.12cm}
\end{figure}

\textbf{Federated missing values imputation for citizen science.} 
We develop \texttt{FedMissEM}, a special instance of \FEDEM\ designed
to missing values imputation in the federated setting; we apply it to
the analysis of part of the {\em eBird} data base
\cite{SULLIVAN20092282,eBird}, a citizen science smartphone
application for biodiversity monitoring. In {\em eBird}, citizens
record wildlife observations, specifying the ecological site they
visited, the date, the species and the number of observed
specimens. Two major challenges occur: (i) ecological sites are
visited irregularly, which leads to missing values and (ii)
non-professional observers have heterogeneous wildlife counting
schemes. 

{\em $\bullet$ Model and the \texttt{FedMissEM} algorithm.}  $I$
observers participate in
the programme, there are $J$ ecological sites and $L$ time
stamps. Each observer $\# i$ provides a $J \times L$ matrix $X^i$ and
a subset of indices $\Omega^i \subseteq [J]^\star \times
[L]^\star$. For $j \in [J]^{\star}$ and $\ell \in [L]^{\star}$, the
variable $X_{j\ell}^i$ encodes the observation that would be collected
by observer $\# i$ if the site $\# j$ were visited at time stamp $\#
\ell$; since there are unvisited sites, we denote by $Y^i \eqdef
\{X^i_{j\ell}, (j,\ell)\in\Omega^i\}$ the set of observed values and
$Z^i \eqdef \{X^i_{j\ell}, (j,\ell)\notin\Omega^i\}$ the set of
unobserved values.  The statistical model is parameterized by a matrix
$\theta\in\rset^{J \times L}$, where $\theta_{j\ell}$ is a scalar
parameter characterizing the distribution of species individuals at
site $j$ and time stamp $\ell$. For instance, $\theta_{j\ell}$ is the
log-intensity of a Poisson distribution when the observations are
count data or the log-odd of a binomial model when the observations
are presence-absence data. This model could be extended to the case
observers $\# i$ and $\# i'$ count different number of specimens on
average at the same location and time stamp, because they
do not have access to the same material or do not have the same level
of expertise: heterogeneity between observers could be modeled by
using different parameters for each individual $\# i$ say
$\theta^{i}\in\rset^{J \times L}$.  Here, we consider the case when $
\theta_{j\ell}^{i} = \theta_{j\ell}$ for all $(j,\ell)\in
      [J]^\star\times[L]^\star$ and $i \in [I]^\star$.
We further assume that the entries $\{ X_{j\ell}^i, i \in [I]^\star,
  j \in [J]^\star, \ell \in [L]^\star\}$ are independent with a
  distribution from an exponential family with respect to some
  reference measure $\nu$ on $\rset$ of the form: $x
  \mapsto \rho(x) \exp\{ x \theta_{j\ell} - \psi(\theta_{j\ell})\}.$
  Algorithm~\ref{algo:fedImput} in Appendix~\ref{app:fedmiss} provides details on the model, and
the pseudo-code for \texttt{FedMissEM}.

$\bullet$ {\em Application to eBird data analysis.}  We apply
\texttt{FedMissEM} to the analysis of part of the {\em eBird} data
base \cite{SULLIVAN20092282,eBird} of field observations reported in
France by $I = 2,465$ observers, across $J = 9,721$ sites and at $L =
525$ monthly time points. We analyze successively two data sets
corresponding to observations of two relatively common species: the
Common Buzzard and the Mallard. These subsamples correspond
respectively to $N=5,980$ and $N=12,185$ field observations. The $I$
field observers are randomly assigned into $n = 10$ groups (the
observations of the field observers from the group $c \in [n]^\star$
are allocated to the server $\# c$).  For $c\in[n]^{\star}$, server
$c$ contains $N_c$ observations; in our two examples, $N_c$ ranges
between $400$ and $1,500$. We run \texttt{FedMissEM} for $150$ epochs;
with $\gamma = 10^{-4}$, $\alpha = 10^{-3}$, $\lbatch = 10^2$, a rank
$r = 2$ and $\lambda = 0$; for the distribution of the variables
$X^i_{j\ell}$, we use a Gaussian distribution with unknown expectation
$\theta_{j\ell}$ and variance $1$.  We recover aggregated temporal
trends at the national French level for these two bird species by
summing the estimated counts across ecological sites, for each time
stamp; the trends are displayed in Figure~\ref{fig:ebird}, along with
a locally estimated scatterplot smoothing (LOESS).

\section{Conclusions}
We introduced \FEDEM\ which is, to the best of our knowledge, the first
algorithm implementing EM in a FL setting, and handles compression of
exchanged information, data heterogeneity and partial
participation. We further extended it to incorporate a variance
reduction scheme, yielding \VRFEDEM. We derived complexity
bounds which highlight the efficiency of the two algorithms, and
illustrated our claims with numerical simulations, as well as an
application to biodiversity monitoring data. In a simultaneously published work,
\citet{marfoq2021federated} consider a different Federated EM
algorithm, in order to address the personalization challenge by
considering a mixture model.  Under the assumption that each local
data distribution is a mixture of unknown underlying distributions,
their algorithm computes a model corresponding to each
distribution. On the other hand, we focus on the curved exponential
family, with variance reduction, partial participation and compression
and on limiting the impact of heterogeneity, but do not address
personalization.

\textbf{Acknowledgments}
The work of A. Dieuleveut and E. Moulines is partially supported by
ANR-19-CHIA-0002-01 /chaire SCAI, and Hi!Paris.  The work of G. Fort is partially
supported by the Fondation Simone et Cino del Duca under the project
OpSiMorE.

\textbf{Broader Impact of this work} This work is mostly theoretical, and we believe it does not currently present any direct societal consequence. However, the methods described in this paper can be used to train machine learning models which could themselves have societal consequences. For instance, the deployment of machine learning models can suffer from gender and racial bias, or amplify existing inequalities.

\newpage


\appendix

\newpage
\begin{center}
\begin{LARGE}
\textbf{Supplementary materials for ``Federated Expectation Maximization with heterogeneity mitigation and variance reduction"}
\end{LARGE}
\end{center}


This supplementary material is organized as
follows. \Cref{app:mainresultsPP} extends the results obtained in
\Cref{theo:dianaem} to the Partial Participation
regime. \Cref{app:quantization} contains additional details on
compression mechanisms satisfying \Cref{hyp:var:quantif}, including an
example of admissible quantization operator. \Cref{app:proof-fedem}
contains the pseudo-code for algorithm \FEDEM~in the full
participation regime case, and the proof of \Cref{theo:dianaem} --
including necessary technical lemmas. \Cref{app:PP} contains details
concerning the extension to partial participation of the workers and
the proof of \Cref{prop:PP}. \Cref{app:VRFEDEM} is devoted to the
proof of \Cref{theo:DS} concerning the convergence of \VRFEDEM~and
necessary technical results; it also contains a discussion on the
complexity of \VRFEDEM~in terms of conditional expectations
evaluations. Finally, \Cref{app:numerical} contains additional details
about the latent variable models used in the numerical section, as
well as the pseudo code for \texttt{FedMissEM}. \\ Note that, in order
to make our numerical results reproducible, code is also provided as
supplementary material.

\paragraph{Notations}
For two vectors $a,b \in \rset^q$, $\pscal{a}{b}$ is the Euclidean
standard scalar product, and $\|\cdot \|$ denotes the associated
norm. For $r\geq 1$, $\|a\|_r$ is the $\ell_r$-norm of a vector
$a$. The Hadamard product $a \odot b$ denotes the entrywise product of
the two vectors $a,b$. By convention, vectors are column-vectors. For
a matrix $A$, $A^{\top}$ denotes its transpose and $\|A\|_F$ is its
Frobenius norm.  For a positive integer $n$, set
$[n]^\star \eqdef \{1, \cdots, n\}$ and $[n] \eqdef \{0, \cdots,
n\}$.
The set of non-negative integers (resp. positive) is denoted by
$\nset$ (resp. $\nset^\star$).  The minimum (resp. maximum) of two
real numbers $a,b$ is denoted by $a \wedge b$ (resp. $a \vee b$).  We
will use the Bachmann-Landau notation $a(x) = O(b(x))$ to characterize
an upper bound of the growth rate of $a(x)$ as being
$b(x)$. \\
We denote by $\mathcal{K}_p(\mu,\Sigma)$ the Gaussian distribution in
$\rset^p$, with expectation $\mu$ and covariance matrix~$\Sigma$.

\section{Results for \FEDEM~with partial participation and compression.}\label{app:mainresultsPP}
In this paragraph, we extend the results of \Cref{theo:dianaem} to the
\textit{Partial Participation} (PP) regime, in which only a fraction of the
workers participate to the training at each step of the learning
process. This is a key feature in the FL framework, as individuals may
not always be available or willing to participate
\cite{mcmahan_communication-efficient_2017}.  To analyze the
convergence in this situation, we make the following assumption.
\begin{assumption} \label{hyp:PP} For all $ k\in [\kmax-1] $,  $ \set_{k+1} \eqdef \{i \in [n]^\star  \ \text{s.t.} \ B_{k+1,i}=1\} $ where the random variables $B_{k+1,i}$ for $ i\in [n]^\star$ and $k\in [\kmax-1]$  are independent Bernoulli random variables with success probability $p \in \ooint{0,1}$.
\end{assumption}
This assumption is standard in the FL
literature~\citep{sattler_robust_2019,tang_doublesqueeze_2019,philippenko2020artemis},
and can easily be extended to worker dependent probabilities of
participation~\citep{horvath_better_2020}. 

\textbf{Usage of the control variates $ (V_{k,i})_{i\in [n]^*} $ with PP.} We have $V_k = n^{-1}\sum_{i=1}^n V_{k,i}$ for all
$k \geq 0$ (see \Cref{prop:meanV:PP}) even when the workers are not
all active at iteration $\# k$. A noteworthy point is
that, upon receiving $\Q(\Delta_{k+1,i})$ for all $i \in \set_{k+1}$,
the central server computes
\[
H_{k+1} = V_k + (np)^{-1} \sum_{i\in \set_{k+1}} \Q(\Delta_{k+1,i})
\]
and \textit{not}
\[
(np)^{-1} \sum_{i\in \set_{k+1}} (V_{k,i} +
\Q(\Delta_{k+1,i}) ) \eqsp.
\]
Though the later solution may appear more natural, it would actually
not only require to store all values $V_{k,i}$ for $i\in [n]^*$ on
the central server, but also impair convergence in the heterogeneous
setting. Indeed, even in the \textit{uncompressed} regime, in which
$\Q(\Delta_{k+1,i})=\Delta_{k+1,i}$, our algorithm differs from a
naive implementation of a distributed EM: \FEDEM~computes
\[
H_{k+1} =
V_k - (np)^{-1} \sum_{i \in \set_{k+1}} V_{k,i} + (np)^{-1} \sum_{i
  \in \set_{k+1}} \left( \Smem_{k+1,i} - \hatS_k \right)
\]
while a naive distributed EM would compute
\[
H_{k+1}^{\mathrm{dEM}}  \eqdef 
(np)^{-1} \sum_{i \in \set_{k+1}} \left( \Smem_{k+1,i} - \hatS_k
\right) \eqsp.
\]
Such an update $H_{k+1}^{\mathrm{dEM}}$ is expected not to be robust
to data heterogeneity as proved in \cite{philippenko2020artemis} for
the Stochastic Gradient algorithm in the FL setting.

The following theorem extends \Cref{theo:dianaem} to the partial
participation regime. Its proof is in \Cref{app:PP}.
\begin{theorem}
	\label{prop:PP}
	Assume \Cref{hyp:model} to \Cref{hyp:PP} and set $L^2 \eqdef n^{-1}
	\sum_{i=1}^n L_i^2$, $\sigma^2 \eqdef n^{-1} \sum_{i=1}^n
	\sigma_i^2$.  Let $\{\hatS_k, k \in [\kmax] \}$ be given by
	\autoref{algo:dianaem-main}, run with $\alpha\eqdef (1+\omega)^{-1}$
	and $\pas_k = \pas \in\ocint{0, \gamma_{\max}}$, where
	\[
	\gamma_{\max} \eqdef \frac{v_{\min}}{2L_{\dot \lyap}} \wedge \frac{
		p \sqrt{n}}{2 \sqrt{2} L (1+\omega) \sqrt{ \omega + (1-p) (1+\omega)/p}} \eqsp.
	\]
	Denote by $K$ the uniform random variable on $[\kmax-1]$. Then,
	taking $V_{0,i} \eqdef \mf_i(\hatS_0)$ for $i \in [n]^\star$, we get
	\[
	v_{\min} \left(1 - \pas \frac{L_{\dot \lyap}}{v_{\min}} \right)
	\PE\ \left[ \|\mf(\hatS_K)\|^2 \right] \leq \frac{\left(
		\lyap(\hatS_0) - \min \lyap \right)}{\pas \kmax} + \pas L_{\dot
		\lyap} \frac{1+5 \left( \omega + (1-p) (1+\omega) /p \right)}{n}
	\sigma^2 \eqsp.
	\]
\end{theorem}
The above expressions can be simplified upon noting
that $\omega + (1-p) (1+\omega)/p \leq (1+\omega)/p$. When $p=1$,
\Cref{theo:dianaem} and \Cref{prop:PP} coincide.
More generally, \Cref{prop:PP}  highlights that partial participation impacts both the limiting variance (which increases by a factor proportional to $p^{-1}$) and the maximal learning rate.

\section{An example of quantization mechanisms: the block-$p$-quantization}
\label{app:quantization}
In this section, we recall the definition of a common lossy data
compression mechanism in FL (see,
e.g. \cite{mishchenko2019distributed}), called block-$p$-quantization,
and demonstrate that such quantizations satisfy the assumptions
required to derive our theoretical results. 

\paragraph{Block-$p$-quantization.}  Let $x \in \rset^q$. Choose
$\{q_\ell, \ell \in [m]^\star\}$ a sequence of positive integers such
that $\sum_{\ell=1}^m q_\ell = q$; and $p \in \nset^\star$.  For
$x\in \rset^q$, we define the block partition
\[ x = \left[\begin{matrix} x_{(1)} \\ \cdots \\ x_{(m)} \end{matrix}
\right],\eqsp x_{(l)}\in\rset^{q_\ell} \text{ for all
}\ell\in[m]^{\star}.\] For all $\ell \in [m]^\star$, set
\begin{equation}
 \label{eq:app:Q-def1}
 \hat X_{(\ell)} \eqdef \|x_{(\ell)}\|_p \left[\begin{matrix}
     \mathrm{sign}(x_{(\ell),1}) \\ \cdots
     \\ \mathrm{sign}(x_{(\ell),q_\ell}) \end{matrix} \right] \odot
 \left[\begin{matrix} U_{\ell,1} \\ \cdots
     \\ U_{\ell,q_\ell} \end{matrix} \right] \qquad U_{\ell,j}
 \stackrel{indep}{\sim} \mathcal{B}\left(
 \frac{|x_{(\ell),j}|}{\|x_{(\ell)}\|_p}\right) \eqsp,
\end{equation}
where $x_{(\ell)} = (x_{(\ell),1}, \cdots,
  x_{(\ell),q_\ell})^{\top} \in \rset^{q_\ell}$ and $\mathcal{B}(u)$
  denotes the Bernoulli random variable with success probability $u$.
The block-$p$-quantization operator $\Q: \rset^q \to \rset^q$ is
defined by
\begin{equation}
\label{eq:app:Q-def2}
\Q(x) \eqdef \left[ \begin{matrix} \hat X_{(1)} \\ \cdots \\ \hat X_{(m)} \end{matrix} \right] \eqsp.
\end{equation}
The following Lemma ensures the block-$p$-quantization operator $\Q$
satisfies the assumption \Cref{hyp:var:quantif} on the compression
mechanism required by \Cref{theo:dianaem}, \Cref{prop:PP} and
\Cref{theo:DS}.
\begin{lemma} \label{lem:quant} Let $p \in \nset^\star$ and
  $\{q_\ell, \ell \in [m]^\star \}$ be positive integers such that
  $\sum_{\ell=1}^m q_\ell =q$. For any $x \in \rset^q$, we have
\[
\PE\left[\Q(x) \right] = x \eqsp, \qquad \PE\left[ \| \Q(x) - x \|^2
\right] = \sum_{\ell=1}^m \left(\|x_{(\ell)}\|_1 \|x_{(\ell)}\|_p -
  \|x_{(\ell)}\|^2 \right) \eqsp,
  \]
  where $\Q$ is the block-$p$-quantization operator defined
  in~\eqref{eq:app:Q-def1} and \eqref{eq:app:Q-def2}. Thus,
  \Cref{hyp:var:quantif} holds. In particular, for $p=2$, we may take
  $\omega= \max_{\ell \in [m]^*} (\sqrt{q_{\ell}} -1)$.
\end{lemma}
\begin{proof}
We start by noticing that, for all $\ell \in [m]^\star$,
  $ \left(\Q(x) \right)_{(\ell)} = \hat X_{(\ell)}$. Furthermore,
\begin{align*}
\PE\left[ \hat X_{(\ell)} \right] &= \|x_{(\ell)}\|_p
\ \left[\begin{matrix} \mathrm{sign}(x_{(\ell),1}) \\ \cdots
    \\ \mathrm{sign}(x_{(\ell),q_\ell}) \end{matrix} \right] \odot
\left[\begin{matrix} \PE\left[U_{\ell,1} \right] \\ \cdots
    \\ \PE\left[ U_{\ell,q_\ell} \right] \end{matrix} \right] =
\|x_{(\ell)}\|_p \ \left[\begin{matrix} \mathrm{sign}(x_{(\ell),1})
    \\ \cdots \\ \mathrm{sign}(x_{(\ell),q_\ell}) \end{matrix} \right]
\odot \left[\begin{matrix} \frac{|x_{(\ell),1}|}{\|x_{(\ell)}\|_p}
    \\ \cdots
    \\ \frac{|x_{(\ell),q_\ell}|}{\|x_{(\ell)}\|_p} \end{matrix}
  \right] \\ & = \left[\begin{matrix} \mathrm{sign}(x_{(\ell),1})
    \\ \cdots \\ \mathrm{sign}(x_{(\ell),q_\ell}) \end{matrix} \right]
\odot \left[\begin{matrix} |x_{(\ell),1}| \\ \vdots
    \\ |x_{(\ell),q_\ell}| \end{matrix} \right] = \left[\begin{matrix}
    x_{(\ell),1} \\ \vdots \\ x_{(\ell),q_\ell} \end{matrix} \right] =
x_{(\ell)} \eqsp,
      \end{align*}
which concludes the proof of the first statement. To prove the second statement, we write
\[
\| \Q(x) - x \|^2 = \sum_{\ell=1}^m \| \hat X_{(\ell)} - x_{(\ell)}
\|^2 = \sum_{\ell=1}^m \|x_{(\ell)}\|_p^2 \, \sum_{j=1}^{q_\ell}
\left( U_{\ell,j} - \PE\left[U_{\ell,j} \right] \right)^2 \eqsp.
\]
Since $U_{\ell,j}$ is a Bernouilli random variable with parameter $|x_{(\ell),j}| / \|x_{(\ell)}\|_p$, it holds that
\[
\PE\left[ \left( U_{\ell,j} - \PE\left[U_{\ell,j} \right] \right)^2
  \right] = \frac{|x_{(\ell),j}| \left(\|x_{(\ell)}\|_p - |x_{(\ell),j}|\right)}{\|x_{(\ell)}\|_p^2} \eqsp.
\]
Hence
\begin{multline*}
\PE\left[ \| \Q(x) - x \|^2 \right] = \sum_{\ell=1}^m
\sum_{j=1}^{q_\ell} \left\{ |x_{(\ell),j}| \left(\|x_{(\ell)}\|_p -
    |x_{(\ell),j}|\right) \right\} \\
    = \sum_{\ell=1}^m
\left(\|x_{(\ell)}\|_1 \|x_{(\ell)}\|_p - \|x_{(\ell)}\|^2 \right) \eqsp,
\end{multline*}
which proves the second statement. In the particular case where $p=2$,
using the fact that
$\| x_{(\ell)} \|_1 \leq \sqrt{q_\ell} \| x_{(\ell)} \|$, we obtain that
\[
\PE\left[ \| \Q(x) - x \|^2 \right] \leq \sum_{\ell=1}^m
(\sqrt{q_\ell}-1) \| x_{(\ell)} \|^2 \leq \max_{\ell \in [m]^*}
(\sqrt{q_{\ell}} -1) \, \| x \|^2,
\]
which concludes the proof.
\end{proof}

\section{Convergence analysis of {\tt \FEDEM}}
\label{app:proof-fedem}
 This section contains all the elements to derive the convergence
 analysis of \FEDEM~developed in \Cref{sec:FEDEM} in the full
 participation regime. The analysis is organized as follows. First,
 \Cref{app:sub:fedem-code} gives the pseudo code of the
 \FEDEM~algorithm; \Cref{sec:FEDEM:tribu} introduces rigorous
 definitions for filtrations and a technical Lemma, and
 \Cref{app:proof-dianaem-preliminary} presents preliminary
 results. Then, the proof of \Cref{theo:dianaem} is given in
 \Cref{app:proof-dianaem} and the proof of \Cref{cor:diana-em} is in
 \Cref{proof:coro:dianaem}.

The assumptions \Cref{hyp:model} to \Cref{hyp:Tmap} are assumed
throughout this section.

\newpage

\subsection{Pseudo code of the \FEDEM~algorithm}
\label{app:sub:fedem-code}

For the sake of completeness of the supplementary material, we start
by recalling the pseudo code which defines the \FEDEM~sequence in the
\textbf{full participation regime.} It is given in
\autoref{algo:dianaemNoPP} below.

\begin{algorithm}[htbp]
  \caption{\FEDEM~\label{algo:dianaemNoPP}} \KwData{ $\kmax \in
    \nset^\star$; for $i \in [n]^\star$, $V_{0,i} \in \rset^q$;
    $\hatS_0 \in \rset^q$; a positive sequence $\{\pas_{k+1}, k \in
         [\kmax-1]\}$; $\alpha>0$} \KwResult{ The sequence:
    $\{\hatS_{k}, k \in [\kmax]\}$} Set $V_0 = n^{-1} \sum_{i=1}^n
  V_{0,i}$ \label{algo:FEDEM:init:V0}\; \For{$k=0, \ldots,
      \kmax-1$}{\For{$i=1, \ldots, n$} {{\em (worker $\# i$)} \;
        Sample $\Smem_{k+1,i}$, an approximation of $\bars_i \circ
        \map(\hatS_k)$ \; Set $\Delta_{k+1,i} = \Smem_{k+1,i} -
        V_{k,i} - \hatS_k$ \; Set $V_{k+1,i} = V_{k,i} + \alpha\,
        \Q(\Delta_{k+1;i})$. Send $\Q(\Delta_{k+1;i})$ to the central
        server \;} {\em (the central server)} \; Compute $H_{k+1} =
      V_k + n^{-1} \sum_{i=1}^n \Q(\Delta_{k+1;i})$ \; Set
      $\hatS_{k+1} = \hatS_k + \pas_{k+1} H_{k+1}$ \; Set $V_{k+1} =
      V_k + \alpha n^{-1} \sum_{i=1}^n \Q(\Delta_{k+1;i})$ \; Send
      $\hatS_{k+1}$ and $\map (\hatS_{k+1})$ to the $n$ workers}
\end{algorithm}

\subsection{Notations and technical lemma}
\label{sec:FEDEM:tribu}
In this section, we start by introducing the appropriate filtrations employed later on to define conditional expectations. Then, we present a technical lemma used in the main proof of \Cref{theo:dianaem} (see \Cref{app:proof-dianaem}).
\paragraph{Notations.}
For any random variable $U$, we denote by $\sigma(U)$ the
sigma-algebra generated by $U$. For $n$ sigma-algebras
$\{ \F_k, k \in [n]^\star\}$, we denote by $\bigvee_{k=1}^n \F_k$ the
sigma-algebra generated by $\{ \F_k, k \in [n]^\star \}$.

\paragraph{Definition of filtrations.}
Let us define the following filtrations. For any $i \in [n]^\star$, we
set
$$\F_{0,i} = \F_{0,i}^+ \eqdef \sigma\left(\hatS_0; V_{0,i} \right)
\text{ and } \qquad \F_{0} \eqdef \bigvee_{i=1}^n \F_{0,i} \eqsp.$$
Then, for all $k\geq 0$,
\begin{enumerate}[(i)]
\item$ \F_{k+1/2,i} \eqdef \F_{k,i}^+ \vee \sigma\left( \Smem_{k+1,i}
  \right)$,
  \item $\F_{k+1,i} \eqdef \F_{k+1/2,i} \vee
  \sigma\left( \Q(\Delta_{k+1,i}) \right)$,
  \item $\F_{k+1} \eqdef
  \bigvee_{i=1}^n \F_{k+1,i}$,
  \item $\F_{k+1,i}^+ \eqdef
  \F_{k+1,i} \vee \F_{k+1}$.
\end{enumerate}
Note that, with these notations, for $k \geq 0$ and $i \in [n]^\star$,
the random variables of the \FEDEM~sequence defined in
Algorithm~\ref{algo:dianaemNoPP} belong to the filtrations defined above
as follows:
\begin{enumerate}[(i)]
\item $\hatS_{k} \in \F_{k,i}^+$, $\hatS_k \in \F_k$,
\item  $\Smem_{k+1,i} ,\Delta_{k+1,i} \in \F_{k+1/2,i}$,
\item  $V_{k+1,i} \in \F_{k+1,i}$,
\item $\hatS_{k+1} , H_{k+1},V_{k+1} \in \F_{k+1}$.
\end{enumerate}
Note also that we have the following inclusions for filtrations:
$\F_k \subset \F_{k,i}^+ \subset \F_{k+1/2,i} \subset \F_{k+1,i}
\subset \F_{k+1}$ for all $i \in [n]^\star$.

\paragraph{Elementary lemma.}
In the main proof of \Cref{theo:dianaem}, we use the following elementary lemma.
\begin{lemma}\label{lem:polar}
  For any $x,y\in \rset^q$ and for any $\alpha \in \rset$, one has:
	\begin{equation*}
	\|\alpha x + (1-\alpha) y\|^2 = \alpha \|x\|^2 + (1-\alpha) \|y\|^2 - \alpha (1-\alpha) \|x-y\|^2.
	\end{equation*}
\end{lemma}
\begin{proof}
  The LHS is equal to
  \[
\alpha^2 \|x \|^2 + (1-\alpha)^2 \|y\|^2 + 2 \alpha(1-\alpha) \pscal{x}{y} \eqsp.
\]
The RHS is equal to
\[
\alpha \|x\|^2 + (1-\alpha) \|y \|^2 - \alpha (1-\alpha) \left(
\|x\|^2 + \|y\|^2 - 2\pscal{x}{y} \right) \eqsp.
\]
The proof is concluded upon noting that $\alpha - \alpha(1-\alpha) =
\alpha^2$ and $(1-\alpha) - \alpha (1-\alpha) = (1-\alpha)^2$.
  \end{proof}

\subsection{Preliminary results}
\label{app:proof-dianaem-preliminary}
In this section, we gather preliminary results on the control of the bias and variance of random variables of interest, which will be used in the main proof of \Cref{theo:dianaem}. Namely, \Cref{prop:fieldH} controls the random field $H_{k+1}$, \Cref{prop:FedEM:Delta} controls the local increments $\Delta_{k+1,i}$ and \Cref{prop:varcont2} controls the memory term $V_{k,i}$.

\subsubsection{Results on the memory terms $V_{k}$.}
\Cref{prop:sumV} shows that, even if the central
  server only receives the variation $\alpha^{-1}(V_{k+1,i} -
  V_{k,i})$ from each local worker $\# i$, it is able to compute
  $n^{-1} \sum_{i=1}^n V_{k+1,i}$ as soon as the quantity $V_0$ is
  correctly initialized.
\begin{proposition} \label{prop:sumV}
For any $k \in [\kmax]$, we have 
\[
V_{k} = \frac{1}{n} \sum_{i=1}^n V_{k,i} \eqsp.
\]
\end{proposition}
\begin{proof}
  The proof is by induction on $k$.  When $k=0$, the property holds
  true by Line~\ref{algo:FEDEM:init:V0} in
  \autoref{algo:dianaemNoPP}. Assume that the property holds for
  $k \leq \kin-2$. Then by definition of $V_{k+1}$ and by the
  induction assumption:
  \begin{align*}
  V_{k+1} &= V_{k} + \alpha \frac{1}{n} \sum_{i=1}^n
  \Q(\Delta_{k+1,i}) = \frac{1}{n} \sum_{i=1}^n \left( V_{k,i}
  + \alpha \Q(\Delta_{k+1,i}) \right) \\ &= \frac{1}{n}
  \sum_{i=1}^n V_{k+1,i} \eqsp.
  \end{align*}
  This concludes the induction.
\end{proof}

\subsubsection{Results on the random field $H_{k+1}$.} We compute in
\Cref{prop:fieldH} the conditional expectation of $H_{k+1}$ with
respect to the appropriate filtration $\F_k$ defined in
\Cref{sec:FEDEM:tribu}, as well as an upper bound on its
variance. These results are combined in an upper bound on the
conditional expectation of the square norm $\|H_{k+1}\|^2$ in
\Cref{coro:Hmoment2}.

\Cref{prop:fieldH} shows that the stochastic field
  $H_{k+1}$ is a (conditionally) unbiased estimator of $\mf(\hatS_k)$.
  In the case of no compression (i.e. $\omega=0$), the conditional
  variance of $H_{k+1}$ is $\sigma^2/n$ where $\sigma^2$ is the mean
  variance of the approximations $\Smem_{k+1,i}$ over the $n$ workers
  (see \Cref{hyp:variance:oracle}); when $\sup_i \sigma^2_i < \infty$,
  the variance is inversely proportional to the number of workers
  $n$.
\begin{proposition} \label{prop:fieldH} Assume \Cref{hyp:var:quantif} and
  \Cref{hyp:variance:oracle} and set
  $\sigma^2 \eqdef n^{-1} \sum_{i=1}^n \sigma_i^2$.  For any
  $k \geq 0$,
\begin{align}
\label{eq:fieldH1}
\CPE{H_{k+1}}{\F_k} &= \mf(\hatS_k)  \eqsp, \\
\label{eq:fieldH1}
  \CPE{\| H_{k+1} - \CPE{H_{k+1}}{\F_k} \|^2}{\F_k}
                    &\le \frac{\omega}{n} \left( \frac{1}{n} \sum_{i=1}^n \CPE{\| \Delta_{k+1,i}\|^2}{\F_k} \right) +\frac{\sigma^2}{n}  \eqsp.
\end{align}
\end{proposition}
\begin{proof}
  Let $k \geq 0$. \Cref{hyp:var:quantif} guarantees
\begin{align} \label{eq:Quant:1/2}
\CPE{ \sum_{i=1}^n \Q(\Delta_{k+1,i})}{\F_{k+1/2,i}} & =
\sum_{i=1}^n \CPE{\Q(\Delta_{k+1,i})}{\F_{k+1/2,i}} \nonumber \\
  &= \sum_{i=1}^n  \{ \Smem_{k+1,i} - V_{k,i} - \hatS_k \} \eqsp.
\end{align}
Note also that, by \Cref{hyp:variance:oracle},
$\CPE{\Smem_{k+1,i}}{\F_{k,i}^+} = \bars_i \circ \map(\hatS_k)$, and
that $V_k\in \F_{k}$ and $\F_k \subset \F_{k,i}^+ \subset
\F_{k+1/2,i}$ (see \Cref{sec:FEDEM:tribu}). Combined
with~\eqref{eq:Quant:1/2} and using that $n^{-1} \sum_{i=1}^n V_{k,i}=
V_k$ (see \Cref{prop:sumV}), this yields
\begin{equation*}
  \CPE{H_{k+1}}{\F_{k}}  =\CPE{n^{-1}
    \sum_{i=1}^n \Q(\Delta_{k+1,i})}{\F_{k}} + V_k \\
  = \frac{1}{n} \sum_{i=1}^n
  \bars_i \circ \map(\hatS_k) - \hatS_k = \mf(\hatS_k)\eqsp.
\end{equation*}
We now prove the second statement, and start by writing
\begin{align*}
H_{k+1} - \mf(\hatS_k)  & = \frac{1}{n} \sum_{i=1}^n
\Q(\Delta_{k+1,i}) + V_k - \frac{1}{n} \sum_{i=1}^n \bars_i \circ
\map(\hatS_k) + \hatS_k \\ & = \frac{1}{n} \sum_{i=1}^n
\left\{\Q(\Delta_{k+1,i}) - \CPE{\Q(\Delta_{k+1,i})}{
  \F_{k+1/2,i}} \right\} \\ & \quad + \frac{1}{n} \sum_{i=1}^n
\{ \Smem_{k+1,i} - \bars_i \circ \map(\hatS_k) \} \eqsp,
\end{align*}
where we applied \eqref{eq:Quant:1/2} to obtain the last equality. Using the fact that $\Smem_{k+1,i} - \bars_i \circ \map(\hatS_k) \in \F_{k+1/2,i}$
and since, conditionally to $\F_k$, the workers are independent we have
\begin{align*}
\CPE{\| H_{k+1} - \mf(\hatS_k) \|^2}{\F_k}
& = \frac{1}{n^2} \sum_{i=1}^n \CPE{\| \Q(\Delta_{k+1,i})
  - \CPE{\Q(\Delta_{k+1,i})}{\F_{k+1/2,i}} \|^2}{\F_k}  \\
& + \frac{1}{n^2} \sum_{i=1}^n \CPE{\| \Smem_{k+1,i}
  - \bars_i \circ \map(\hatS_k) \|^2}{\F_k} \eqsp.
\end{align*}
The second terme in the RHS is upped bounded by $n^{-1} \sigma^2$ (see
\Cref{hyp:variance:oracle}).  For the first term, using
\Cref{hyp:var:quantif} and since $\Delta_{k+1,i} \in \F_{k+1/2,i}$,
for any $i \in [n]^\star$ we have
\begin{align*}
& \CPE{\| \Q(\Delta_{k+1,i}) - \CPE{\Q(\Delta_{k+1,i})}{\F_{k+1/2,i}} \|^2}{ \F_{k+1/2,i}}  \\
&  = \CPE{\| \Q(\Delta_{k+1,i}) \|^2}{\F_{k+1/2,i}} - \|\Delta_{k+1,i}\|^2 \\
& \qquad \leq (1+\omega) \| \Delta_{k+1,i} \|^2 - \| \Delta_{k+1,i}\|^2 = \omega \| \Delta_{k+1,i}\|^2 \eqsp,
\end{align*}
which concludes the proof upon conditioning with respect to $\F_k$.
\end{proof}
\begin{corollary}[of \Cref{prop:fieldH}] \label{coro:Hmoment2}
	\begin{align*}
          \CPE{\| H_{k+1} \|^2}{\F_k} \le \|
          \mf(\hatS_k) \|^2+ \frac{\omega}{n} \left( \frac{1}{n}\sum_{i=1}^n
          \CPE{\| \Delta_{k+1,i}\|^2}{\F_k}  \right) +\frac{\sigma^2}{n}  \eqsp.
	\end{align*}
 \end{corollary}

 \subsubsection{Results on the local increments $\Delta_{k+1,i}$.} We
 compute in \Cref{prop:FedEM:Delta} an upper bound on the second
 conditional moment of $\Delta_{k+1,i}$, with respect to the
 appropriate filtration $\F_k$ (see \Cref{sec:FEDEM:tribu}).

\begin{proposition}  \label{prop:FedEM:Delta}
  Assume \Cref{hyp:variance:oracle}.  For any $i \in [n]^\star$ and $k
  \in [\kmax -1]$,
  \[
\CPE{\|\Delta_{k+1,i} \|^2}{\F_k} \leq \| V_{k,i} - \mf_i(\hatS_k) \|^2 +\sigma_i^2 \eqsp.
  \]
  \end{proposition}
\begin{proof}
  Let $i \in [n]^\star$ and $k \in [\kmax -1]$.  By
  \Cref{hyp:variance:oracle},
  $\CPE{\Smem_{k+1,i} - \hatS_k}{\F_{k,i}^+} = \mf_i(\hatS_k)$; in
  addition, $\hatS_k \in \F_k$, $V_{k,i} \in \F_{k,i}^+$ and
  $\F_k \subset \F_{k,i}^+$. Hence, we get
\begin{align}
  \nonumber
  \CPE{\|\Delta_{k+1,i}\|^2}{\F_{k,i}^+}
  &= \CPE{\|\Smem_{k+1,i} - V_{k,i} - \hatS_k\|^2}{\F_{k,i}^+} \\
  &= \|\mf_i(\hatS_k)- V_{k,i}\|^2 +\CPE{\| \Smem_{k+1,i} - \hatS_k -\mf_i(\hatS_k) \|^2}{\F_{k,i}^+}
    \nonumber \\
  & = \|\mf_i(\hatS_k)- V_{k,i}\|^2 +\CPE{\| \Smem_{k+1,i} - \bars_i \circ \map(\hatS_k) \|^2}{\F_{k,i}^+}  \nonumber 
  \\
  &\overset{\Cref{hyp:variance:oracle}}{\le} \| \mf_i(\hatS_k)-
    V_{k,i}\|^2 +\sigma_i^2 \eqsp.
\end{align}
The proof is concluded upon noting that $\F_k \subset \F_{k,i}^+$,
$\hatS_k \in \F_k$ and $V_{k,i} \in \F_k$.
\end{proof}

\subsubsection{Results on the memory terms $V_{k,i}$.} Our final
preliminary result is to compute in \Cref{prop:varcont2} an upper
bound to control the conditional variance of the local memory terms
$V_{k,i}$ with respect to the appropriate filtration $\F_k$ (see
\Cref{sec:FEDEM:tribu}).
\begin{proposition} \label{prop:varcont2} Assume \Cref{hyp:lipschitz},
  \Cref{hyp:var:quantif} and \Cref{hyp:variance:oracle}; set
  $L^2 \eqdef n^{-1} \sum_{i=1}^n L_i^2$ and
  $\sigma^2 \eqdef n^{-1} \sum_{i=1}^n \sigma_i^2$. For any
  $k \geq 0$, set
        \[
        G_k \eqdef \frac{1}{n}\sum_{i=1}^n \| V_{k,i} - \mf_i(\hatS_k)
        \|^2 \eqsp.
        \]
        For any $k \in [\kmax -1]$ and  $\alpha \in \ocint{0,(1/(1+\omega))}$, it holds that
	\begin{multline*}
          \CPE{G_{k+1}}{\F_k} \le \left (1-\frac{\alpha}{2} + 2
            \pas_{k+1}^2 \frac{L^2}{\alpha} \frac{\omega}{n}\right )
          G_k
          + 2 \pas_{k+1}^2 \frac{L^2}{\alpha} \| \mf(\hatS_k) \| ^2 \\
          + 2 \left( \alpha + \pas_{k+1}^2 \frac{L^2}{\alpha}
            \frac{1+\omega}{n}\right ) \sigma^2 \eqsp.
	\end{multline*}
\end{proposition}
\begin{proof}
  We start by computing an upper bound for the local conditional
  expectations $\CPE{\|V_{k+1,i}-\mf_i(\hatS_{k+1})\|^2}{\F_{k}}$,
  $i\in[n]^{\star}$ and then derive the result of \Cref{prop:varcont2}
  by averaging over the $n$ local workers.

Let $i\in[n]^{\star}$; from \Cref{lem:polar}, we have for any $s \in \rset^q$
\begin{multline*}
  \left\|\CPE{V_{k+1,i} -s}{\F_{k+1/2,i}} \right\|^2  =
   \|(1-\alpha) \, (V_{k,i} -s) + \alpha \ (\Smem_{k+1,i} - \hatS_k
   -s) \|^2 \\
    = (1-\alpha) \, \| V_{k,i} -s \|^2
   + \alpha \| \Smem_{k+1,i} - \hatS_k-s \|^2 - \alpha (1-\alpha) \|
   \Delta_{k+1,i}\|^2 \eqsp.
\end{multline*}
On the other hand,
  \[
 \left \|V_{k+1,i} - \CPE{V_{k+1,i}}{\F_{k+1/2,i}} \right\|^2 =
  \alpha^2 \left\|\Q(\Delta_{k+1,i}) - \CPE{\Q(\Delta_{k+1,i})}{\F_{k+1/2,i}}\right \|^2
  \]
  and by \Cref{hyp:var:quantif} (see the proof of
  Proposition~\autoref{prop:fieldH} for the same computation)
    \begin{align*}
 \CPE{\left\|V_{k+1,i} - \CPE{V_{k+1,i}}{\F_{k+1/2,i}}
   \right\|^2}{\F_{k+1/2,i}}  \leq  \alpha^2 \omega \| \Delta_{k+1,i}\|^2 \eqsp.
    \end{align*}
    Hence
 \begin{multline} \label{eq:tool127}
 \CPE{\|V_{k+1,i} -s \|^2}{\F_{k+1/2,i}} \leq 
\CPE{\left\|V_{k+1,i}-s   - \CPE{V_{k+1,i} -s}{\F_{k+1/2,i}} \right\|^2}{\F_{k+1/2,i}}  \\
+ \CPE{\left\| \CPE{V_{k+1,i} -s }{\F_{k+1/2,i}}\right\|^2}{\F_{k+1/2,i}} \\
 \leq
 (1-\alpha) \, \| V_{k,i} -s \|^2 + \alpha \| \Smem_{k+1,i} - \hatS_k
 -s \|^2  + \alpha \left( \alpha (1+ \omega) -1 \right)
 \|\Delta_{k+1,i}\|^2 \eqsp.
 \end{multline}
For any $\beta>0$, using that $\|a+b\|^2 \leq (1+\beta^2) \|a\|^2 + (1+\beta^{-2}) \|b\|^2$, we have
\begin{align*}
& \CPE{\|V_{k+1,i} - \mf_i(\hatS_{k+1}) \|^2}{\F_{k}} \\ 
& \le (1+\beta^{-2}) \CPE{\|V_{k+1,i} - \mf_i(\hatS_{k}) \|^2}{\F_{k}} +
(1+\beta^2)\CPE{\|\mf_i(\hatS_{k}) - \mf_i(\hatS_{k+1})\|^2}{\F_{k}} \\ 
&\overset{\Cref{hyp:lipschitz}}{\le} (1+\beta^{-2})
          \CPE{\CPE{\|V_{k+1,i} - \mf_i(\hatS_{k}) \|^2}{\F_{k+1/2,i}}}{\F_{k}} +
          (1+\beta^2) L_i^2 \pas_{k+1}^2 \PE[\| H_{k+1} \| ^2 \vert
            \F_{k} ] \\ &\overset{\eqref{eq:tool127}}{\le}
          (1+\beta^{-2}) \bigg ( (1-\alpha) \, \| V_{k,i}
          -\mf_i(\hatS_{k}) \|^2 \\ & + \alpha \PE[ \| \Smem_{k+1,i} -
            \hatS_k -\mf_i(\hatS_{k}) \|^2 \vert \F_k] + \alpha \left(
          \alpha (1+ \omega) -1 \right ) \CPE{\|\Delta_{k+1,i}\|^2}{\F_k} \bigg ) \\ 
          & + (1+\beta^2) L_i^2 \pas_{k+1}^2 \CPE{\|H_{k+1} \| ^2}{\F_{k}} \eqsp,
	\end{align*}
	where we have used \eqref{eq:tool127} with $ s=
        \mf_i(\hatS_{k}) \in \F_k \subset \F_{k+1/2,i}$. Choose $\beta >0$ such that
        \[
        \beta^{-2} \eqdef \left\{ \begin{array}{cc}
          \frac{\alpha}{2(1-\alpha)} &  \text{if $\alpha \leq 2/3$} \\
          1 & \text{if $\alpha \geq 2/3$}
          \end{array}
        \right.
        \]
which implies that $ (1+\beta^{-2})(1-\alpha) \leq 1-\alpha/2$; note also
that $1 \leq 1+\beta^{-2} \le 2$. By \Cref{coro:Hmoment2}, we have (remember that $ \alpha(1+\omega)-1 \leq 0$)
\begin{align*}
  & \CPE{\|V_{k+1,i} - \mf_i(\hatS_{k+1}) \|^2}{\F_{k}}
    {\le} \left (1-\frac{\alpha}{2} \right ) \, \| V_{k,i}
    -\mf_i(\hatS_{k}) \|^2 \\ 
  & + 2 \alpha \CPE{\| \Smem_{k+1,i} - \bars_i \circ \map(\hatS_k) \|^2}{\F_k} 
    + \alpha \left( \alpha (1+ \omega) -1 \right) \CPE{\|\Delta_{k+1,i}\|^2}{\F_k} \\ 
  & + \frac{2}{\alpha} L_i^2 \pas_{k+1}^2 \bigg(
    \frac{\omega}{n^2}\sum_{i=1}^n \CPE{\| \Delta_{k+1,i}\|^2}{\F_k} + 
    \|\mf(\hatS_k) \|^2+ \frac{\sigma^2}{n} \bigg)\eqsp .
\end{align*}
Since $\alpha (1+ \omega) -1 \le 0 $, using \Cref{hyp:variance:oracle}
and finally \Cref{prop:FedEM:Delta}, we get:
	\begin{align*}
          \PE\left[ \|V_{k+1,i} - \mf_i(\hatS_{k+1}) \|^2 \vert \F_{k}
          \right] & {\le}\left (1-\frac{\alpha}{2} \right ) \, \|
                    V_{k,i} -\mf_i(\hatS_{k}) \|^2 + 2 \alpha\sigma_i^2 \\ & +2 \pas_{k+1}^2
                                                                             \frac{L_i^2}{\alpha} 
                                                                             \frac{\omega}{n^2}\sum_{i=1}^n \| h_i(\hatS_k)- V_{k,i}\|^2 +
                                                                             2 \pas_{k+1}^2 \frac{L_i^2}{\alpha}  \| \mf(\hatS_k) \|^2 \\ &+
                                                                                                                                            2 \pas_{k+1}^2 \frac{L_i^2}{\alpha}  \frac{1+\omega}{n} \sigma^2 \eqsp.
	\end{align*}
	Overall, by averaging the previous inequality over all
        workers, we get:
	\begin{multline*}
          \PE[G_{k+1}|\F_k] \le \left (1-\frac{\alpha}{2} +
           2 \pas_{k+1}^2 \frac{L^2}{\alpha}  \frac{\omega}{n}\right )
          G_k
          + 2 \pas_{k+1}^2 \frac{L^2}{\alpha} \| \mf(\hatS_k) \|^2 \\
          + 2 \left (\alpha + \pas_{k+1}^2 \frac{L^2}{\alpha}
            \frac{1+\omega}{n}\right ) \sigma^2 \eqsp.
	\end{multline*}
\end{proof}

\subsection{Proof of \autoref{theo:dianaem}}
\label{app:proof-dianaem}
Equipped with the necessary results, we now provide the main proof of
\Cref{theo:dianaem}. We proceed in three steps, as follows. First, for
$k\geq 1$, we compute an upper bound on the average decrement
$\CPE{\lyap(\hatS_{k+1})}{\F_k} - \lyap(\hatS_{k})$ of the Lyapunov
function $\lyap$ (defined in \Cref{hyp:DS:lyap}).  Second, we
introduce the maximal value of the learning rate. Third and finally,
we deduce the result of \Cref{theo:dianaem} by computing the
expectation w.r.t. a randomly chosen termination time $K$ in
$[\kmax-1]$; in this step, we restrict the
  computations to the case the step sizes are constant ($\pas_{k+1} =
  \pas$ for any $k \geq 0$).

\paragraph{Step 1: Upper bound on the decrement.}
 Let $k \geq 0$; from \Cref{hyp:DS:lyap}, we have
\begin{align}
\lyap(\hatS_{k+1}) & \leq  \lyap(\hatS_k) + \pscal{\nabla \lyap(\hatS_k)}{
	\hatS_{k+1} - \hatS_k} + \frac{L_{\dot \lyap}}{2} \| \hatS_{k+1} -
\hatS_{k} \|^2  \nonumber \\
& \leq  \lyap(\hatS_k) -  \pas_{k+1} \, \pscal{B(\hatS_k) \, \mf(\hatS_k)}{
	H_{k+1}} + \frac{L_{\dot \lyap}}{2} \pas_{k+1}^2 \| H_{k+1} \|^2   \eqsp. \label{eq:FedEM:DL2}
\end{align}
Since $\hatS_k \in \F_k$, by \Cref{prop:fieldH} and \Cref{hyp:DS:lyap}
we have
\begin{equation}\label{key}
\CPE{\pscal{B(\hatS_k) \, \mf(\hatS_k)}{	H_{k+1}}}{\F_{k}}  =\pscal{B(\hatS_k) \, \mf(\hatS_k)}{	\mf(\hatS_k)} \geq v_{\min}  \|\mf(\hatS_k)\|^2.
\end{equation}

Hence, combining \eqref{eq:FedEM:DL2} and \eqref{key}, we have
\begin{align*}
&\CPE{\lyap(\hatS_{k+1})}{\F_k} \leq \lyap(\hatS_k) -\pas_{k+1} {v_{\min}} \|\mf(\hatS_k)\|^2 + \pas_{k+1}^2 \frac{L_{\dot
       \lyap}}{2} \CPE{\| H_{k+1} \|^2}{\F_k}\\ 
& \leq \lyap(\hatS_k) -\pas_{k+1} {v_{\min}} \|\mf(\hatS_k)\|^2 +
   \pas_{k+1}^2 \frac{L_{\dot \lyap}}{2} 
   \CPE{\| H_{k+1} - \CPE{H_{k+1}}{\F_k} \|^2}{\F_k} + \pas_{k+1}^2 \frac{L_{\dot \lyap}}{2} \| \mf(\hatS_k)
   \|^2 \\ 
& \leq \lyap(\hatS_k) -\pas_{k+1} {v_{\min}} \left (1-
   \pas_{k+1} \frac{L_{\dot \lyap}}{2v_{\min}} \right )
   \|\mf(\hatS_k)\|^2  + \pas_{k+1}^2 \frac{L_{\dot \lyap}}{2} \CPE{\| H_{k+1} - \CPE{H_{k+1}}{\F_k}\|^2}{\F_k} \eqsp.
\end{align*}
Applying~\Cref{prop:fieldH}, we obtain that
\begin{multline}
  \CPE{\lyap(\hatS_{k+1})}{\F_k} \leq \lyap(\hatS_k) -\pas_{k+1}
  {v_{\min}} \left (1- \pas_{k+1} \frac{L_{\dot \lyap}}{2v_{\min}}
  \right ) \|\mf(\hatS_k)\|^2 \\ + \pas_{k+1}^2 \frac{L_{\dot
      \lyap}}{2} \frac{\omega}{n} \left( \frac{1}{n} \sum_{i=1}^n
    \CPE{\| \Delta_{k+1,i}\|^2}{\F_k} \right) + \pas_{k+1}^2
  \frac{L_{\dot \lyap}}{2n} \sigma^2 \eqsp. \label{eq:upperbound1}
\end{multline}
Finally, using \Cref{prop:FedEM:Delta} and \eqref{eq:upperbound1}, we
get:
\begin{align}
  \PE[\lyap(\hatS_{k+1}) \vert \F_k] & \leq \lyap(\hatS_k) -\pas_{k+1}
                                       {v_{\min}} \left (1- \pas_{k+1} \frac{L_{\dot \lyap}}{2v_{\min}}
                                       \right ) \|\mf(\hatS_k)\|^2 \nonumber \\ & + \pas_{k+1}^2
                                                                                  \frac{L_{\dot \lyap}}{2} \frac{\omega}{n} G_k + \pas_{k+1}^2
                                                                                  \frac{L_{\dot \lyap}}{2n} (1+\omega) \sigma^2 \eqsp, \label{eq:expansion}
\end{align}
where
\[
G_k \eqdef \frac{1}{n}\sum_{i=1}^n \| V_{k,i} - \mf_i(\hatS_k)\|^2 \eqsp.
\]

\paragraph{Step 2: Maximal learning rate $\gamma_{k+1}$ when $\omega \neq 0$.}
From \Cref{prop:varcont2}, for any non-increasing positive sequence
$\{\pas_k, k \in [\kmax-1] \}$ such that
\[
\pas_{k+1}^2 \leq \frac{\alpha^2}{8 L^2} \frac{n}{\omega},
\]
and for any positive sequence $\{C_k, k \in [\kmax-1]\}$, it holds
\begin{multline}
  C_{k+1} \PE\left[G_{k+1} \vert \F_k \right] \leq C_{k+1} \left(
    1-\frac{\alpha}{4} \right) G_k \\ + C_{k+1} \pas_{k+1}^2
  \frac{2}{\alpha} L^2 \| \mf(\hatS_k) \| ^2 + 2 C_{k+1} \left( \alpha
    + \pas_{k+1}^2 \frac{L^2}{\alpha} \frac{1+\omega}{n}\right )
  \sigma^2 \eqsp. \label{eq:constantCk}
\end{multline}
Combining equations~\eqref{eq:expansion} and \eqref{eq:constantCk}, we thus have
\begin{align*}
\PE[\lyap(\hatS_{k+1}) \vert \F_k] & + C_{k+1} \PE\left[G_{k+1} \vert
  \F_k \right]  \leq \lyap(\hatS_k) + C_k G_k \\ & -\pas_{k+1}
   {v_{\min}} \left (1- \pas_{k+1} \frac{L_{\dot \lyap}}{2v_{\min}} -
   \frac{C_{k+1}}{v_{\min}} \pas_{k+1} \frac{2}{\alpha} L^2 \right )
   \|\mf(\hatS_k)\|^2 \\ & + \left(\pas_{k+1}^2 \frac{L_{\dot
       \lyap}}{2} \frac{\omega}{n} - C_k + C_{k+1} - C_{k+1}
   \frac{\alpha}{4} \right) G_k \\ & + \left\{2 \alpha C_{k+1} +
   \pas_{k+1}^2 \frac{(1+\omega)}{n} \left( \frac{L_{\dot \lyap}}{2} +
   2 C_{k+1} \frac{L^2}{\alpha} \right) \right\} \sigma^2 \eqsp.
  \end{align*}
We choose the sequence $\{C_k\}$ as follows: 
\[
C_{k} \eqdef \pas_{k}^2 \frac{2 L_{\dot \lyap}}{\alpha}
\frac{\omega}{n} \eqsp;
\]
the sequence satisfies $C_{k+1} \leq C_{k}$ (since $\pas_{k+1} \leq
\pas_k$) and $\pas_{k+1}^2 L_{\dot \lyap} \omega/(2n) \leq C_{k+1}
\alpha/4$.  By convention, $\pas_0 \in
\coint{\pas_1,+\infty}$. Therefore
\begin{align}
\PE[\lyap(\hatS_{k+1}) \vert \F_k] & + \pas_{k+1}^2 \frac{2 L_{\dot
    \lyap}}{\alpha} \frac{\omega}{n} \PE\left[G_{k+1} \vert \F_k
  \right] \leq \lyap(\hatS_k) + \pas_{k}^2 \frac{2 L_{\dot
    \lyap}}{\alpha} \frac{\omega}{n} G_k \\ & -\pas_{k+1} {v_{\min}}
\left (1- \pas_{k+1} \frac{L_{\dot \lyap}}{2v_{\min}} \left\{ 1 + 8
\pas_{k+1}^2 \frac{\omega}{\alpha^2 n} L^2 \right\}\right ) \|\mf(\hatS_k)\|^2
\\ & + 4 \pas_{k+1}^2 L_{\dot \lyap} \frac{\omega}{n} \left\{1 +
\frac{(1+\omega)}{8 \omega} \left( 1 + \pas_{k+1}^2 8
\frac{L^2}{\alpha^2} \frac{\omega}{n}\right) \right\} \sigma^2 \eqsp.\label{eq:GFbound}
\end{align}

\paragraph{Step 3: Computing the expectation.} Let us apply the expectations,  sum from $k=0$ to $k = \kmax-1$, and divide by $\kmax$.
This yields
\begin{align*}
  & \frac{v_{\min}}{\kmax} \sum_{k=0}^{\kmax-1} \pas_{k+1} \left (1-
    \pas_{k+1} \frac{L_{\dot \lyap}}{2v_{\min}} \left\{ 1 + 8 \pas_{k+1}^2
    \frac{\omega}{\alpha^2 n} L^2 \right\}\right ) \|\mf(\hatS_k)\|^2
  \\ & \leq \kmax^{-1} \left\{ \lyap(\hatS_{0}) + \pas_{0}^2 \frac{2
       L_{\dot \lyap}}{\alpha} \frac{\omega}{n} G_0 -
       \PE\left[\lyap(\hatS_{\kmax}) \right] - \pas_{\kmax}^2 \frac{2 L_{\dot
       \lyap}}{\alpha} \frac{\omega}{n} \PE\left[ G_{\kmax} \right]
       \right\} \\ & + 4 L_{\dot \lyap} \frac{\omega}{n} \frac{1}{\kmax
                     }\sum_{k=0}^{\kmax-1} \pas_{k+1}^2 \left\{1 + \frac{(1+\omega)}{8
                     \omega} \left( 1 + \pas_{k+1}^2 8 \frac{L^2}{\alpha^2}
                     \frac{\omega}{n}\right) \right\} \sigma^2 \eqsp.
\end{align*}
We now focus on the case when $\pas_{k+1} = \pas$ for any $k \geq
0$. Denote by $K$ a uniform random variable on $[\kmax-1]$,
independent of the path $\{\hatS_k, k \in [\kmax] \}$.  Since $\pas^2
\leq \alpha^2 n /(8 L^2 \omega)$, we have
\[
1 + 8 \pas^2 \frac{\omega}{\alpha^2 n} L^2  \leq 2 \eqsp.
\]
This yields
\begin{align}
  & v_{\min} \pas \left (1- \pas \frac{L_{\dot \lyap}}{v_{\min}}\right )
    \PE\left[  \|\mf(\hatS_K)\|^2 \right]  \nonumber\\ & \leq \kmax^{-1} \left\{ \lyap(\hatS_{0}) +
                                               \pas^2 \frac{2 L_{\dot \lyap}}{\alpha} \frac{\omega}{n} G_0 -
                                               \PE\left[\lyap(\hatS_{\kmax}) \right] - \pas^2 \frac{2 L_{\dot
                                               \lyap}}{\alpha} \frac{\omega}{n} \PE\left[ G_{\kmax} \right]
                                               \right\} \nonumber\\ & + 4 L_{\dot \lyap} \frac{\omega}{n} \pas^2 \left\{1 +
                                                             \frac{(1+\omega)}{4 \omega} \right\} \sigma^2 \eqsp. \label{eq:aux-maxLR}
\end{align}
Note that $4 (1 +(1+\omega)/(4 \omega)) = (5\omega+1)/\omega$.
\paragraph{Step 4.  Conclusion (when $\omega \neq 0$).} By
  choosing $V_{0,i} = \mf_i$ for any $i \in [n]^\star$, we have $G_0
  =0$. The roots of $\pas \mapsto \pas (1-\pas L_{\dot
    \lyap}/v_{\min})$ are $0$ and $v_{\min}/L_{\dot \lyap}$ and its
  maximum is reached at $v_{\min}/(2L_{\dot \lyap})$: this function is
  increasing on $\ocint{0, v_{\min}/(2L_{\dot \lyap})}$. We therefore
  choose $\pas \in \ocint{0,\pas_\max(\alpha)}$ where
\[
\pas_\max(\alpha) \eqdef \min \left( \frac{v_{\min}}{2 L_{\dot \lyap}};
  \frac{\alpha}{2 \sqrt{2} L} \frac{\sqrt{n}}{\sqrt{\omega}} \right)
\] 
Finally, since $\alpha \in \ocint{0, 1/(1+\omega)}$, we choose $\alpha = 1/(1+\omega)$. This yields
\[
\pas_\max \eqdef \min \left( \frac{v_{\min}}{2 L_{\dot \lyap}};
  \frac{1}{2 \sqrt{2} L} \frac{\sqrt{n}}{\sqrt{\omega}(1+\omega)} \right) \eqsp.
\]

\paragraph{Case $\omega = 0$.} From \eqref{eq:expansion}, applying the expectation we have
\[
\pas_{k+1} v_{\min} \left(1 - \pas_{k+1} \frac{L_{\dot \lyap}}{2
  v_{\min}} \right) \PE\left[ \| \mf(\hatS_k) \|^2\right] \leq
\PE\left[\lyap(\hatS_k)\right] - \PE\left[\lyap(\hatS_{k+1})\right] +
\pas_{k+1}^2 \frac{L_{\dot \lyap} \sigma^2}{2n} \eqsp.
\]
We now sum from $k=0$ to $k = \kmax-1$ and then divide by $\kmax$. In the case $\pas_{k+1}= \pas$, we have
\begin{equation} \label{eq:fedem:nocompression}
\pas v_{\min} \left(1 - \pas \frac{L_{\dot \lyap}}{2 v_{\min}} \right)
\PE\left[ \| \mf(\hatS_K) \|^2\right] \leq \kmax^{-1} \,
\left(\PE\left[\lyap(\hatS_0)\right] - \min \lyap \right) + \pas^2
\frac{L_{\dot \lyap} \sigma^2}{2n} \eqsp.
\end{equation}

\paragraph{Remark on the maximal learning rate.} The condition
$ \pas_{k+1} \le \frac{\alpha}{2 \sqrt{2} L}
\frac{\sqrt{n}}{\sqrt{\omega}} $ is used twice in the proof:
\begin{enumerate}[leftmargin=*]
	\item To ensure that $  \left (1-
	\pas_{k+1} \frac{L_{\dot \lyap}}{2v_{\min}} \left\{ 1 + 8 \pas_{k+1}^2
	\frac{\omega}{\alpha^2 n} L^2 \right\}\right ) \geq \left (1-
	\pas_{k+1} \frac{L_{\dot \lyap}}{v_{\min}}\right )$ in order to obtain \Cref{eq:aux-maxLR}. 
	\item To ensure that the process $ (G_k)_{k\geq0} $ is ``pseudo-contractive'' (i.e., satisfies a recursion of the form $ u_{k+1}\le \rho u_k + v_k $, with $ \rho<1 $) in \Cref{prop:varcont2}.
\end{enumerate}
A more detailed analysis can get rid of this condition (and thus the dependency $ \gamma = O_{\omega\to \infty} (\omega^{-3/2} )$, as we recall that $ \alpha^{1} \varpropto_{\omega\to \infty} \omega$) for the \textit{first point}. Indeed, we ultimately only require \begin{align}
\left (1-
\pas_{k+1} \frac{L_{\dot \lyap}}{2v_{\min}} \left\{ 1 + 8 \pas_{k+1}^2
\frac{\omega}{\alpha^2 n} L^2 \right\}\right ) \geq \frac{1}{2} \label{eq:aux-maxLR-2}
                                                                                                                                                                                                                                                                          \end{align}   to conclude the proof. This is for example  satisfied if $ \pas_{k+1} \frac{L_{\dot \lyap}}{2v_{\min}} \le \frac{1}{4} $ and $   8 \pas_{k+1}^3 \frac{L_{\dot \lyap}}{2v_{\min}} 
                                                                                                                                                                                                                                                                          \frac{\omega}{\alpha^2 n} L^2 \le \frac{1}{4}$. This approach results in a better asymptotic dependency of the maximal learning rate w.r.t. $ \omega $ to obtain \Cref{eq:aux-maxLR-2}:   $ \gamma = O_{\omega\to \infty} (\omega^{-1} )$.  \textit{However,} the condition $ 	\pas_{k+1} \le 
                                                                                                                                                                                                                                                                          \frac{\alpha}{2 \sqrt{2} L} \frac{\sqrt{n}}{\sqrt{\omega}}  $ seems to be \textit{necessary} to obtain the \textit{second point} and  \Cref{prop:varcont2}. The possibility of providing a similar result to  \Cref{prop:varcont2} without the $\omega^{-3/2}$ dependency, is an interesting open problem. 

\subsection{Proof of \Cref{cor:diana-em}}
\label{proof:coro:dianaem}
In \eqref{eq:theo:fedem}, the RHS is of the form
   $A/\pas + \pas B$ for some positive constants $A,B$: we have
   $A/\pas + \pas B \geq 2 \sqrt{AB}$ with equality reached with
   $\pas_\star \eqdef \sqrt{A/B}$.  Hence, we set
\[
\pas_\star \eqdef \frac{1}{\sigma}\left( \frac{n \left(\lyap(\hatS_0) - \min \lyap\right)}{L_{\dot \lyap} (1+5\omega)} \right)^{1/2} \frac{1}{\sqrt{\kmax}} \eqsp.
\]
If $\pas_\star \leq \pas_{\max}$, then let us apply
\eqref{eq:theo:fedem} with $\pas = \pas_\star$ which yields a RHS
given by $2 \sqrt{A/B}$ \ie\
\[
2 \sigma\left( \left(\lyap(\hatS_0) - \min \lyap\right) L_{\dot \lyap}
\frac{(1+5\omega)}{n} \right)^{1/2} \frac{1}{\sqrt{\kmax}} \eqsp.
\]
If $\pas_\star \geq \pas_\max$, we write
\[
\frac{A}{\pas_{\max}} + B \pas_{\max} \leq \frac{A}{\pas_{\max}}
+\frac{A}{\pas_{\max}} \frac{\pas_\max^2 B}{A}=\frac{A}{\pas_{\max}} +
\frac{A}{\pas_{\max}} \frac{\pas_\max^2}{\pas_\star^2} \leq 2 \frac{A}{\pas_{\max}} \eqsp.
\]
and the RHS is upper bounded by
\[
2 \frac{\lyap(\hatS_0) - \min \lyap}{ \pas_{\max} \kmax} \eqsp.
\]
Finally, in the LHS of \eqref{eq:theo:fedem}, we have
\[
1 - \pas \frac{L_{\dot \lyap}}{v_{\min}} \geq 1 - \pas_{\max}
\frac{L_{\dot \lyap}}{v_{\min}} \geq 1 - \frac{v_{\min}}{2 L_{\dot
    \lyap}} \frac{L_{\dot \lyap}}{v_{\min}} = \frac{1}{2}\eqsp.
\]
This concludes the proof.

\section{Partial Participation case}
\label{app:PP}
In this section, we generalize the result of \Cref{theo:dianaem} to
the \textit{partial participation case}. This extra scheme could be
incorporated into the main proof, but we choose to present it
separately to improve the readability of the main proof in
\Cref{app:proof-fedem}. We first provide an
  equivalent description of \autoref{algo:dianaem-main} in
  \Cref{sec:app:equivPP}; \autoref{algo:dianaem-main:equiv} will be
  used throughout this section. Then, we introduce a new family of
filtrations. In \Cref{sec:app:prop:PP}, we first
  establish preliminary results and then give the proof of
  \Cref{prop:PP} in \Cref{sec:app:propPP}.

The assumptions \Cref{hyp:model} to \Cref{hyp:Tmap} hold throughout
this section.

\subsection{An equivalent algorithm}
\label{sec:app:equivPP}
In this Section, we describe an equivalent algorithm, that outputs the same result as Algorithm~\ref{algo:dianaem-main}, and for which the analysis is conducted.

\begin{algorithm}[H]
  \caption{\FEDEM~with partial
    participation \label{algo:dianaem-main:equiv}} \KwData{ $\kmax \in
    \nset^\star$; for $i \in [n]^\star$, $V_{0,i} \in \rset^q$;
    $\hatS_0 \in \rset^q$; a positive sequence $\{\pas_{k+1}, k \in
         [\kmax-1]\}$; $\alpha>0$; $p \in \ooint{0,1}$.}  \KwResult{
    The \texttt{\FEDEM-PP} sequence: $\{\hatS_{k}, k \in [\kmax]\}$}
  Set $V_0 = n^{-1} \sum_{i=1}^n V_{0,i}$ \; \For{$k=0, \ldots,
    \kmax-1$}{ \For{$i=1, \ldots, n$} {{\em (worker $\# i$)}\; Sample
      $\Smem_{k+1,i}$, an approximation of $\bars_i \circ
      \map(\hatS_k)$ \; Set $\Delta_{k+1,i} = \Smem_{k+1,i} - V_{k,i}
      - \hatS_k$ \; Sample a Bernoulli r.v. $B_{k+1,i}$ with success
      probability $p$ \; Set $V_{k+1,i} = V_{k,i} + \alpha \,
      B_{k+1,i} \Q(\Delta_{k+1,i})$.  \; Send $ B_{k+1,i}
      \Q(\Delta_{k+1,i})$ to the central server \;} {\em (the central
      server)} \; Set $H_{k+1} = V_k + (np)^{-1} \sum_{i=1}^n
    B_{k+1,i} \Q(\Delta_{k+1,i})$ \; Set $\hatS_{k+1} = \hatS_k +
    \pas_{k+1} H_{k+1}$ \; Set $V_{k+1} = V_k + \alpha n ^{-1} \sum_{i
      =1}^n B_{k+1,i} \Q(\Delta_{k+1,i})$ \; Send $\hatS_{k+1}$ and
    $\map (\hatS_{k+1})$ to the $n$ workers}
\end{algorithm}

\subsection{Notations}\label{app:sub:PPhyp}
\label{sec:FEDEMwithPP:tribu}
Let us introduce a new sequence of filtrations.
For any $i \in [n]^\star$, we
set
$$\F_{0,i} = \F_{0,i}^+ \eqdef \sigma\left(\hatS_0; V_{0,i} \right)
\text{ and } \qquad \F_{0} \eqdef \bigvee_{i=1}^n \F_{0,i} \eqsp.$$
Then, for all $k\geq 0$,
\begin{enumerate}[(i)]
\item$ \F_{k+1/3,i} \eqdef \F_{k,i}^+ \vee \sigma\left( \Smem_{k+1,i}
  \right)$,
  \item$ \F_{k+2/3,i} \eqdef \F_{k+1/3,i} \vee \sigma\left(
  \Q(\Delta_{k+1,i})    \right)$,
  \item $\F_{k+1,i} \eqdef \F_{k+2/3,i} \vee \sigma\left( B_{k+1,i}
    \right)$,
  \item $\F_{k+1} \eqdef
  \bigvee_{i=1}^n \F_{k+1,i}$,
  \item $\F_{k+1,i}^+ \eqdef
  \F_{k+1,i} \vee \F_{k+1}$.
\end{enumerate}
Note that, with these notations, for $k \geq 0$ and $i \in [n]^\star$,
the random variables of the \FEDEM~sequence defined in
\autoref{algo:dianaem-main:equiv} belong to the filtrations defined
above as follows:
\begin{enumerate}[(i)]
\item $\hatS_{k} \in \F_{k,i}^+$, $\hatS_k \in \F_k$,
\item  $\Smem_{k+1,i} ,\Delta_{k+1,i} \in \F_{k+1/3,i}$,
\item  $V_{k+1,i} \in \F_{k+1,i}$,
\item $\hatS_{k+1} , H_{k+1}, V_{k+1} \in \F_{k+1}$.
\end{enumerate}
Note also that we have the following inclusions for filtrations: $\F_k
\subset \F_{k,i}^+ \subset \F_{k+1/3,i} \subset \F_{k+2/3,i} \subset
\F_{k+1,i} \subset \F_{k+1}$ for all $i \in [n]^\star$.

\subsection{Preliminary results}
\label{sec:app:prop:PP}
In this section, we extend \Cref{prop:sumV}, \Cref{prop:fieldH} (that
controls the random field $H_{k+1}$) and \Cref{prop:varcont2} (that
controls the memory term $V_{k,i}$). We start by verifying the simple
following proposition, that ensures that the global variable $ V_k $
corresponds to the mean of the local control variables
$ (V_{k,i})_{i\in [n]^{*}} $.

\begin{proposition} \label{prop:meanV:PP} 
  For any $k \in [\kmax]$,
  \[
V_{k} = \frac{1}{n} \sum_{i=1}^n V_{k,i} \eqsp.
  \]
  \end{proposition}
\begin{proof}
  By definition of $V_0$, the property holds true when $k=0$. Assume this holds true for $k \in [\kmax-1]$. We write
  \begin{align*}
    V_{k+1} &= V_k + \frac{\alpha}{n} \sum_{i =1}^n B_{k+1,i} \,
    \Q(\Delta_{k+1,i}) \\ & = \frac{1}{n} \sum_{i=1}^n V_{k,i} +
    \frac{1}{n} \sum_{i =1}^n \left( V_{k+1,i} - V_{k,i} \right) \\ &
    = \frac{1}{n} \sum_{i=1}^n V_{k+1,i} \eqsp.
    \end{align*}
  This concludes the induction.
\end{proof}
We now prove that the unbiased character of $ H_{k} $ is preserved,
and we provide a new control on its second order
moment. \Cref{fieldH:PP} is \Cref{prop:fieldH} with $ \omega $
replaced with $\omega_p$. When $p=1$, \Cref{fieldH:PP} and
\Cref{prop:fieldH} are the same.
\begin{proposition}\label{fieldH:PP} Assume \Cref{hyp:var:quantif}, \Cref{hyp:variance:oracle} and \Cref{hyp:PP}. Set $\sigma^2 \eqdef n^{-1} \sum_{i=1}^n \sigma_i^2$. For any $k \in [\kmax-1]$, we have 
	\[
	\PE\left[H_{k+1} \vert \F_k \right] = \mf(\hatS_k) \eqsp,
	\]
	and
	\begin{align*}
	\PE\left[ \| H_{k+1} - \PE\left[H_{k+1} \vert \F_k \right]
          \|^2 \vert \F_k \right] & \le \frac{\omega_p}{n}
        \frac{1}{n}\sum_{i=1}^n \PE\left[ \| \Delta_{k+1,i}\|^2 \vert
          \F_k \right] + \frac{\sigma^2}{n} \eqsp,
	\end{align*}
        where
        \begin{equation}\label{eq:PP:defomegap}
          \omega_p \eqdef \frac{1-p}{p} (1+\omega) +\omega \eqsp.
          \end{equation}
\end{proposition}

\begin{proof} 
  Let $k \in [\kmax-1]$. By definition, we have
  \[
H_{k+1} = V_k + \frac{1}{np} \sum_{i=1}^n B_{k+1,i} \Q(\Delta_{k+1,i})
\]
where the Bernoulli random variables $\{B_{k+1,i}, i \in [n]^\star \}$
are independent with the same success probability $p$. By definition
of the filtrations, we have $B_{k+1,i} \in\F_{k+1,i}$,
$\Q(\Delta_{k+1,i}) \in \F_{k+2/3,i}$, $V_k \in \F_k$ and
$\Delta_{k+1,i} \in \F_{k+1/3,i}$; and the inclusions $\F_k \subset
\F_{k+1/3,i} \subset \F_{k+2/3,i} \subset \F_{k+1,i}$. Therefore,
\begin{align*}
	 \CPE{H_{k+1} }{ \F_k} & = V_k + \frac{1}{np} \sum_{i=1}^n
         \CPE{ \CPE{ B_{k+1,i}}{ \F_{k+2/3,i}}
           \Q(\Delta_{k+1,i})}{\F_{k}}\nonumber \\ & = V_k +
         \frac{1}{n} \sum_{i=1}^n \CPE{\CPE{
             \Q(\Delta_{k+1,i})}{\F_{k+1/3,i}}}{\F_k} = V_k +
         \frac{1}{n} \sum_{i=1}^n \CPE{ \Delta_{k+1,i}}{\F_k}
         \nonumber \\ & = V_k + \frac{1}{n} \sum_{i=1}^n \left(
         \CPE{\Smem_{k+1,i}}{\F_k} - \hatS_k - V_{k,i} \right) \\ & =
         \frac{1}{n} \sum_{i=1}^n \mf_i(\hatS_k) = \mf(\hatS_k) \eqsp,
         \nonumber
	\end{align*}
where we used $\CPE{ B_{k+1,i}}{ \F_{k+2/3,i}} =p$ (see
\Cref{hyp:PP}), \Cref{hyp:var:quantif}, \Cref{hyp:variance:oracle} and
\Cref{prop:meanV:PP}.  This concludes the proof of the first statement
of \Cref{fieldH:PP}.  For the second point, we write
\begin{align*}
  H_{k+1} - \mf(\hatS_k) &= \frac{1}{n}\sum_{i=1}^n \Xi_{k+1,i}
  \\ \Xi_{k+1,i} & \eqdef \Smem_{k+1,i} - \CPE{
    \Smem_{k+1,i}}{\F_{k,i}^+} \\ & + \Q(\Delta_{k+1,i}) - \CPE{
    \Q(\Delta_{k+1,i})}{\F_{k+1/3,i}} \\ & + \frac{1}{p} \left(
  B_{k+1,i} - \CPE{B_{k+1,i}}{\F_{k+2/3,i}} \right) \Q(\Delta_{k+1,i})
  \eqsp;
  \end{align*}
note indeed that $\mf_i(\hatS_k) = \CPE{\Smem_{k+1,i} }{\F_{k,i}^+} -
\hatS_k$, $\CPE{ \Q(\Delta_{k+1,i})}{\F_{k+1/3,i}} = \Delta_{k+1,i}$,
$\Delta_{k+1,i} = V_{k,i} + \Smem_{k+1,i} - \hatS_k$, $V_k = n^{-1}
\sum_{i=1}^n V_{k,i}$ and $p = \CPE{B_{k+1,i}}{\F_{k+2/3,i}}$. Write
$H_{k+1} - \mf(\hatS_k) = \frac{1}{n}\sum_{i=1}^n \Xi_{k+1,i}$.  Since
the workers are independent, we have
\[
\CPE{\| H_{k+1} - \mf(\hatS_k)\|^2}{\F_k} = \frac{1}{n^2} \sum_{i=1}^n
\CPE{\| \Xi_{k+1,i}\|^2}{\F_k} \eqsp.
\]
Fix $i \in [n]^\star$. $\Xi_{k+1,i}$ is the sum of three terms
$\sum_{\ell=1}^3 \Xi_{k+1,i,\ell}$ and observe that for any $\ell \neq
\ell'$ we have
\[
\CPE{\pscal{\Xi_{k+1,i,\ell}}{\Xi_{k+1,i,\ell'}}}{\F_k} = 0 \eqsp.
\]
Therefore $\CPE{\| \Xi_{k+1,i}\|^2}{\F_k} = \sum_{\ell=1}^3 \CPE{\|
  \Xi_{k+1,i,\ell}\|^2}{\F_k}$. We have by \Cref{hyp:variance:oracle}
\[
\CPE{\| \Smem_{k+1,i} - \CPE{ \Smem_{k+1,i}}{\F_{k,i}^+} \|^2}{\F_k}
\leq \sigma_i^2 \eqsp;
\]
by \Cref{hyp:var:quantif},
\[
\CPE{\|\Q(\Delta_{k+1,i}) - \CPE{
    \Q(\Delta_{k+1,i})}{\F_{k+1/3,i}}\|^2}{\F_k} \leq \omega
\CPE{\|\Delta_{k+1,i}\|^2}{\F_k} \eqsp;
\]
and by \Cref{hyp:var:quantif} and \Cref{hyp:PP}
\begin{align*}
& \CPE{\frac{1}{p^2} \left( B_{k+1,i} - \CPE{B_{k+1,i}}{\F_{k+2/3,i}}
    \right)^2 \| \Q(\Delta_{k+1,i}) \|^2}{\F_k} \\ & \leq
  \frac{1-p}{p} \CPE{\|\Q(\Delta_{k+1,i})\|^2}{\F_k} \\ & \leq
  \frac{1-p}{p} (1+\omega) \CPE{\|\Delta_{k+1,i}\|^2}{\F_k} \eqsp.
\end{align*}
This concludes the proof.
\end{proof}
\begin{proposition} \label{prop:Delta:PP} Assume
  \Cref{hyp:variance:oracle} and set
  $\sigma^2 \eqdef n^{-1} \sum_{i=1}^n \sigma_i^2$.  For any
  $k \in [\kmax-1]$,
\[
\frac{1}{n}\sum_{i=1}^n \CPE{\|\Delta_{k+1,i}\|^2}{\F_k} \leq
\frac{1}{n}\sum_{i=1}^n \| V_{k,i} - \mf_i(\hatS_k) \|^2 + \sigma^2
\eqsp.
\]
\end{proposition}
The proof is on the same lines as the proof of \Cref{prop:FedEM:Delta} and is omitted.

 \Cref{prop:varcont-PP2} extends \Cref{prop:varcont2}: the result is
 similar but with $\alpha$ replaced with $\alpha
 p$ and $ \omega $ by $ \omega_p $. 
\begin{proposition}	\label{prop:varcont-PP2}
	Assume \Cref{hyp:lipschitz}, \Cref{hyp:var:quantif},
        \Cref{hyp:variance:oracle} and \Cref{hyp:PP}; set $L^2 \eqdef
        n^{-1} \sum_{i=1}^n L_i^2$ and $\sigma^2 \eqdef n^{-1}
        \sum_{i=1}^n \sigma_i^2$. Choose $\alpha \in
        \ocint{0,1/(1+\omega)}$.  For any $k \geq 0$, define
	\[
	G_k \eqdef \frac{1}{n}\sum_{i=1}^n \| V_{k,i} - \mf_i(\hatS_k)
        \|^2 \eqsp.
	\]
	We have, for any $k \in [\kmax -1]$

	\begin{multline*}
	\CPE{G_{k+1}}{\F_k} \le \left (1-\frac{\alpha p }{2} + 2
        \pas_{k+1}^2 \frac{L^2}{\alpha p } \frac{\omega_p}{n} \right)
        G_k + 2 \pas_{k+1}^2 \frac{L^2}{\alpha p } \| \mf(\hatS_k) \|
        ^2 \\ + 2 \left( \alpha p+ \pas_{k+1}^2 \frac{L^2}{\alpha p}
        \frac{\omega_p}{n} \right ) \sigma^2 \eqsp,
	\end{multline*}
        where $\omega_p$ is defined in \Cref{fieldH:PP}.  
\end{proposition}
\begin{proof}
  Let $i \in [n]^\star$.  We follow the same line of the proof as
  \Cref{prop:varcont2}: for any $\beta>0$, using that $\|a+b\|^2 \leq
  (1+\beta^2) \|a\|^2 + (1+\beta^{-2}) \|b\|^2$, we have
	\begin{align*}
	& \CPE{ \|V_{k+1,i} - \mf_i(\hatS_{k+1}) \|^2 }{ \F_{k} } \\ &
          \le (1+\beta^{-2}) \CPE{\|V_{k+1,i} - \mf_i(\hatS_{k}) \|^2
          }{ \F_{k} } + (1+\beta^2) \CPE{\|\mf_i(\hatS_{k}) -
            \mf_i(\hatS_{k+1}) \|^2 }{\F_{k}}
          \\ &\overset{\Cref{hyp:lipschitz}}{\le} (1+\beta^{-2}) \CPE{
            \|V_{k+1,i} - \mf_i(\hatS_{k}) \|^2 }{\F_{k}} +
          (1+\beta^2) L_i^2 \pas_{k+1}^2 \CPE{\| H_{k+1} \| ^2 }{
            \F_{k} }\eqsp.
	\end{align*}
We then provide a control for $ \CPE{ \|V_{k+1,i} - \mf_i(\hatS_{k})
  \|^2 }{\F_{k}}$. Recall that:
	\begin{align*}
V_{k+1,i} &= V_{k,i} + \alpha\,  B_{k+1,i} \Q(\Delta_{k+1;i}).
        \end{align*}
We write $f(B_{k+1,i}) = f(1) \1_{B_{k+1,i} =1} + f(0)
\1_{B_{k+1,i}=0}$ for any measurable positive function $f$; and then
use $\CPE{\1_{B_{k+1,i}}}{\F_{k+2/3,i}} = p$ (see \Cref{hyp:PP}),
$\Q(\Delta_{k+1,i}), \hatS_k, V_{k,i} \in \F_{k+2/3,i}$ . We get
\begin{align*}
	&\PE\left[ \|V_{k+1,i} - \mf_i(\hatS_{k}) \|^2 \vert \F_{k}
    \right] \\ &= p \CPE{ \|V_{k,i} -
    \mf_i(\hatS_{k})-\alpha\Q(\Delta_{k+1,i}) \|^2}{ \F_{k} } + (1-p)
  \|V_{k,i} - \mf_i(\hatS_{k}) \|^2 \\ &\overset{
    \eqref{eq:tool127}}{=} p (1-\alpha) \, \| V_{k,i}
  -\mf_i(\hatS_{k}) \|^2 + \alpha p \,\CPE{ \| \Smem_{k+1,i} - \hatS_k
    -\mf_i(\hatS_{k}) \|^2 }{ \F_k} \\ & + \alpha p \left( \alpha (1+
  \omega) -1 \right ) \CPE{\|\Delta_{k+1,i}\|^2}{ \F_k} + (1-p) \, \|
  V_{k,i} -\mf_i(\hatS_{k}) \|^2 \\ &= (1-\alpha p ) \, \| V_{k,i}
  -\mf_i(\hatS_{k}) \|^2 \\ & + \alpha p \, \CPE{ \| \Smem_{k+1,i} -
    \hatS_k -\mf_i(\hatS_{k}) \|^2 }{ \F_k} + \alpha p \left( \alpha
  (1+ \omega) -1 \right ) \CPE{\|\Delta_{k+1,i}\|^2}{ \F_k} \eqsp .
\end{align*}
The end of the proof is identical to the proof of
\Cref{prop:varcont2}: we choose $\beta_p >0$ such that $\beta_p^{-2} =
1$ if $\alpha p \geq 2/3$ and $\beta_p^{-2}= \frac{\alpha
  p}{2(1-\alpha p)} $ if $\alpha p \leq 2/3$. We have
\[
(1-\alpha {p })(1+\beta_p^{-2}) \le 1-\frac{\alpha p }{2} \eqsp,
\qquad (1+\beta_p^2) \le \frac{2}{\alpha p } \eqsp, \qquad 1 \leq
1+\beta_p^{-2} \le 2 \eqsp;
\]
and this yields
\begin{align*}
& \CPE{\|V_{k+1,i} - \mf_i(\hatS_{k+1}) \|^2 }{ \F_{k} } {\le} \left
  (1-\frac{\alpha p}{2} \right ) \, \| V_{k,i} -\mf_i(\hatS_{k}) \|^2
  \\ & + 2 \alpha p \, \CPE{ \| \Smem_{k+1,i} - \bars_i \circ
    \map(\hatS_k) \|^2 }{\F_k }+ \alpha p \left( \alpha (1+ \omega) -1
  \right ) \CPE{\|\Delta_{k+1,i}\|^2 }{\F_k} \\ & + \frac{2
  }{\alpha p} L_i^2 \pas_{k+1}^2 \CPE{\|H_{k+1}\|^2}{\F_k} \eqsp.
  \end{align*}
By definition of the conditional expectation and \Cref{fieldH:PP} we
have
\begin{align*}
  \CPE{\|H_{k+1}\|^2}{\F_k} & = \| \CPE{H_{k+1}}{\F_k}\|^2 +
  \CPE{\|H_{k+1}- \CPE{H_{k+1}}{\F_k}\|^2}{\F_k} \\ & =
  \|\mf(\hatS_k)\|^2 + \CPE{\|H_{k+1}-\mf(\hatS_k)\|^2}{\F_k} \eqsp.
  \end{align*}
Since $ \left( \alpha (1+ \omega) -1 \right ) \le 0 $, using
\Cref{hyp:variance:oracle} and \Cref{fieldH:PP} again, we get:
\begin{align*}
\CPE{ G_{k+1}}{ \F_{k} }& {\le}\left (1-\frac{\alpha p }{2} \right )
\, G_k + 2 \alpha p \sigma^2 + \frac{2}{\alpha p} L^2 \pas_{k+1}^2
\frac{1}{n}\left( \sigma^2 + \omega_p \frac{1}{n}\sum_{i=1}^n
\CPE{\|\Delta_{k+1,i}\|^2}{\F_k} \right) \eqsp.
\end{align*}
Finally, from \Cref{prop:Delta:PP},
\[
\frac{1}{n}\sum_{i=1}^n \CPE{\|\Delta_{k+1,i}\|^2}{\F_k} \leq G_k +
\sigma^2 \eqsp.
\]
This concludes the proof.

\end{proof}

\subsection{Proof of \Cref{prop:PP}} 
\label{sec:app:propPP}
Throughout this proof, set
\[
\omega_p \eqdef \frac{1-p}{p} (1+\omega) +\omega \eqsp.
\]

\paragraph{Step 1: Upper bound on the decrement.}
 Let $k \geq 0$. Following the same lines as in the proof of \Cref{theo:dianaem}, we have 
\begin{align*}
&\CPE{\lyap(\hatS_{k+1})}{\F_k}  \\
& \leq \lyap(\hatS_k) -\pas_{k+1} {v_{\min}} \left (1-
   \pas_{k+1} \frac{L_{\dot \lyap}}{2v_{\min}} \right )
   \|\mf(\hatS_k)\|^2  + \pas_{k+1}^2 \frac{L_{\dot \lyap}}{2} \CPE{\| H_{k+1} - \CPE{H_{k+1}}{\F_k}\|^2}{\F_k} \eqsp.
\end{align*}
Applying~\Cref{fieldH:PP} and \Cref{prop:Delta:PP}, we obtain that
\begin{multline}
  \CPE{\lyap(\hatS_{k+1})}{\F_k} \leq \lyap(\hatS_k) -\pas_{k+1}
  {v_{\min}} \left (1- \pas_{k+1} \frac{L_{\dot \lyap}}{2v_{\min}}
  \right ) \|\mf(\hatS_k)\|^2  \\  + \pas_{k+1}^2
  \frac{L_{\dot \lyap}}{2} \frac{\omega_p}{n} G_k + \pas_{k+1}^2
  \frac{L_{\dot \lyap}}{2n} (1+\omega_p) \sigma^2
  \eqsp, \label{eq:expansion:PP}
\end{multline}
where
\[
G_k \eqdef \frac{1}{n}\sum_{i=1}^n \| V_{k,i} - \mf_i(\hatS_k)\|^2
\eqsp.
\]

\paragraph{Step 2: Maximal learning rate $\gamma_{k+1}$ when $\omega \neq 0$.}
From \Cref{prop:varcont2}, for any non-increasing positive sequence
$\{\pas_k, k \in [\kmax-1] \}$ such that
\[
\pas_{k+1}^2 \leq \frac{\alpha^2 p^2}{8 L^2} \frac{n}{\omega_p},
\]
and for any positive sequence $\{C_k, k \in [\kmax-1]\}$, it holds
\begin{multline}
  C_{k+1} \PE\left[G_{k+1} \vert \F_k \right] \leq C_{k+1} \left(
    1-\frac{\alpha p}{4} \right) G_k \\ + C_{k+1} \pas_{k+1}^2
  \frac{2}{\alpha p} L^2 \| \mf(\hatS_k) \| ^2 + 2 C_{k+1} \left( \alpha p
    + \pas_{k+1}^2 \frac{L^2}{\alpha p} \frac{1+\omega_p}{n}\right )
  \sigma^2 \eqsp. \label{eq:constantCk:PP}
\end{multline}
Combining equations~\eqref{eq:expansion:PP} and \eqref{eq:constantCk:PP}, we thus have
\begin{align*}
\PE[\lyap(\hatS_{k+1}) \vert \F_k] & + C_{k+1} \PE\left[G_{k+1} \vert
  \F_k \right]  \leq \lyap(\hatS_k) + C_k G_k \\ & -\pas_{k+1}
   {v_{\min}} \left (1- \pas_{k+1} \frac{L_{\dot \lyap}}{2v_{\min}} -
   \frac{C_{k+1}}{v_{\min}} \pas_{k+1} \frac{2}{\alpha p} L^2 \right )
   \|\mf(\hatS_k)\|^2 \\ & + \left(\pas_{k+1}^2 \frac{L_{\dot
       \lyap}}{2} \frac{\omega_p}{n} - C_k + C_{k+1} - C_{k+1}
   \frac{\alpha p}{4} \right) G_k \\ & + \left\{2 \alpha  p  C_{k+1} +
   \pas_{k+1}^2 \frac{(1+\omega_p)}{n} \left( \frac{L_{\dot \lyap}}{2} +
   2 C_{k+1} \frac{L^2}{\alpha p} \right) \right\} \sigma^2 \eqsp.
  \end{align*}
We choose the sequence $\{C_k\}$ as follows:
\[
C_{k} \eqdef \pas_{k}^2 \frac{2 L_{\dot \lyap}}{\alpha p}
\frac{\omega_p}{n} \eqsp;
\]
the sequence satisfies $C_{k+1} \leq C_{k}$ (since $\pas_{k+1} \leq
\pas_k$) and $\pas_{k+1}^2 L_{\dot \lyap} \omega_p/(2n) \leq C_{k+1}
\alpha p/4$.  By convention, $\pas_0 \in \coint{\pas_1,+\infty}$.
Therefore
\begin{align*}
\PE[\lyap(\hatS_{k+1}) \vert \F_k] & + \pas_{k+1}^2 \frac{2 L_{\dot
    \lyap}}{\alpha p} \frac{\omega_p}{n} \PE\left[G_{k+1} \vert \F_k
  \right] \leq \lyap(\hatS_k) + \pas_{k}^2 \frac{2 L_{\dot
    \lyap}}{\alpha p} \frac{\omega_p}{n} G_k \\ & -\pas_{k+1} {v_{\min}}
\left (1- \pas_{k+1} \frac{L_{\dot \lyap}}{2v_{\min}} \left\{ 1 + 8
\pas_{k+1}^2 \frac{\omega_p}{\alpha^2 p^2 n} L^2 \right\}\right ) \|\mf(\hatS_k)\|^2
\\ & + 4 \pas_{k+1}^2 L_{\dot \lyap} \frac{\omega_p}{n} \left\{1 +
\frac{(1+\omega_p)}{8 \omega_p} \left( 1 + \pas_{k+1}^2 8
\frac{L^2}{\alpha^2 p^2} \frac{\omega_p}{n}\right) \right\} \sigma^2 \eqsp.\label{eq:GFbound:PP}
\end{align*}

\paragraph{Step 3: Computing the expectation.} Let us apply the expectations,  sum from $k=0$ to $k = \kmax-1$, and divide by $\kmax$.
This yields
\begin{align*}
  & \frac{v_{\min}}{\kmax} \sum_{k=0}^{\kmax-1} \pas_{k+1} \left (1-
    \pas_{k+1} \frac{L_{\dot \lyap}}{2v_{\min}} \left\{ 1 + 8 \pas_{k+1}^2
    \frac{\omega_p}{\alpha^2 p^2 n} L^2 \right\}\right ) \|\mf(\hatS_k)\|^2
  \\ & \leq \kmax^{-1} \left\{ \lyap(\hatS_{0}) + \pas_{0}^2 \frac{2
       L_{\dot \lyap}}{\alpha} \frac{\omega_p}{n} G_0 -
       \PE\left[\lyap(\hatS_{\kmax}) \right] - \pas_{\kmax}^2 \frac{2 L_{\dot
       \lyap}}{\alpha p} \frac{\omega_p}{n} \PE\left[ G_{\kmax} \right]
       \right\} \\ & + 4 L_{\dot \lyap} \frac{\omega_p}{n} \frac{1}{\kmax
                     }\sum_{k=0}^{\kmax-1} \pas_{k+1}^2 \left\{1 + \frac{(1+\omega_p)}{8
                     \omega} \left( 1 + \pas_{k+1}^2 8 \frac{L^2}{\alpha^2 p^2}
                     \frac{\omega_p}{n}\right) \right\} \sigma^2 \eqsp.
\end{align*}
We now focus on the case when $\pas_{k+1} = \pas$ for any $k \geq
0$. Denote by $K$ a uniform random variable on $[\kmax-1]$,
independent of the path $\{\hatS_k, k \in [\kmax] \}$.  Since $\pas^2
\leq \alpha^2 p^2 n /(8 L^2 \omega_p)$, we have
\[
1 + 8 \pas^2 \frac{\omega_p}{\alpha^2 p^2 n} L^2  \leq 2 \eqsp.
\]
This yields
\begin{align*}
  & v_{\min} \pas \left (1- \pas \frac{L_{\dot \lyap}}{v_{\min}}\right )
    \PE\left[  \|\mf(\hatS_K)\|^2 \right] \\ & \leq \kmax^{-1} \left\{ \lyap(\hatS_{0}) +
                                               \pas^2 \frac{2 L_{\dot \lyap}}{\alpha p} \frac{\omega_p}{n} G_0 -
                                               \PE\left[\lyap(\hatS_{\kmax}) \right] - \pas^2 \frac{2 L_{\dot
                                               \lyap}}{\alpha p} \frac{\omega_p}{n} \PE\left[ G_{\kmax} \right]
                                               \right\} \\ & + 4 L_{\dot \lyap} \frac{\omega_p}{n} \pas^2 \left\{1 +
                                                             \frac{(1+\omega_p)}{4 \omega_p} \right\} \sigma^2 \eqsp.
\end{align*}
Note that $4 (1 +(1+\omega_p)/(4 \omega_p)) = (5\omega_p+1)/\omega_p$.
\paragraph{Step 4.  Conclusion (when $\omega \neq 0$)} By
  choosing $V_{0,i} = \mf_i$ for any $i \in [n]^\star$, we have $G_0
  =0$. The roots of $\pas \mapsto \pas (1-\pas L_{\dot
    \lyap}/v_{\min})$ are $0$ and $v_{\min}/L_{\dot \lyap}$ and its
  maximum is reached at $v_{\min}/(2L_{\dot \lyap})$: this function is
  increasing on $\ocint{0, v_{\min}/(2L_{\dot \lyap})}$. We therefore
  choose $\pas \in \ocint{0,\pas_\max(\alpha)}$ where
\[
\pas_\max(\alpha) \eqdef \min \left( \frac{v_{\min}}{2 L_{\dot \lyap}};
  \frac{\alpha p}{2 \sqrt{2} L} \frac{\sqrt{n}}{\sqrt{\omega_p}} \right)
\] 
Finally, since $\alpha \in \ocint{0, 1/(1+\omega)}$, we choose $\alpha = 1/(1+\omega)$. This yields
\[
\pas_\max \eqdef \min \left( \frac{v_{\min}}{2 L_{\dot \lyap}};
  \frac{p}{2 \sqrt{2} L} \frac{\sqrt{n}}{\sqrt{\omega_p}(1+\omega)} \right) \eqsp.
\]

\section{Convergence Analysis of \VRFEDEM~}
\label{app:VRFEDEM}
\label{sec:proof:DS}
The assumptions \Cref{hyp:model} to \Cref{hyp:Tmap} hold throughout
this section. We will use the notations
  \begin{equation}
  \label{eq:def-L_i}
  L_i^2 \eqdef m^{-1}
\sum_{j=1}^m L_{ij}^2\eqsp,  \qquad L^2 \eqdef n^{-1}
\sum_{i=1}^n L_{i}^2 \eqsp,
  \end{equation}
  where $L_{ij}$ is defined in \Cref{hyp:DS:lipschitz},
and 
\[
\mf_i(s) \eqdef \frac{1}{m} \sum_{j=1}^m \bars_{ij}\circ \map(s) -s
\eqsp, \qquad \mf(s) \eqdef \frac{1}{n} \sum_{i=1}^n \mf_i(s) \eqsp.
\]
\subsection{Notations and elementary result}

Let us define the following filtrations: for any $i \in [n]^\star$ and $t
\in [\kouter]^\star$, $k \in [\kmax-1]$, set
\begin{align*}
  & \F_{1,0,i} = \F_{1,0,i}^+ \eqdef \sigma\left(\hatS_\init; V_{1,0,i} \right)
    \eqsp, \qquad \F_{1,0} \eqdef \bigvee_{i=1}^n \F_{1,0,i} \eqsp, \\ &
                                                                         \F_{t,k+1/2,i} \eqdef \F_{t,k,i}^+ \vee \sigma\left( \batch_{t,k+1,i}
                                                                         \right) \eqsp, \qquad \F_{t,k+1,i} \eqdef \F_{t,k+1/2,i} \vee
                                                                         \sigma\left( \Q(\Delta_{t,k+1,i}) \right) \eqsp, \\ & \F_{t,k+1} \eqdef
                                                                                                                               \bigvee_{i=1}^n \F_{t,k+1,i} \eqsp, \qquad \F_{t,k+1,i}^+ \eqdef
                                                                                                                               \F_{t,k+1} \eqsp.
\end{align*}
With these notations, for $t \in [\kouter]^\star$, $k \in [\kmax-1]$
and $i \in [n]^\star$,
$\hatS_{t,k+1} \in \F_{t,k+1,i}^+$, $\Smem_{t,k+1,i} \in \F_{t,k+1/2,i}$,$ \Delta_{t,k+1,i} \in \F_{t,k+1/2,i} $, $ V_{t,k+1,i}
\in \F_{t,k+1,i}$, $\hatS_{t,k+1} \in \F_{t,k+1}$ $ H_{t,k+1} \in \F_{t,k+1} $, and $ V_{t,k+1} \in \F_{t,k+1} $.


\subsection{Computed conditional expectations complexity.}
\label{app:KCE}
In this section, we provide a discussion on the computed conditional
expectations complexity $ \mathcal K_\CE $ that was removed from the
main text due to spaces constraints.

The number of calls to conditional expectations (i.e., computing
$\bars_{ij}$) to perform $\kouter$ outer steps of
\autoref{algo:DIANASPIDEREM}, each composed of $ \kin $ inner
iterations, with $ n $ workers and mini-batches of size $ \lbatch $ is
\[
nm \kouter + n (2 \lbatch) \kin \kouter = n \kin \kouter \left
(\frac{m}{\kin} + 2b \right ) \eqsp;
\]
it corresponds to one full pass on the data at the beginning of each
outer loop and two batches of size $ \lbatch $ on each worker $ i\in
[n]^\star $, at each inner iteration.  In oder to reach an accuracy $
\epsilon $, we need $ (\kin \kouter \pas)^{-1} = O(\epsilon) $ with
the parameter choices in \Cref{theo:DS} (esp. on $ \lbatch $) we thus
have
\[
\mathcal K_{\CE} (\epsilon)= O \left (\frac{n }{\epsilon \pas }
\left(\frac{m}{\kin} + 2 \frac{\kin}{(1+\omega)^2}\right)\right ) \eqsp.
\]
This complexity is minimized with $ \kin = (1+\omega) \sqrt{m/2}$.  We
then obtain an overall complexity $\mathcal{K}_{\CE}$ of $ O \left
(\frac{ \sqrt{m} }{\epsilon\pas } \frac{ n}{(1+\omega)}\right )$.  We
stress the following two points:
\begin{enumerate}[topsep=0pt,itemsep=2pt,leftmargin=*,noitemsep,wide]
\item 
\textbf{Dependency w.r.t. $ {m} $}: the complexity increases as
$ \sqrt{m} $. For $ n=1, \omega=0 $, this yields a scaling equal to
$ \sqrt{m}\epsilon^{-1} $ that corresponds to the optimal
$ \mathcal K_{\CE} $ of \texttt{SPIDER-EM}~\cite{SPIDER-EM};
\item \textbf{Dependency w.r.t. $ \omega $.} Again, the dependency on
  $ \omega $ depends on the regime for $ \pas $. In the (worst case
  regime), $ \pas = O(\sqrt{n}/\omega^{3/2}) $, we get
  \[
  \mathcal
  K_{\CE} (\epsilon) = O\left (\frac{\sqrt{m} \sqrt{n}
    \sqrt{\omega}}{\epsilon } \right )
  \]
  when $\epsilon\to 0$ and $\omega,n \to \infty $, which corresponds
  to a sublinear increase w.r.t. $ \omega $ (that compares to a linear
  increase in the cost of each communication).
\end{enumerate}

\subsection{Preliminary results}
\subsubsection{Results on the minibatch $\batch_{t,k+1}$}
The proof of the following proposition is given in \citep[Lemma
4]{SPIDER-EM}. It establishes the bias and the variance of the sum
along the random set of indices $\batch_{t,k+1}$ conditionally to the
past.
\begin{proposition} \label{prop:batch}
Let $\batch$ be a minibatch of size $\lbatch$, sampled at random (with
or without replacement) from $[m]^\star$. It holds for any $i \in
[n]^\star$ and $s \in \rset^q$,
\[
\PE\left[ \frac{1}{\lbatch} \sum_{j \in \batch} \bars_{ij} \circ \map
  (s) \right] = \frac{1}{m} \sum_{j=1}^m \bars_{ij} \circ \map (s) \eqsp;
\]
and for any $s,s' \in \rset^q$,
\begin{multline*}
  \PE\left[\Big\| \frac{1}{\lbatch} \sum_{j \in \batch} \left\{
      \bars_{ij} \circ \map (s) - s) - (\bars_{ij} \circ \map(s') -
      s') \right\}  \right. \\ \left.
- \frac{1}{m} \sum_{j=1}^m \left\{ (\bars_{ij}
      \circ \map (s) - s) - (\bars_{ij} \circ \map(s') -s') \right\}
    \Big\|^2 \right] \leq \frac{L_i^2}{\lbatch}\|s-s'\|^2 \eqsp.
\end{multline*}
\end{proposition}

\subsubsection{Results on the statistics $\Smem_{t,k,i}$}
\label{sec:DS:controlvariate}
\Cref{prop:DS:biasS} shows that for $k \geq 1$, $\Smem_{t,k+1,i}$ is a
biased approximation of
$m^{-1} \sum_{j=1}^m \bars_{ij} \circ \map(\hatS_{t,k})$; and this
bias is canceled at the beginning of each outer loop since
$\Smem_{t,1,i} = m^{-1} \sum_{j=1}^m \bars_{ij} \circ
\map(\hatS_{t,0})$.
\Cref{coro:DS:biasS} establishes an upper bound for the conditional
variance and the mean squared error of $\Smem_{t,k+1,i}$. \\ Let us
comment the definition of $\Smem_{t,k+1,i}$. For any
$t \in [\kouter]^\star$, $k \in [\kin-1]$ and $i \in [n]^\star$,
\[
\Smem_{t,k+1,i} = \frac{1}{\lbatch} \sum_{j \in \batch_{t,k+1,i}}
\bars_{ij} \circ \map(\hatS_{t,k}) + \Upsilon_{t,k+1,i} \eqsp, \quad
\Upsilon_{t,k+1,i} \eqdef \Smem_{t,k,i} - \frac{1}{\lbatch} \sum_{j \in
  \batch_{t,k+1,i}} \bars_{ij} \circ \map(\hatS_{t,k-1}) \eqsp.
\]
It is easily seen that
\[
\Upsilon_{t,k+1,i} =\Upsilon_{t,k,i} + \frac{1}{\lbatch} \sum_{j \in
  \batch_{t,k,i}} \bars_{ij} \circ \map(\hatS_{t,k-1})- \frac{1}{\lbatch}
\sum_{j \in \batch_{t,k+1,i}} \bars_{ij} \circ
\map(\hatS_{t,k-1})\eqsp,
\]
and since $\Upsilon_{t,1,i} = \Smem_{t,0,i} - \lbatch^{-1} \sum_{j \in
  \batch_{t,1,i}} \bars_{ij} \circ \map(\hatS_{t,-1})$, we have by
using \Cref{prop:DS:biasS},
\begin{align*}
 \Upsilon_{t,k,i} &=\sum_{\ell=1}^k \left\{ \frac{1}{\lbatch} \sum_{j
   \in \batch_{t,\ell,i}} \bars_{ij} \circ \map(\hatS_{t,\ell-1})-
 \frac{1}{\lbatch} \sum_{j \in \batch_{t,\ell+1,i}} \bars_{ij} \circ
 \map(\hatS_{t,\ell-1}) \right\} \\ & + \frac{1}{m} \sum_{j =1}^m \bars_{ij}
 \circ \map(\hatS_{t,-1}) - \frac{1}{\lbatch} \sum_{j \in \batch_{t,1,i}}
 \bars_{ij} \circ \map(\hatS_{t,-1}) \eqsp.
\end{align*}
We have $\CPE{\Upsilon_{t,k,i}}{\F_{t,0}} = 0$ but conditionally to
the past $\F_{t,k-1,i}^+$, the variable $\Upsilon_{t,k,i}$ is {\em
  not} centered.

\begin{proposition} \label{prop:DS:biasS}
  For any $t \in [\kouter]^\star$ and $i \in [n]^\star$,
  \[
\Smem_{t,1,i} - \frac{1}{m} \sum_{j=1}^m \bars_{ij} \circ
\map(\hatS_{t,0}) = \Smem_{t,0,i} - \frac{1}{m} \sum_{j=1}^m
\bars_{ij} \circ \map(\hatS_{t,-1})  = 0 \eqsp.
  \]
  For any $t \in [\kouter]^\star$, $k \in [\kin-1]$ and $i \in
  [n]^\star$, we have
  \[
\CPE{\Smem_{t,k+1,i} }{ \F_{t,k,i}^+ }- \frac{1}{m}
\sum_{j=1}^m \bars_{ij} \circ \map(\hatS_{t,k}) = \Smem_{t,k,i} -
\frac{1}{m} \sum_{j=1}^m \bars_{ij} \circ \map(\hatS_{t,k-1}) \eqsp.
  \]
\end{proposition}
\begin{proof}
  Let $t \in [\kouter]^\star$ and $i \in [n]^\star$. We have by
  definition of $\Smem_{t,1,i}$ and $\Smem_{t,0,i}$
  \[
\Smem_{t,1,i} = \Smem_{t,0,i} + \lbatch^{-1} \sum_{j \in
  \batch_{t,1,i}} \left( \bars_{ij} \circ \map(\hatS_{t,0}) -
\bars_{ij} \circ \map(\hatS_{t,-1}) \right) = \Smem_{t,0,i} =
\frac{1}{m} \sum_{j=1}^m \bars_{ij} \circ \map(\hatS_{t,0})
\]
where we used that $\hatS_{t,0} = \hatS_{t,-1}$.

Let $k \in [\kin -1]$. By definition of $\Smem_{t,k+1,i}$, we have
\[
\Smem_{t,k+1,i} - \Smem_{t,k,i} = \lbatch^{-1} \sum_{j \in
  \batch_{t,k+1,i}} \left( \bars_{ij} \circ \map(\hatS_{t,k}) -
\bars_{ij} \circ \map(\hatS_{t,k-1}) \right) \eqsp.
\]
Since $\hatS_{t,k}, \hatS_{t,k-1} \in \F_{t,k,i}^+$, we have by
\Cref{prop:batch}
\begin{multline*}
\CPE{\lbatch^{-1} \sum_{j \in \batch_{t,k+1,i}} \left( \bars_{ij}
  \circ \map(\hatS_{t,k}) - \bars_{ij} \circ \map(\hatS_{t,k-1})
  \right) }{\F_{t,k,i}^+ } \\ = \frac{1}{m} \sum_{j=1}^m
\left( \bars_{ij} \circ \map(\hatS_{t,k}) - \bars_{ij} \circ
\map(\hatS_{t,k-1}) \right)
\end{multline*}
and the proof follows.
\end{proof}

\begin{corollary}[of \Cref{prop:DS:biasS}] \label{coro:DS:biasS}
  Assume \Cref{hyp:DS:lipschitz}.  For any $t \in [\kouter]^\star$,
  $k \in [\kin-1]$ and $i \in [n]^\star$,
 \begin{align*}
  & \CPE{\| \Smem_{t,k+1,i} - \CPE{ \Smem_{t,k+1,i} }{
      \F_{t,k,i}} \|^2 }{ \F_{t,k} } \leq
  \frac{L_i^2}{\lbatch} \pas_{t,k}^2 \| H_{t,k} \|^2 \eqsp, \\
  & \CPE{\| \Smem_{t,k+1,i} - \frac{1}{m} \sum_{j=1}^m \bars_{ij}
    \circ \map(\hatS_{t,k}) \|^2 }{\F_{t,0} }\leq
  \frac{L_i^2}{\lbatch} \sum_{\ell=1}^k \pas_{t,\ell}^2 \CPE{ \|
    H_{t,\ell} \|^2 }{ \F_{t,0} } \eqsp.
 \end{align*}
 By convention, $H_{t,0} =0$ and $\sum_{\ell=1}^0 a_\ell =0$.
\end{corollary}
\begin{proof}
  Note that $\hatS_{t,k}, \hatS_{t,k-1} \in \F_{t,k}$. By
  \Cref{prop:DS:biasS}, we have {\small
  \begin{multline*}
  \small
    \CPE{\| \Smem_{t,k+1,i} - \CPE{ \Smem_{t,k+1,i} }{
        \F_{t,k,i}^+}\|^2 }{ \F_{t,k} } \\
        \small
         = \CPE{\Big\|
      \frac{1}{\lbatch} \sum_{j \in \batch_{t,k+1,i}} \left(\bars_{ij}
      \circ \map(\hatS_{t,k}) - \bars_{ij} \circ \map(\hatS_{t,k-1})
      \right) \right.  \left. - \frac{1}{m} \sum_{j=1}^m
      \left(\bars_{ij} \circ \map(\hatS_{t,k}) - \bars_{ij} \circ
      \map(\hatS_{t,k-1}) \right) \Big\|^2 }{ \F_{t,k} } \eqsp.
  \end{multline*}}
  By \Cref{prop:batch}, it holds
  \begin{align*}
   \PE\left[\| \Smem_{t,k+1,i} - \PE\left[ \Smem_{t,k+1,i} \vert
       \F_{t,k,i}\right] \|^2 \vert \F_{t,k} \right] \leq
   \frac{L_i^2}{\lbatch} \| \hatS_{t,k} - \hatS_{t,k-1} \|^2 =
   \frac{L_i^2}{\lbatch} \pas_{t,k}^2 \| H_{t,k} \|^2 \eqsp;
  \end{align*}
  with the convention that $H_{t,0} =0$ since $\hatS_{t,0} =
  \hatS_{t,-1}$.  The proof of the first statement is concluded.

  For the second statement, by definition of the conditional
  expectation and since $\hatS_{t,k} \in \F_{t,k} \subset
  \F_{t,k,i}^+$, it holds
  \begin{multline*}
    \CPE{\| \Smem_{t,k+1,i} - \frac{1}{m} \sum_{j=1}^m \bars_{ij}
      \circ \map(\hatS_{t,k}) \|^2 }{\F_{t,k} } =
    \CPE{\| \Smem_{t,k+1,i} - \CPE{\Smem_{t,k+1,i}}{
        \F_{t,k,i}^+}\|^2 }{\F_{t,k} } \\ + \CPE{
      \| \CPE{\Smem_{t,k+1,i} }{ \F_{t,k,i}^+}-
      \frac{1}{m} \sum_{j=1}^m \bars_{ij} \circ \map(\hatS_{t,k}) \|^2
     }{ \F_{t,k}} \eqsp.
  \end{multline*}
  By \Cref{prop:DS:biasS},
  \[
 \Bigg \| \CPE{ \Smem_{t,k+1,i}}{ \F_{t,k,i}^+}- \frac{1}{m}
  \sum_{j=1}^m \bars_{ij} \circ \map(\hatS_{t,k})\Bigg \|^2 =
 \Bigg\|\Smem_{t,k,i} - \frac{1}{m} \sum_{j=1}^m \bars_{ij} \circ
  \map(\hatS_{t,k-1}) \Bigg\|^2.
  \]
  Hence, by using $\Smem_{t,1,i} -m^{-1} \sum_{j=1}^m \bars_{ij} \circ
  \map(\hatS_{t,0}) =0$ (see \Cref{prop:DS:biasS}), we
  have
  \begin{align*}
    & \CPE{\Big\| \Smem_{t,k+1,i} - \frac{1}{m} \sum_{j=1}^m
      \bars_{ij} \circ \map(\hatS_{t,k})\Big \|^2 }{ \F_{t,0} }
    \\ & \qquad \leq \frac{L_i^2}{\lbatch} \pas_{t,k}^2 \CPE{ \|
      H_{t,k} \|^2 }{ \F_{t,0} } + \CPE{ \Big\|
      \Smem_{t,k,i}- \frac{1}{m} \sum_{j=1}^m \bars_{ij} \circ
      \map(\hatS_{t,k-1})\Big \|^2}{ \F_{t,0} } \\ & \qquad \leq
    \frac{L_i^2}{\lbatch} \sum_{\ell=1}^k \pas_{t,\ell}^2 \CPE{ \|
      H_{t,\ell} \|^2 }{\F_{t,0} }\eqsp.
  \end{align*}
\end{proof}

\subsubsection{Results on $\Delta_{t,k+1,i}$}
\Cref{prop:DS:Delta} provides an upper bound for the mean value of the
conditional variance of $\Delta_{t,k+1, \cdot}$ and for its
$L_2$-moment.  \Cref{prop:DS:varDelta} prepares the control of the
varianc of the random field $H_{t,k+1}$ upon noting that
\[
H_{t,k+1} - \CPE{H_{t,k+1}}{\F_{t,k}} = \frac{1}{n} \sum_{i=1}^n
\left( \Q(\Delta_{t,k+1,i}) - \CPE{ \Delta_{t,k+1,i} }{\F_{t,k}}
\right) \eqsp.
\]
\begin{proposition}
  \label{prop:DS:Delta}
  Assume \Cref{hyp:DS:lipschitz}.  For any $t \in [\kouter]^\star$ and
  $k \in [\kin-1]$,
  \begin{multline*}
  \frac{1}{n} \sum_{i=1}^n \CPE{\|\Delta_{t,k+1,i}\|^2 }{ \F_{t,0} }
  \\ \leq 2 \frac{L^2}{\lbatch} \sum_{\ell=1}^k \pas_{t,\ell}^2 \CPE{ \|
    H_{t,\ell} \|^2 }{\F_{t,0} } + \frac{2}{n} \sum_{i=1}^n \CPE{ \|
    \mf_i(\hatS_{t,k}) - V_{t,k,i} \|^2 }{ \F_{t,0} } \eqsp.
\end{multline*}
In addition, 
\[
 \frac{1}{n} \sum_{i=1}^n \CPE{ \|   \Delta_{t,k+1,i} - \PE\left[
    \Delta_{t,k+1,i} \vert \F_{t,k} \right] \|^2 }{ \F_{t,k}} \leq \frac{L^2}{ \lbatch} \pas_{t,k}^2 \|H_{t,k}\|^2 \eqsp.
\]
  \end{proposition}
\begin{proof}
  Let $i \in [n]^\star$, $t \in [\kouter]^\star$ and $k \in [\kin-1]$.
  We write
  \[
\Delta_{t,k+1,i} = \Smem_{t,k+1,i} - \frac{1}{m}\sum_{j=1}^m
\bars_{ij} \circ \map(\hatS_{t,k}) + \mf_i(\hatS_{t,k}) - V_{t,k,i}  \eqsp.
\]
When $k=0$, we have $\Smem_{t,1,i} - \frac{1}{m}\sum_{j=1}^m
\bars_{ij} \circ \map(\hatS_{t,0}) =0$ (see \Cref{prop:DS:biasS}) so
that $\Delta_{t,1,i} = \mf_i(\hatS_{t,0}) - V_{t,0,i}$. For $k \geq
1$, we write
\begin{multline*}
  \CPE{ \|\Delta_{t,k+1,i}\|^2 }{\F_{t,0} } \leq 2 \, \CPE{\|
    \Smem_{t,k+1,i} - \frac{1}{m}\sum_{j=1}^m \bars_{ij} \circ
    \map(\hatS_{t,k}) \|^2 }{ \F_{t,0} }  \\
  + 2 \, \CPE{\| \mf_i(\hatS_{t,k}) - V_{t,k,i} \|^2 }{ \F_{t,0} }
\end{multline*}
and the proof of the first statement is concluded by \Cref{coro:DS:biasS}.

  By definition of $\Delta_{t,k+1,i}$, it holds
  \begin{align}
    \Delta_{t,k+1,i} - \CPE{\Delta_{t,k+1,i} }{\F_{t,k}} =
    \Smem_{t,k+1,i} - \CPE{\Smem_{t,k+1,i}}{ \F_{t,k}}
    \eqsp. \label{eq:tool746}
\end{align}
The proof is concluded by \eqref{eq:tool746} and \Cref{coro:DS:biasS}.
\end{proof}

\begin{proposition}
\label{prop:DS:varDelta}
Assume \Cref{hyp:var:quantif}
  and \Cref{hyp:DS:lipschitz}. For any $t \in [\kouter]^\star$ and
  $k \in [\kin-1]$,
\begin{multline*}
  \frac{1}{n} \sum_{i=1}^n \CPE{\|\Q(\Delta_{t,k+1,i}) -
    \CPE{\Delta_{t,k+1,i} }{ \F_{t,k} } \|^2}{ \F_{t,0} } \leq
  \frac{\omega}{n} \sum_{i=1}^n  \CPE{ \| \Delta_{t,k+1,i} \|^2 }{ \F_{t,0}} \\
  + \frac{L^2}{\lbatch} \pas_{t,k}^2 \CPE{ \| H_{t,k}\|^2 }{ \F_{t,0}
  } \eqsp.
\end{multline*}
\end{proposition}
\begin{proof}
  Let $i \in [n]^\star$, $t \in [\kouter]^\star$ and $k \in [\kin-1]$.
  We write
\[
\Q(\Delta_{t,k+1,i}) - \CPE{\Delta_{t,k+1,i} }{ \F_{t,k} }= \Q(\Delta_{t,k+1,i}) - \Delta_{t,k+1,i} + \Delta_{t,k+1,i} -
\CPE{\Delta_{t,k+1,i} }{ \F_{t,k} } \eqsp;
\]
and use the property
\begin{multline*} 
\CPE{ \| \Q(\Delta_{t,k+1,i}) - \CPE{\Delta_{t,k+1,i} }{
    \F_{t,k} } \|^2}{ \F_{t,0} } = \CPE{ \|
  \Q(\Delta_{t,k+1,i}) - \Delta_{t,k+1,i} \|^2 }{ \F_{t,0} } \\
+ \CPE{ \|\Delta_{t,k+1,i} - \CPE{\Delta_{t,k+1,i} }{
    \F_{t,k} } \|^2 }{ \F_{t,0} }\eqsp.
\end{multline*}
By \Cref{hyp:var:quantif} and $\F_{t,k} \subset \F_{t,k+1/2,i}$, we have
\begin{align*}
  & \CPE{ \| \Q(\Delta_{t,k+1,i}) - \Delta_{t,k+1,i} \|^2 }{
    \F_{t,0} } \nonumber \\ &  = \CPE{ \CPE{\|
                              \Q(\Delta_{t,k+1,i}) - \Delta_{t,k+1,i} \|^2 }{
                              \F_{t,k+1/2,i} } }{ \F_{t,0} } \leq \omega
                              \CPE{ \| \Delta_{t,k+1,i} \|^2 }{ \F_{t,0}}
                              \eqsp; 
\end{align*}
in addition, by \Cref{prop:DS:Delta},
\begin{align*}
 n^{-1} \sum_{i=1}^n  \CPE{\| \Delta_{t,k+1,i} - \CPE{\Delta_{t,k+1,i} }{
  \F_{t,k} }\|^2 }{ \F_{t,0} }  \leq
  \frac{L^2}{\lbatch} \pas_{t,k}^2 \CPE{ \| H_{t,k}\|^2 }{   \F_{t,0} } \eqsp. 
\end{align*}
This concludes the proof.
\end{proof}

\subsubsection{Results on the memory terms  $V_{t,k+1,i}$}
\Cref{lem:DS:V} proves that the memory term $V_{t,k+1}$
  computed by the central server is the mean value of the local
  $V_{t,k+1,i}$ computed by each worker $\# i$.  \Cref{prop:DS:driftV}
  establishes a contraction-like inequality on the mean quantity
  $n^{-1} \sum_{i=1}^n \| V_{t,k+1,i} - \mf_i(\hatS_{t,k+1}) \|^2$
  thus providing the intuition that $V_{t,k+1,i}$ approximates
  $\mf_i(\hatS_{t,k+1})$.

\begin{lemma}\label{lem:DS:V}
  For any $t \in [\kouter]^\star$ and $k \in [\kin-1]$,
  \[
V_{t,k+1} = \frac{1}{n} \sum_{i=1}^n V_{t,k+1,i} \eqsp, \qquad V_{t,0}
= \frac{1}{n} \sum_{i=1}^n V_{t,0,i} \eqsp.
  \]
  \end{lemma}
\begin{proof}
  The proof is by induction on $t$ and $k$. Consider the case
  $t=1$. When $k=0$, the property holds true by
  Line~\ref{algo:DS:init:V0} in \autoref{algo:DIANASPIDEREM}. Assume
  that the property holds for $k \leq \kin-2$. Then by definition of
  $V_{1,k+1}$ and by the induction assumption:
  \begin{align*}
  V_{1,k+1} &= V_{1,k} + \alpha \frac{1}{n} \sum_{i=1}^n
  \Q(\Delta_{1,k+1,i}) = \frac{1}{n} \sum_{i=1}^n \left( V_{1,k,i}
  + \alpha \Q(\Delta_{1,k+1,i}) \right) \\ &= \frac{1}{n}
  \sum_{i=1}^n V_{1,k+1,i} \eqsp.
  \end{align*}
  By Lines~\ref{algo:DS:init:V1} and ~\ref{algo:DS:init:V2} in~\autoref{algo:DIANASPIDEREM} and by the induction on $k$,
  we obtain
  \[
  V_{2,0} = V_{1,\kin} = \frac{1}{n} \sum_{i=1}^n V_{1,\kin,i}
  =\frac{1}{n} \sum_{i=1}^n V_{2,0,i} \eqsp.
  \]
 Assume that for $t \in [\kouter-1]^\star$ we have $V_{t,0} = n^{-1}
 \sum_{i=1}^n V_{t,0,i}$. As in the case $t=1$, we prove by induction
 on $k$ that for any $k \in [\kin-1]$, $V_{t,k+1} = n^{-1}
 \sum_{i=1}^n V_{t,k+1,i}$ (details are omitted). This implies, by
 using Lines~\ref{algo:DS:init:V1} and ~\ref{algo:DS:init:V2}
 of~\autoref{algo:DIANASPIDEREM}, that
 \[
  V_{t+1,0} = V_{t,\kin} = \frac{1}{n} \sum_{i=1}^n V_{t,\kin,i}
  =\frac{1}{n} \sum_{i=1}^n V_{t+1,0,i} \eqsp.
  \]
  This concludes the induction.
  \end{proof}

  \begin{proposition} \label{prop:DS:driftV} Assume
    \Cref{hyp:var:quantif} and \Cref{hyp:DS:lipschitz}.  Let
    $\alpha \in \ocint{0,(1+\omega)^{-1}}$.  For any
    $t \in [\kouter]^\star$, $k \in [\kin-1]$ and $i \in [n]^\star$,
    it holds
\[
\CPE{ V_{t,k+1,i} }{\F_{t,k+1/2,i} } = (1-\alpha) \, V_{t,k,i}
+ \alpha \ \left( \Smem_{t,k+1,i} - \hatS_{t,k} \right) \eqsp,
\]
Define for $t \in [\kouter]^\star$ and $k \in [\kin]$
\[
G_{t,k} \eqdef \frac{1}{n}\sum_{i=1}^n \|V_{t,k,i} -
\mf_i(\hatS_{t,k}) \|^2 \eqsp.
\]
We have
\begin{align*}
  \CPE{G_{t,k+1}}{ \F_{t,0}} & \leq (1-\alpha/2)
                               \CPE{G_{t,k}}{ \F_{t,0}} \\ &  + \frac{2}{\alpha} L^2
                                                             \pas_{t,k+1}^2  \CPE{ \|H_{t,k+1}\|^2}{ \F_{t,0}} + 2
                                                             \alpha \frac{L^2}{\lbatch} \sum_{\ell=1}^k \pas_{t,\ell}^2
                                                             \CPE{\|H_{t,\ell}\|^2 }{ \F_{t,0} } \\ & +
                                                                                                      \alpha \left(\alpha (1+\omega) - 1  \right)
                                                                                                      \frac{1}{n}\sum_{i=1}^n  \CPE{ \| \Delta_{t,k+1,i}\|^2 }{
                                                                                                      \F_{t,0} } \eqsp.
 \end{align*}
\end{proposition}
\begin{proof}
Let $t \in
[\kouter]^\star$, $k \in [\kin-1]$ and $i \in [n]^\star$. By
definition of $V_{t,k+1,i}$, $\Delta_{t,k+1,i}$ and by
\Cref{hyp:var:quantif}, it holds
 \begin{align*}
\CPE{ V_{t,k+1,i}}{\F_{t,k+1/2,i} }& = V_{t,k,i} +
\alpha \CPE{\Q(\Delta_{t,k+1,i}) }{ \F_{t,k+1/2,i} }
\\ &= V_{t,k,i} + \alpha \, \left( \Smem_{t,k+1,i} - \hatS_{t,k} -
V_{t,k,i}\right) \eqsp.
  \end{align*}
This concludes the proof of the first statement.  For the second
statement, we write for any $\beta >0$:
\begin{align}
\|V_{t,k+1,i} - \mf_i(\hatS_{t,k+1}) \|^2 & \leq (1+\beta^2)
\|\mf_i(\hatS_{t,k+1}) - \mf_i(\hatS_{t,k}) \|^2 + (1+\beta^{-2}) \|
V_{t,k+1,i} - \mf_i(\hatS_{t,k}) \|^2 \nonumber \\ & \leq (1+\beta^2)
L_i^2 \pas_{t,k+1}^2 \|H_{t,k+1}\|^2 + (1+\beta^{-2}) \| V_{t,k+1,i} -
\mf_i(\hatS_{t,k}) \|^2 \eqsp, \label{eq:tool272}
\end{align}
where we used \Cref{hyp:DS:lipschitz} and the definition of
$\hatS_{t,k+1}$ in the last inequality. For any $s \in \rset^q$
  \begin{multline}
 \CPE{\|V_{t,k+1,i} - s \|^2 }{ \F_{t,k+1/2,i}} =
 \CPE{ \|V_{t,k+1,i} - \CPE{V_{t,k+1,i} }{ \F_{t,k+1/2,i}
     } \|^2 }{\F_{t,k+1/2,i} } \\ +
 \|\CPE{V_{t,k+1,i} - s}{\F_{t,k+1/2,i}} \|^2 \eqsp. \label{eq:tool273}
 \end{multline}
 On one hand,
\[  \|V_{t,k+1,i} - \CPE{V_{t,k+1,i}}{ \F_{t,k+1/2,i} } \|^2  =
  \alpha^2 \|\Q(\Delta_{t,k+1,i}) - \CPE{\Q(\Delta_{t,k+1,i}) }{
    \F_{t,k+1/2,i} } \|^2
  \]
  and by \Cref{hyp:var:quantif},
    \begin{equation}
 \CPE{ \|V_{t,k+1,i} - \CPE{V_{t,k+1,i} }{\F_{t,k+1/2,i}
     } \|^2 }{ \F_{t,k+1/2,i} }\leq \alpha^2 \omega
 \|\Delta_{t,k+1,i}\|^2 \eqsp. \label{eq:tool274}
\end{equation}
On the other hand, for any $s \in \rset^q$, and using \Cref{lem:polar}
  \begin{align}
  & \|\CPE{V_{t,k+1,i} -s }{ \F_{t,k+1/2,i} } \|^2 =
    \|(1-\alpha) \, (V_{t,k,i} -s) + \alpha \ (\Smem_{t,k+1,i} -
    \hatS_{t,k} -s) \|^2  \nonumber \\ & = (1-\alpha) \, \| V_{t,k,i} -s\|^2 +
    \alpha \| \Smem_{t,k+1,i} - \hatS_{t,k}-s \|^2 - \alpha (1-\alpha)
    \| V_{t,k,i} - \Smem_{t,k+1,i} + \hatS_{t,k}\|^2  \nonumber \\ & = (1-\alpha)
    \, \| V_{t,k,i} -s \|^2 + \alpha \| \Smem_{t,k+1,i} -
    \hatS_{t,k}-s \|^2 - \alpha (1-\alpha) \| \Delta_{t,k+1,i}\|^2 \label{eq:tool275}
    \eqsp.
  \end{align}
  Let us combine \eqref{eq:tool272} to \eqref{eq:tool275}, the last
  one being applied with $s \leftarrow \mf_i(\hatS_{t,k}) \in
  \F_{t,k,i}^+ \subseteq \F_{t,k+1/2,i}$. Since
\[
 \| \Smem_{t,k+1,i} - \hatS_{t,k}- \mf_i(\hatS_{t,k}) \|^2  = \|
 \Smem_{t,k+1,i} - \frac{1}{m} \sum_{j=1}^m \bars_{ij} \circ \map
 (\hatS_{t,k}) \|^2 \eqsp,
 \]
 we write
 \begin{align*}
& \CPE{\|V_{t,k+1,i} - \mf_i(\hatS_{t,k+1}) \|^2 }{\F_{t,k}
     } \leq (1+\beta^2) L_i^2 \pas_{t,k+1}^2 \CPE{
     \|H_{t,k+1}\|^2 }{ \F_{t,k} } \\ \qquad & +
   (1+\beta^{-2}) \left\{ \alpha^2 \omega \CPE{ \|
     \Delta_{t,k+1,i}\|^2 }{ \F_{t,k} } + (1-\alpha)
   \|V_{t,k,i} - \mf_i(\hatS_{t,k}) \|^2 \right. \\ & \left. +\alpha
   \CPE{ \| \Smem_{t,k+1,i} - \frac{1}{m} \sum_{j=1}^m \bars_{ij}
     \circ \map (\hatS_{t,k}) \|^2 }{ \F_{t,k} }-
   \alpha(1-\alpha) \CPE{\| \Delta_{t,k+1,i}\|^2 }{\F_{t,k}
    }\right\} \eqsp.
   \end{align*}
 Choose $\beta^2>0$ such that
 \[
\beta^{-2} \eqdef \left\{ \begin{array}{cc} 1 & \text{if $\alpha \geq
    2/3$} \\ \frac{\alpha}{2(1-\alpha)} & \text{if $\alpha \leq 2/3$}
   \end{array} \right.
 \]
This implies that
 \[
(1+\beta^{-2}) (1-\alpha) \leq 1 - \frac{\alpha}{2} \eqsp, \qquad
 1+\beta^2 \leq \frac{2}{\alpha}  \eqsp, \qquad 1+\beta^{-2} \leq 2 \eqsp.
 \]
Hence,
 \begin{align*}
& \CPE{\|V_{t,k+1,i} - \mf_i(\hatS_{t,k+1}) \|^2 }{ \F_{t,k}} \leq (1-\alpha/2) \|V_{t,k,i} - \mf_i(\hatS_{t,k}) \|^2
   \\ & \qquad + \frac{2}{\alpha} L_i^2 \pas_{t,k+1}^2 \CPE{
     \|H_{t,k+1}\|^2 }{ \F_{t,k} }+ 2 \alpha \CPE{ \|
     \Smem_{t,k+1,i} - \frac{1}{m} \sum_{j=1}^m \bars_{ij} \circ \map
     (\hatS_{t,k}) \|^2  }{\F_{t,k} } \\ & \qquad + \alpha
   \left(\alpha \omega - 1 +\alpha \right) \CPE{ \|
     \Delta_{t,k+1,i}\|^2 }{\F_{t,k}} \eqsp;
 \end{align*}
 (in the last equality, we use $1+\beta^{-2} \geq 1$ since $\alpha \omega-1+\alpha \leq 0$).
 Finally, by using \Cref{coro:DS:biasS}, we have
 \begin{align*}
   & \CPE{\|V_{t,k+1,i} - \mf_i(\hatS_{t,k+1}) \|^2 }{\F_{t,0}}\leq (1-\alpha/2) \CPE{\|V_{t,k,i} -
     \mf_i(\hatS_{t,k}) \|^2 }{ \F_{t,0} } \\ & \qquad +
                                                \frac{2}{\alpha} L_i^2 \pas_{t,k+1}^2 \CPE{ \|H_{t,k+1}\|^2
                                                }{ \F_{t,0} } + 2 \alpha \frac{L_i^2}{\lbatch}
                                                \sum_{\ell=1}^k \pas_{t,\ell}^2 \CPE{ \|H_{t,\ell}\|^2 }{
                                                \F_{t,0} } \\ & \qquad +  \alpha \left(\alpha \omega - 1
                                                                +\alpha \right) \CPE{ \| \Delta_{t,k+1,i}\|^2 }{ \F_{t,0}
                                                                }\eqsp.
 \end{align*}
 The proof is concluded.
\end{proof}

\subsubsection{Results on the random field $H_{t,k+1}$}
\Cref{prop:DS:biasH} shows that the random field
  $H_{t,k+1}$ is a biased approximation of the field
  $\mf(\hatS_{t,k})$, and this bias is canceled at the beginning of
  each outer loop. Observe also that the bias exists even when there is no compression: when $\omega=0$ (so that $\Q(u) = u$) we have 
\[
\CPE{H_{t,k+1}}{ \F_{t,k}} - \mf(\hatS_{t,k}) = H_{t,k} -
\mf(\hatS_{t,k-1}) \eqsp,
\]
and the bias is again canceled at the beginning of each outer
loop. \Cref{prop:DS:Hvariance} provides an upper bound for the
variance and the mean squared error of the random field
$H_{t,k+1}$. In the case of no compression ($\omega=0$) and of a
single worker ($n=1$) so that \VRFEDEM~is \texttt{SPIDER-EM},
\Cref{prop:DS:Hvariance} retrieves the variance and the mean squared
error of the random field $H_{t,k+1}$ in \texttt{SPIDER-EM} (see
\cite[Proposition 13]{SPIDER-EM}). 
\begin{proposition}
  \label{prop:DS:biasH} 
  Assume \Cref{hyp:var:quantif}.  For any $t \in [\kouter]^\star$,
  $\CPE{H_{t,2} }{ \F_{t,0} } - \mf(\hatS_{t,1}) =
  \CPE{H_{t,1}}{ \F_{t,0} } - \mf(\hatS_{t,0}) = 0$
  and for any $k \in [\kin-1]^\star$,
  \begin{align*}
    \CPE{H_{t,k+1}}{ \F_{t,k}} - \mf(\hatS_{t,k}) & =
                                                    H_{t,k} - \mf(\hatS_{t,k-1}) - n^{-1} \sum_{i=1}^n \left(
                                                    \Q(\Delta_{t,k,i}) - \Delta_{t,k,i} \right) \\ & = n^{-1}
                                                                                                     \sum_{i=1}^n \left( \CPE{ \Smem_{t,k+1,i} }{\F_{t,k} }
                                                                                                     - m^{-1} \sum_{j=1}^m \bars_{ij} \circ \map(\hatS_{t,k}) \right)
                                                                                                     \eqsp.
  \end{align*}
  \end{proposition}
\begin{proof}
  Let $t \in [\kouter]^\star$.

  $\bullet$ By definition of $H_{t,1}$ and $\Delta_{t,1,i}$, by
  \Cref{hyp:var:quantif} and by \Cref{lem:DS:V}, we have
  \begin{align*}
\CPE{ H_{t,1} }{ \F_{t,0}} & =V_{t,0} + n^{-1}
\sum_{i=1}^n \CPE{ \Q(\Delta_{t,1,i}) }{ \F_{t,0} } =
V_{t,0} + n^{-1} \sum_{i=1}^n \CPE{\Delta_{t,1,i} }{ \F_{t,0}} \\ & = V_{t,0} + n^{-1} \sum_{i=1}^n \left( \CPE{
  \Smem_{t,1,i} }{\F_{t,0}} - \hatS_{t,0} - V_{t,0,i}
\right) \\ & = n^{-1} \sum_{i=1}^n \CPE{\Smem_{t,1,i} }{  \F_{t,0} } - \hatS_{t,0} \eqsp.
  \end{align*}
By \Cref{prop:DS:biasS} $n^{-1} \sum_{i=1}^n \CPE{ \Smem_{t,1,i}
  }{\F_{t,0} }- \hatS_{t,0} = \mf(\hatS_{t,0})$.

$\bullet$ Consider the case $k=1$. We have by definition of $H_{t,2}$
\[
\CPE{ H_{t,2} }{\F_{t,1} } - \mf(\hatS_{t,1}) =
\frac{1}{n} \sum_{i=1}^n \left( \CPE{ \Smem_{t,2,i}}{ \F_{t,1}} - m^{-1} \sum_{j=1}^m \bars_{ij} \circ \map(\hatS_{t,1})
\right) \eqsp;
\]
\Cref{prop:DS:biasS} concludes the proof.

$\bullet$ Let $k \geq 2$. As in the case $k=0$, we have
\begin{align*}
\CPE{H_{t,k+1} }{\F_{t,k} }& = V_{t,k} + n^{-1}
\sum_{i=1}^n \CPE{\Q(\Delta_{t,k+1,i}) }{ \F_{t,k} } =
V_{t,k} + n^{-1} \sum_{i=1}^n \CPE{\Delta_{t,k+1,i} }{  \F_{t,k} } \\ & = V_{t,k} + n^{-1} \sum_{i=1}^n \left(
 \CPE{ \Smem_{t,k+1,i} }{\F_{t,k} } - \hatS_{t,k} -
V_{t,k,i} \right) \\ & = n^{-1} \sum_{i=1}^n \CPE{ \Smem_{t,k+1,i}
  }{\F_{t,k} } - \hatS_{t,k} \eqsp,
\end{align*}
so that
\begin{equation} \label{eq:tool204}
\CPE{H_{t,k+1} }{\F_{t,k} } - \mf(\hatS_{t,k}) = n^{-1}
\sum_{i=1}^n \left( \CPE{ \Smem_{t,k+1,i}}{ \F_{t,k}} -
m^{-1} \sum_{j=1}^m \bars_{ij} \circ \map(\hatS_{t,k}) \right) \eqsp.
\end{equation}
By \Cref{prop:DS:biasS}, upon noting that $\F_{t,k} \subset
\F_{t,k,i}^+$ and $\Smem_{t,k,i}, \hatS_{t,k-1} \in \F_{t,k}$, we have
\begin{equation} \label{eq:tool205}
n^{-1} \sum_{i=1}^n \CPE{ \Smem_{t,k+1,i} }{ \F_{t,k} } -
m^{-1} \sum_{j=1}^m \bars_{ij} \circ \map(\hatS_{t,k}) = n^{-1}
\sum_{i=1}^n \left( \Smem_{t,k,i} - m^{-1} \sum_{j=1}^m \bars_{ij}
\circ \map(\hatS_{t,k-1}) \right) \eqsp.
\end{equation}
On the other hand, observe that
\begin{align*}
  H_{t,k} & = V_{t,k-1} + n^{-1} \sum_{i=1}^n \Q(\Delta_{t,k,i}) \\ &
  = V_{t,k-1} + n^{-1} \sum_{i=1}^n \Smem_{t,k,i} -\hatS_{t,k-1} -
  n^{-1} \sum_{i=1}^n V_{t,k-1,i} + n^{-1} \sum_{i=1}^n \left(
  \Q(\Delta_{t,k,i}) - \Delta_{t,k,i} \right) \\ & = n^{-1}
  \sum_{i=1}^n \Smem_{t,k,i} -\hatS_{t,k-1} + n^{-1} \sum_{i=1}^n
  \left( \Q(\Delta_{t,k,i}) - \Delta_{t,k,i} \right) \eqsp,
\end{align*}
where we used \Cref{lem:DS:V}.  This yields
\begin{multline} \label{eq:tool206}
  H_{t,k} - \mf(\hatS_{t,k-1}) \\ = n^{-1} \sum_{i=1}^n \left(
  \Smem_{t,k,i} - m^{-1} \sum_{j=1}^m \bars_{ij} \circ
  \map(\hatS_{t,k-1}) \right) + n^{-1} \sum_{i=1}^n \left(
  \Q(\Delta_{t,k,i}) - \Delta_{t,k,i} \right) \eqsp.
\end{multline}
The proof is concluded by combining \eqref{eq:tool204},
\eqref{eq:tool205} and \eqref{eq:tool206}.
\end{proof}

\begin{proposition}
  Assume \Cref{hyp:var:quantif} and \Cref{hyp:DS:lipschitz}.
  \label{prop:DS:Hvariance}
  For any $t \in [\kouter]^\star$,
  \[
\CPE{       \| H_{t,1} - \mf(\hatS_{t,0}) \|^2 }{ \F_{t,0} }\leq \frac{\omega}{n} \left( \frac{1}{n}\sum_{i=1}^n \, \| V_{t,0,i} -
\mf_i(\hatS_{t,0}) \|^2 \right) \eqsp,
\]
and for any $k \in [\kin-1]^\star$,
\begin{align*}
  \CPE{\| H_{t,k+1} - \mf(\hatS_{t,k}) \|^2 }{ \F_{t,0}
  } & \leq \frac{\omega}{n} \frac{1}{n} \sum_{i=1}^n
      \CPE{ \|\Delta_{t,k+1,i}\|^2 }{ \F_{t,0} } +
      \frac{L^2}{n \lbatch} \sum_{\ell=1}^k \pas_{t,\ell}^2 \CPE{ \|
      H_{t,\ell}\|^2 }{ \F_{t,0} } \eqsp, \\ \CPE{\|
  \CPE{H_{t,k+1} }{ \F_{t,k} } - \mf(\hatS_{t,k}) \|^2
  }{\F_{t,0} }& \leq \frac{L^2}{n \lbatch}
                \sum_{\ell=1}^{k-1} \pas_{t,\ell}^2 \CPE{ \| H_{t,\ell}\|^2
                }{\F_{t,0} } \eqsp.
\end{align*}
  \end{proposition}
\begin{proof}
 {\bf $\bullet$ Case $k=1$.} From \Cref{prop:DS:biasH} and the definition of $H_{t,1}$, we have
  \begin{align*}
H_{t,1} - \mf(\hatS_{t,0}) & = H_{t,1} - \CPE{[H_{t,1} }{
  \F_{t,0} } = n^{-1} \sum_{i=1}^n \left( \Q(\Delta_{t,1,i}) -
\CPE{ \Q(\Delta_{t,1,i})}{ \F_{t,0} } \right) \\ & =
n^{-1} \sum_{i=1}^n \left( \Q(\Delta_{t,1,i}) - \CPE{
  \Delta_{t,1,i} }{ \F_{t,0}} \right) \eqsp,
  \end{align*}
  where we used $\CPE{ \Q(\Delta_{t,1,i}) }{\F_{t,1/2,i}} = \Delta_{t,1,i}$ and $\F_{t,0} \subset \F_{t,1/2,i}$ in
  the last equality. In addition, since $\hatS_{t,0} = \hatS_{t,-1}$, we have (see \Cref{prop:DS:biasS})
  \[
\Smem_{t,1,i} = \Smem_{t,0,i} = \mf_i(\hatS_{t,0}) + \hatS_{t,0} \eqsp.
\]
Hence,
  \[
\Delta_{t,1,i} = \Smem_{t,1,i} - \hatS_{t,0} - V_{t,0,i} =
\mf_i(\hatS_{t,0}) - V_{t,0,i} \eqsp.
  \]
Therefore, $\CPE{\Delta_{t,1,i} }{ \F_{t,0} } =
\Delta_{t,1,i} = \mf_i(\hatS_{t,0}) - V_{t,0,i}$.  Since the workers
are independent, we write
\[
\CPE{ \| H_{t,1} - \mf(\hatS_{t,0}) \|^2 }{\F_{t,0} } =
\frac{1}{n^2} \sum_{i=1}^n \CPE{ \| \Q(\mf_i(\hatS_{t,0}) -
  V_{t,0,i}) - \left( \mf_i(\hatS_{t,0}) - V_{t,0,i} \right) \|^2
  }{ \F_{t,0}}\eqsp.
\]
By \Cref{hyp:var:quantif}, this yields
\[
\CPE{ \| H_{t,1} - \mf(\hatS_{t,0}) \|^2 }{ \F_{t,0} }
\leq \frac{\omega}{n} \frac{1}{n} \sum_{i=1}^n \|\mf_i(\hatS_{t,0}) -
V_{t,0,i} \|^2 \eqsp.
\]

{\bf $\bullet$ Case $k \geq 1$.}  Let $t \in [\kouter]^\star$ and $k
\in [\kin-1]^\star$.  We write
\begin{multline}
\CPE{ \| H_{t,k+1} - \mf(\hatS_{t,k}) \|^2 }{\F_{t,0} }
= \CPE{ \| H_{t,k+1} - \CPE{ H_{t,k+1} }{ \F_{t,k} }
  \|^2 }{ \F_{t,0}} \\
+ \CPE{ \| \CPE{H_{t,k+1} }{
    \F_{t,k} } - \mf(\hatS_{t,k}) \|^2 }{ \F_{t,0} } \eqsp. \label{eq:tool510}
\end{multline}
Let us first consider the bias term. From \Cref{prop:DS:biasS}, \Cref{prop:DS:biasH} and the
definition of $\Smem_{t,k+1,i}$ (remember that $\Smem_{t,k,i}$,
$\hatS_{t,k}$ and $\hatS_{t,k-1}$ are in $\F_{t,k,i}^+ \supset
\F_{t,k}$), it holds
\begin{align*}
&\CPE{ \Big\| \CPE{ H_{t,k+1} }{\F_{t,k} } -
    \mf(\hatS_{t,k}) \Big\|^2 }{\F_{t,0} } \\ & \qquad =
  \CPE{ \Big\| n^{-1} \sum_{i=1}^n ( \CPE{ \Smem_{t,k+1,i}
      }{ \F_{t,k} }- m^{-1} \sum_{j=1}^m \bars_{ij} \circ
    \map(\hatS_{t,k}) ) \Big\|^2}{\F_{t,0}} \\ & \qquad
  \leq \CPE{\Big \| n^{-1} \sum_{i=1}^n( \Smem_{t,k,i} - m^{-1}
    \sum_{j=1}^m \bars_{ij} \circ \map(\hatS_{t,k-1}) )\Big \|^2
    }{ \F_{t,0} }\eqsp.
\end{align*}
By \Cref{prop:DS:biasS} again, the RHS is zero when $k=1$; when $k \geq 2$,
by \Cref{coro:DS:biasS} and the independence of the workers, we have
yields
\begin{equation} \label{eq:tool511}
\CPE{ \| \CPE{H_{t,k+1} }{ \F_{t,k} }-
  \mf(\hatS_{t,k}) \|^2 }{\F_{t,0} }\leq \frac{L^2}{n
  \lbatch} \sum_{\ell=1}^{k-1} \pas_{t,\ell}^2 \CPE{\|
  H_{t,\ell}\|^2 }{ \F_{t,0} } \eqsp.
\end{equation}
Let us now consider the variance term. We have from the definition of
$H_{t,k+1}$ and \Cref{hyp:var:quantif}
\[
H_{t,k+1} - \CPE{ H_{t,k+1} }{ \F_{t,k}  }= \frac{1}{n}
\sum_{i=1}^n \left( \Q(\Delta_{t,k+1,i}) - \CPE{\Delta_{t,k+1,i}
  }{ \F_{t,k} } \right)
\]
and here again, by the independence of the workers
\begin{multline}\label{eq:tool512}
\CPE{ \|H_{t,k+1} - \CPE{H_{t,k+1} }{ \F_{t,k} }
  \|^2 }{\F_{t,0} } \\ \leq \frac{1}{n^2} \sum_{i=1}^n
\CPE{ \| \Q(\Delta_{t,k+1,i}) - \CPE{\Delta_{t,k+1,i} }{
    \F_{t,k}  }\|^2 }{ \F_{t,0} }\eqsp.
\end{multline}
The proof follows from \eqref{eq:tool510} to \eqref{eq:tool512} and
\Cref{prop:DS:varDelta}.
\end{proof}

\subsection{Proof of~\autoref{theo:DS}}
\Cref{theo:DS} is a corollary of the more general following
proposition.
\begin{proposition}\label{prop:DS}
	Assume \Cref{hyp:model,hyp:bars,hyp:Tmap}, \Cref{hyp:DS:lyap},
	\Cref{hyp:var:quantif} and \Cref{hyp:DS:lipschitz}. Set
	$L^2 \eqdef n^{-1} m^{-1} \sum_{i=1}^n \sum_{j=1}^m L^2_{ij}$. Let
	$\{\hatS_{t,k}, t \in [\kouter]^\star, k \in [\kin-1] \}$ be given
	by~\autoref{algo:DIANASPIDEREM} run with any
	$\alpha \le 1/(1+\omega)$, and $ \lbatch\geq 1 $, with $V_{1,0,i} = \mf_i(\hatS_{1,0})$
	for any $i \in [n]^\star$.  Let $(\tau, K)$ be a uniform random variable on
	$[\kouter]^\star \times [\kin-1]$, independent of
	$\{\hatS_{t,k}, t \in [\kouter]^\star, k \in [\kin-1]\}$. Then, it
	holds
		\[
		v_{\min} \left( 1 - \pas \Lambda_\star \right) \PE\left[\|
		H_{\tau,K+1}\|^2 \right] \leq \pas^{-1} \kin^{-1} \kouter^{-1}
		\left( \PE\left[\lyap (\hatS_{1,0}) \right] - \min \lyap \right)
		\eqsp,
		\]
		where
		\[
		\Lambda_\star \eqdef \frac{L_{\dot \lyap}}{2 v_{\min}} + 2 \sqrt{2} \frac{v_{\max}}{v_{\min}}  \frac{L}{\sqrt{n}\alpha} \left( \omega + \frac{\kin \alpha^2}{8\lbatch} (1+10 \omega)\right)^{1/2} \eqsp.
		\]
\end{proposition} 
The proof of  \Cref{theo:DS} from \Cref{prop:DS} (which corresponds to particular choices of $ \lbatch, \alpha $, etc. is detailed in \Cref{subsec:proofthDS}).

\subsubsection{Control of  $H_{\tau, K}$}
Let $t \in [\kouter]^\star$ and $k \in [\kin-1]$. By
\Cref{hyp:DS:lyap}, we have
\[
\lyap (\hatS_{t,k+1}) \leq \lyap (\hatS_{t,k}) + \pscal{\nabla
  \lyap(\hatS_{t,k})}{\hatS_{t,k+1} - \hatS_{t,k}} + \frac{L_{\dot
    \lyap}}{2} \| \hatS_{t,k+1} - \hatS_{t,k}\|^2 \eqsp.
\]
Since $\hatS_{t,k+1} - \hatS_{t,k} = \pas_{t,k+1} H_{t,k+1}$, we have
using again \Cref{hyp:DS:lyap}
\[
\lyap (\hatS_{t,k+1}) \leq \lyap (\hatS_{t,k}) - \pas_{t,k+1}
\pscal{B(\hatS_{t,k}) \mf(\hatS_{t,k})}{H_{t,k+1}} + \frac{L_{\dot
    \lyap}}{2} \pas_{t,k+1}^2 \| H_{t,k+1}\|^2 \eqsp.
\]
We have the inequality, for any $\beta >0$:
\[
-\pscal{Bh}{H} \leq - \pscal{BH}{H} - \pscal{B(h-H)}{H} \leq -
\pscal{BH}{H} + \frac{\beta^2}{2} \|H\|^2 + \frac{1}{2\beta^2} \|B(H-h)\|^2 \eqsp.
\]
By \Cref{hyp:DS:lyap} again, this inequality yields for any
$\beta_{t,k+1} >0$ after applying the conditional expectation
\begin{multline}
\CPE{\lyap (\hatS_{t,k+1})  }{ \F_{t,0} } \leq \CPE{\lyap (\hatS_{t,k}) }{\F_{t,0} }- \pas_{t,k+1} v_{\min}
\Lambda_{t,k+1} \CPE{\| H_{t,k+1}\|^2 }{ \F_{t,0} }
\\ +
\frac{\pas_{t,k+1}}{2 \beta^2_{t,k+1}} v_{\max}^2 \CPE{\|H_{t,k+1} -
\mf(\hatS_{t,k}) \|^2 }{ \F_{t,0} } \eqsp, \label{eq:DS:maintheo:tool1}
\end{multline}
where
\[
\Lambda_{t,k+1} \eqdef 1 - \pas_{t,k+1}\frac{L_{\dot \lyap}}{2
  v_{\min} } - \frac{\beta^{2}_{t,k+1}}{2 v_{\min}} \eqsp.
\]
By \eqref{eq:DS:maintheo:tool1} and \Cref{prop:DS:Hvariance}, it holds
\begin{multline}
\PE\left[ \lyap (\hatS_{t,k+1})  \vert \F_{t,0} \right] \leq \PE\left[\lyap (\hatS_{t,k}) \vert \F_{t,0} \right]- \pas_{t,k+1} v_{\min}
\Lambda_{t,k+1} \PE\left[\| H_{t,k+1}\|^2 \vert \F_{t,0} \right]
\\
+ \frac{\pas_{t,k+1}}{2 \beta^2_{t,k+1}} v_{\max}^2 \frac{L^2}{n
     \lbatch} \sum_{\ell=1}^k \pas_{t,\ell}^2 \PE\left[ \|
     H_{t,\ell}\|^2 \vert \F_{t,0} \right] \\ +
\frac{\pas_{t,k+1}}{2 \beta^2_{t,k+1}} v_{\max}^2 \frac{\omega}{n} \frac{1}{n} \sum_{i=1}^n \PE\left[
  \|\Delta_{t,k+1,i}\|^2 \vert \F_{t,0} \right]   \eqsp. \label{eq:DS:maintheo:tool2}
\end{multline}
Set 
\[
G_{t,k} \eqdef \frac{1}{n}\sum_{i=1}^n \|V_{t,k,i} -
\mf_i(\hatS_{t,k}) \|^2 \eqsp.
\]
From \Cref{prop:DS:Delta}, we obtain
\begin{multline}
  \PE\left[ \lyap (\hatS_{t,k+1}) \vert \F_{t,0} \right] \leq
  \PE\left[\lyap (\hatS_{t,k}) \vert \F_{t,0} \right]- \pas_{t,k+1}
  v_{\min} \Lambda_{t,k+1} \PE\left[\| H_{t,k+1}\|^2 \vert \F_{t,0}
  \right]
  \\
  + \frac{\pas_{t,k+1}}{2 \beta^2_{t,k+1}} v_{\max}^2 \frac{L^2}{n
    \lbatch} (1+2 \omega) \sum_{\ell=1}^k \pas_{t,\ell}^2 \PE\left[ \|
    H_{t,\ell}\|^2 \vert \F_{t,0} \right] + \frac{\pas_{t,k+1}}{
    \beta^2_{t,k+1}} v_{\max}^2 \frac{\omega}{n} \CPE{ G_{t,k}}{
    \F_{t,0} } \eqsp. \label{eq:tool892}
\end{multline}
Assume that $k \mapsto \pas_{t,k+1}/\beta^2_{t,k+1}$ is a non-increasing sequence and set

\begin{align}
C_{t,k+1} \eqdef \frac{2 \omega}{\alpha n} v^2_{\max} \frac{\pas_{t,k+1}}{\beta^2_{t,k+1}} \eqsp.\label{eq:aux2}
\end{align}

From \Cref{prop:DS:driftV}, since $\alpha \in \ocint{0,1/(1+\omega)}$, we have 
\begin{multline}
  C_{t,k+1} \CPE{G_{t,k+1}}{\F_{t,0}} \leq (1-\alpha/2) C_{t,k+1}
  \CPE{G_{t,k}}{ \F_{t,0}} + \frac{2}{\alpha} L^2 \pas_{t,k+1}^2
  C_{t,k+1} \CPE{ \|H_{t,k+1}\|^2}{ \F_{t,0}} \\ + 2 \alpha
  \frac{L^2}{\lbatch} C_{t,k+1} \sum_{\ell=1}^k \pas_{t,\ell}^2
  \CPE{\|H_{t,\ell}\|^2 }{ \F_{t,0} } \eqsp. \label{eq:tool893}
\end{multline}
Upon noting that by definition of $C_{t,k+1}$ we have  (remember that $C_{t,k+1} \leq C_{t,k}$)
\[
(1-\alpha/2) C_{t,k+1} -C_{t,k} + \frac{\pas_{t,k+1}}{
    \beta^2_{t,k+1}} v_{\max}^2 \frac{\omega}{n}  \leq 0 \eqsp,
\]
this yields from \eqref{eq:tool892} and \eqref{eq:tool893}
\begin{multline*}
 \PE\left[ \lyap (\hatS_{t,k+1}) \vert \F_{t,0} \right] + C_{t,k+1}  \CPE{G_{t,k+1}}{\F_{t,0}}   \leq
  \PE\left[\lyap (\hatS_{t,k}) \vert \F_{t,0} \right] + C_{t,k}  \CPE{G_{t,k}}{\F_{t,0}}   \\
- \left( \pas_{t,k+1}
  v_{\min} \Lambda_{t,k+1} -\frac{2}{\alpha} L^2 \pas_{t,k+1}^2
  C_{t,k+1} \right) \PE\left[\| H_{t,k+1}\|^2 \vert \F_{t,0}
  \right]
  \\
  + \left( \frac{\pas_{t,k+1}}{2 \beta^2_{t,k+1}} v_{\max}^2 \frac{L^2}{n
    \lbatch} (1+2 \omega) + 2 \alpha
  \frac{L^2}{\lbatch} C_{t,k+1} \right) \sum_{\ell=1}^k \pas_{t,\ell}^2 \PE\left[ \|
    H_{t,\ell}\|^2 \vert \F_{t,0} \right]  \eqsp.
\end{multline*}
Let us restrict the computations to the case $\pas_{t,k} = \pas$,
$\beta_{t,k}= \beta$ (which implies $C_{t,k+1} = C_{t,k} =: C$); we
obtain
\begin{multline*}
  \pas v_{\min} \left( 1 - \pas\frac{L_{\dot \lyap}}{2 v_{\min} } -
    \frac{\beta^{2}}{2 v_{\min}} - \frac{\pas^2}{\beta^2} \frac{4
      v^2_{\max}}{v_{\min}} L^2 \frac{\omega}{\alpha^2 n} \right)
  \PE\left[\| H_{t,k+1}\|^2 \vert \F_{t,0}
  \right] \\
  \leq \PE\left[\lyap (\hatS_{t,k}) \vert \F_{t,0} \right] + C
  \CPE{G_{t,k}}{\F_{t,0}} - \PE\left[ \lyap (\hatS_{t,k+1}) \vert
    \F_{t,0} \right] - C \CPE{G_{t,k+1}}{\F_{t,0}}
  \\
  + \frac{\pas^3}{2 \beta^2} v_{\max}^2 \frac{L^2}{n
    \lbatch} \left( 1+10 \omega \right) \sum_{\ell=1}^k
  \PE\left[ \| H_{t,\ell}\|^2 \vert \F_{t,0} \right]
  \eqsp.
\end{multline*}
We now sum from $k=0$ to $k= \kin-1$ and divide by $\kin$:
\begin{multline*}
  \pas v_{\min} \left( 1 - \pas\frac{L_{\dot \lyap}}{2 v_{\min} } -
    \frac{\beta^{2}}{2 v_{\min}} - \frac{\pas^2}{\beta^2} \frac{4
      v^2_{\max}}{v_{\min}} L^2 \frac{\omega}{\alpha^2 n} \right)
\frac{1}{\kin} \sum_{k=1}^{\kin} \PE\left[\| H_{t,k}\|^2 \vert \F_{t,0}
  \right] \\
  \leq \kin^{-1} \PE\left[\lyap (\hatS_{t,0}) \vert \F_{t,0} \right] + \frac{C}{\kin}
  \CPE{G_{t,0}}{\F_{t,0}} \\
 - \kin^{-1} \PE\left[ \lyap (\hatS_{t, \kin}) \vert
    \F_{t,0} \right] - \frac{C}{\kin} \CPE{G_{t, \kin}}{\F_{t,0}}
  \\
  + \frac{\pas^3}{2 \beta^2} v_{\max}^2 \frac{L^2}{n
    \lbatch} \left( 1+10 \omega \right) \sum_{k=1}^{\kin}\PE\left[ \| H_{t,k}\|^2 \vert \F_{t,0} \right]
  \eqsp.
\end{multline*}
As a conclusion, we have 
\begin{align*}
  \pas v_{\min} \left( 1 - \pas\frac{L_{\dot \lyap}}{2 v_{\min} } -  \pas \bar \Lambda \right)
 & \frac{1}{\kin} \sum_{k=0}^{\kin-1} \PE\left[\| H_{t,k+1}\|^2 \vert \F_{t,0}
  \right] \\
  &\leq \kin^{-1} \PE\left[\lyap (\hatS_{t,0}) \vert \F_{t,0} \right] + \frac{C}{\kin}
  \CPE{G_{t,0}}{\F_{t,0}} \\
  &- \kin^{-1} \PE\left[ \lyap (\hatS_{t, \kin}) \vert
  \F_{t,0} \right] -\frac{C}{\kin} \CPE{G_{t, \kin}}{\F_{t,0}}
  \eqsp.
\end{align*}
where
\[
\bar \Lambda \eqdef \frac{\beta^{2}}{2 v_{\min} \pas} +
\frac{\pas}{\beta^2} \frac{4 v^2_{\max}}{v_{\min}} L^2
\frac{\omega}{\alpha^2 n} + \frac{\pas}{2 \beta^2} \frac{v_{\max}^2}{v_{\min}} \frac{L^2 \kin}{n
    \lbatch} \left( 1+10 \omega \right) \eqsp.
\]
Next, we sum from $t=1$ to $t = \kouter$, divide by $\kouter$.  
\begin{align}
\pas v_{\min} \left( 1 - \pas\frac{L_{\dot \lyap}}{2 v_{\min} } - \pas \bar \Lambda \right)
&    \frac{1}{\kouter \kin }	 \sum_{k=1}^{\kouter}  \sum_{k=1}^{\kin} \PE\left[\| H_{t,k+1}\|^2 
\right] \nonumber\\
&\leq \kin^{-1} \kouter^{-1}  \left( \PE\left[\lyap (\hatS_{1,0})  \right]  - \min \lyap \right) + \frac{C}{\kin \kouter}
\PE\left[G_{1,0} \right]  \label{eq:auxi_cor}
\eqsp.
\end{align}
Finally, we apply the expectation, with $(\tau, K)$ a uniform random
variable on $[\kouter]^\star \times [\kin-1]$, independent of
$\{\hatS_{t,k}, t \in [\kouter]^\star, k \in [\kin-1]\}$, upon noting
that $G_{t,\kin} = G_{t+1,0}$ and $\hatS_{t,\kin} = \hatS_{t+1,0}$,
this yields
\begin{align}
  \pas v_{\min} \left( 1 - \pas\frac{L_{\dot \lyap}}{2 v_{\min} } - \pas \bar \Lambda \right)
 &  \PE\left[\| H_{\tau,K+1}\|^2 
  \right] \nonumber\\
  &\leq \kin^{-1} \kouter^{-1}  \left( \PE\left[\lyap (\hatS_{1,0})  \right]  - \min \lyap \right) + \frac{C}{\kin \kouter} 
  \PE\left[G_{1,0} \right] \label{eq:mainresult}
  \eqsp.
\end{align}

\paragraph{Impact of initialization.} With
$V_{1,0,i} = \mf_i(\hatS_{1,0})$ for any $i \in [n]^\star$, we have
$ G_{1,0} =0 $. 

\paragraph{Choice of $ \beta $.} The latter inequality is true for all
parameter $ \beta^2 >0 $ (coming from Young's inequality). We can thus
optimize the value of $ \beta^2 $ to minimize the value of
$ \bar \Lambda $.  We here discuss this choice. First, to ensure that
$ \bar\Lambda $ is independent of $ \gamma $, we introduce $ \pa $,
and set $\beta^2 = \pa \pas$ so that 
\begin{align*}
\bar \Lambda  &=  \frac{\pa}{2 v_{\min}} +
 \frac{1}{\pa} \frac{4 v^2_{\max}}{v_{\min}} L^2
\frac{\omega}{\alpha^2 n} + \frac{1}{2 \pa} \frac{v_{\max}^2}{v_{\min}} \frac{L^2 \kin}{n
    \lbatch} \left( 1+10 \omega \right) \\ &=  \frac{\pa}{2 v_{\min}} +  \frac{4}{ \pa}  \frac{v_{\max}^2}{v_{\min}} \frac{L^2}{n
    \alpha^2} \left(\omega+ \frac{\kin \alpha^2}{8\lbatch} (1+10 \omega)\right)   \eqsp.
\end{align*}
Next, we optimize the value of $ \pa $.\footnote{Remark that this
  optimization step is crucial to optimize the dependency of
  $ \bar \Lambda $ w.r.t. $ \omega $: this ensures that
  $ \bar \Lambda \varpropto \omega^{3/2} $.}  Upon noting that
$\pa \mapsto A\pa + B/\pa$ (for $A,B >0$) is lower bounded by
$2 \sqrt{AB}$ and its minimizer is $\pa_\star \eqdef\sqrt{B/A}$, we
choose
\[
\pa_\star \eqdef  2 \sqrt{2} v_{\max} \frac{L}{\sqrt{n} \alpha}\left(\omega+ \frac{\kin \alpha^2}{8\lbatch} (1+10 \omega)  \right)^{1/2} \eqsp.
\]
and obtain 
\begin{align}
\bar \Lambda = 2 \sqrt{2} \frac{v_{\max}}{v_{\min}}  \frac{L}{\sqrt{n} \alpha} \left( \omega + \frac{\kin \alpha^2}{8\lbatch} (1+10 \omega)\right)^{1/2} \eqsp. \label{eq:barlambda}
\end{align}
Combining \Cref{eq:barlambda} and \Cref{eq:mainresult}, we obtain
	\[
v_{\min} \left( 1 - \pas \Lambda_\star \right) \PE\left[\|
H_{\tau,K+1}\|^2 \right] \leq \pas^{-1} \kin^{-1} \kouter^{-1}
\left( \PE\left[\lyap (\hatS_{1,0}) \right] - \min \lyap \right)
\eqsp,
\]
where
\begin{align}
\Lambda_\star \eqdef \frac{L_{\dot \lyap}}{2 v_{\min}} + 2 \sqrt{2} \frac{v_{\max}}{v_{\min}}  \frac{L}{\sqrt{n} \alpha} \left( \omega + \frac{\kin \alpha^2}{8\lbatch} (1+10 \omega)\right)^{1/2} \eqsp, \label{eq:lambdabaropt}
\end{align}
which is the result of \Cref{prop:DS}.

\subsection{Proof of \Cref{theo:DS} (\Cref{eq:thm-VR-conv}) from
  \Cref{prop:DS}}
\label{subsec:proofthDS}

We apply \Cref{prop:DS} with: 	$ \lbatch \eqdef \lceil\frac{\kin}{(1+\omega)^2} \rceil$ and the largest possible learning rate $ \alpha = (1+\omega) ^{-1}$: this gives in \Cref{eq:lambdabaropt}
\begin{align*}
\Lambda_\star & = \frac{L_{\dot \lyap}}{2 v_{\min}} + 2 \sqrt{2} \frac{v_{\max}}{v_{\min}}  \frac{L}{\sqrt{n}} (1+\omega) \left( \omega + \frac{1+10 \omega}{8}\right)^{1/2} \\
& = \frac{L_{\dot \lyap}}{2 v_{\min}}  \left (1 +  4 \sqrt{2} \frac{v_{\max}}{L_{\dot \lyap}}  \frac{L}{\sqrt{n}} (1+\omega) \left( \omega + \frac{1+10 \omega}{8}\right)^{1/2} \right ) \eqsp.
\end{align*}
Next, we choose $ \pas $ to be the largest possible value to ensure  $ \left( 1 - \pas \Lambda_\star \right) \geq \frac{1}{2}$. For all $ t,k $, 
\[
\pas_{t,k} = \pas \eqdef \frac{1}{2\Lambda_\star } = \frac{v_{\min}}{L_{\dot \lyap}}  \left (1 +  4 \sqrt{2} \frac{v_{\max}}{L_{\dot \lyap}}  \frac{L}{\sqrt{n}} (1+\omega) \left( \omega + \frac{1+10 \omega}{8}\right)^{1/2} \right )^{-1}. \] 
This gives the first result of \Cref{theo:DS}, namely \Cref{eq:thm-VR-conv}. We give the proof of the second result, \Cref{eq:thm-VR-control} in the following subsection.

\subsection{Proof of \Cref{theo:DS} (\Cref{eq:thm-VR-control}): control on $\mf(\hatS_{\tau,K})$}
We now establish~\eqref{eq:thm-VR-control} for $\pas_{t,k} = \pas$.
Let $t \in [\kouter]^\star$ and $k \in [\kin-1]$. We have
\begin{equation} \label{eq:tool625}
\| \mf(\hatS_{t,k}) \|^2 \leq 2 \|\PE\left[ H_{t,k+1} \vert \F_{t,k}
  \right] \|^2 + 2 \| \mf(\hatS_{t,k}) - \PE\left[H_{t,k+1} \vert \F_{t,k}
  \right] \|^2 \eqsp.
\end{equation}
Let us consider the first term in \eqref{eq:tool625}. By Jensen's inequality and the tower property of conditional expectations
\[
\PE\left[ \|\PE\left[ H_{t,k+1} \vert \F_{t,k} \right] \|^2 \vert
  \F_{t,0} \right] \leq \PE\left[ \PE\left[ \| H_{t,k+1} \|^2 \vert
    \F_{t,k} \right] \vert \F_{t,0} \right] = \PE\left[ \| H_{t,k+1}
  \|^2 \vert \F_{t,0} \right] \eqsp.
\]
Let us now consider the second term in \eqref{eq:tool625}.  By
\Cref{prop:DS:biasH} and \Cref{prop:DS:Hvariance}, we have
\[
\PE\left[\| \PE\left[ H_{t,k+1} \vert \F_{t,k} \right] -
  \mf(\hatS_{t,k}) \|^2 \vert \F_{t,0} \right] \leq \left\{
\begin{array}{ll}
  \pas^2 \frac{L^2 }{n \lbatch} \sum_{\ell=1}^{k-1} \PE\left[ \|
    H_{t,\ell} \|^2 \vert \F_{t,0} \right] & \text{when $k \geq 2$} \\ 0 &
  \text{when $k \in \{0,1\}$ \eqsp.}
\end{array} \right.
\]
Therefore, we write
\begin{align*}
\PE\left[ \| \mf(\hatS_{t,k}) \|^2 \right] \leq 2 \PE\left[
  \|H_{t,k+1}\|^2 \right]+ 2 \pas^2 \frac{L^2}{n \lbatch}
\sum_{\ell=1}^{k-1} \PE\left[ \| H_{t,\ell} \|^2\right]
  \end{align*}
We now sum from $k=0$ to $k = \kin-1$, then from $t=1$ to
$t=\kouter$, and finally we divide by $\kin \kouter$. This yields
\begin{align*}
\PE\left[ \| \mf(\hatS_{\tau,K}) \|^2 \right] & \leq 2 \PE\left[
  \|H_{\tau,K+1}\|^2 \right]+ 2 \pas^2 \frac{L^2}{n \lbatch}
\frac{1}{\kin \kouter} \sum_{t=1}^{\kouter}
\sum_{k=2}^{\kin-1}\sum_{\ell=1}^{k-1} \PE\left[ \| H_{t,\ell}
  \|^2\right] \\ & \leq 2 \PE\left[ \|H_{\tau,K+1}\|^2 \right]+ 2
\pas^2 \frac{L^2}{n \lbatch} \frac{1}{\kouter} \sum_{t=1}^{\kouter}
\sum_{k=1}^{\kin-2}\PE\left[ \| H_{t,k} \|^2\right] \\ & \leq 2
\PE\left[ \|H_{\tau,K+1}\|^2 \right]+ 2 \pas^2 \frac{L^2}{n}
\frac{\kin}{\lbatch} \PE\left[ \| H_{\tau,K+1} \|^2\right] \\ & \leq 2
\left(1+ \pas^2 \frac{L^2}{n} \frac{\kin}{\lbatch} \right) \PE\left[
  \| H_{\tau,K+1} \|^2\right] \eqsp.
\end{align*}

\subsection{On the convergence of the $V_{t,k,i}$'s}
In this subsection, we provide a complementary result, to support the
assertion made in the text, that the variable $V_{t,k,i}$
approximates $\mf_i(\hatS_{t,k})$. Recall that for
$t \in [\kouter]^\star$ and $k \in [\kin]$,
$ G_{t,k} \eqdef \frac{1}{n}\sum_{i=1}^n \|V_{t,k,i} -
\mf_i(\hatS_{t,k}) \|^2 \eqsp.$
\begin{proposition}\label{cor:control_var}
  When running~\autoref{algo:DIANASPIDEREM} with a constant step size
  $ \pas$ equal to
\[
\pas \eqdef \frac{v_{\min}}{L_{\dot \lyap}} \left (1 + 4 \sqrt{2}
  \frac{v_{\max}}{L_{\dot \lyap}} \frac{L}{\sqrt{n}} (1+\omega) \left(
    \omega + \frac{1+10 \omega}{8}\right)^{1/2} \right )^{-1} \eqsp, 
\] with $ \lbatch \eqdef \lceil\frac{\kin}{(1+\omega)^2} \rceil$ and
$ \alpha \eqdef 1/(\omega+1)$, we have
\[ 
\frac{ 1}{ \kouter \kin }	 \sum_{t=1}^{\kouter}  \sum_{k=1}^{\kin} 
            \PE[G_{t,k}] \leq 
            \frac{2(1+\omega)}{ \kin \kouter}	\PE[{G_{1,0}}] + 16 \frac{\pas}{\kin \kouter}  \frac{(1+\omega)^2 L^2}{ v_{\min}} 
            \left( \PE\left[\lyap (\hatS_{1,0})  \right]  - \min \lyap  \right)\eqsp.
            \]
      \end{proposition}

In words, the Cesaro average
$ \frac{ 1}{ \kouter \kin } \sum_{t=1}^{\kouter} \sum_{k=1}^{\kin}
\PE[G_{t,k}] $
decreases proportionally to the number of iterations $ \kin \kouter
$.
Consequently, the average squared distance between $V_{t,k,i} $ and
$ \mf_i(\hatS_{t,k})$ (i.e., $ G_{t,k} $), converges to 0 in the sense
of Cesaro.
\begin{proof}
  From \Cref{prop:DS:driftV}, we have that, $t \in [\kouter]^\star$
  and $k \in [\kin]$, and any $ \alpha \le (\omega+1)^{-1} $:
	\begin{align*}
	\CPE{G_{t,k+1}}{ \F_{t,0}} & \leq (1-\alpha/2)
	\CPE{G_{t,k}}{ \F_{t,0}} \\ &  + \frac{2}{\alpha} L^2
	\pas_{t,k+1}^2  \CPE{ \|H_{t,k+1}\|^2}{ \F_{t,0}} + 2
	\alpha \frac{L^2}{\lbatch} \sum_{\ell=1}^k \pas_{t,\ell}^2
	\CPE{\|H_{t,\ell}\|^2 }{ \F_{t,0} }\eqsp.
	\end{align*}
	Equivalently:
		\begin{align*}
\alpha/2 
\CPE{G_{t,k}}{ \F_{t,0}}& \leq 
	\CPE{G_{t,k}}{ \F_{t,0}} - 	\CPE{G_{t,k+1}}{ \F_{t,0}} \\ &  + \frac{2}{\alpha} L^2
	\pas_{t,k+1}^2  \CPE{ \|H_{t,k+1}\|^2}{ \F_{t,0}} + 2
	\alpha \frac{L^2}{\lbatch} \sum_{\ell=1}^k \pas_{t,\ell}^2
	\CPE{\|H_{t,\ell}\|^2 }{ \F_{t,0} }\eqsp.
	\end{align*}
Summing from $k=0$ to $k= \kin-1$, we get, with $ \pas_{t,k+1}^2  =\pas$:
	\begin{align*}
\frac{\alpha}{2}  \sum_{k=0}^{\kin-1} 
\CPE{G_{t,k}}{ \F_{t,0}}& \leq 
\CPE{G_{t,0}}{ \F_{t,0}} - 	\CPE{G_{t,\kin}}{ \F_{t,0}} \\ &  + \frac{2}{\alpha} L^2
\pas^2  \sum_{k=1}^{\kin}  \CPE{ \|H_{t,k}\|^2}{ \F_{t,0}} + 2
\alpha \frac{L^2}{\lbatch} \sum_{k=1}^{\kin-1} \sum_{\ell=1}^k \pas^2
\CPE{\|H_{t,\ell}\|^2 }{ \F_{t,0} }\eqsp \\
& \leq 
\CPE{G_{t,0}}{ \F_{t,0}} - 	\CPE{G_{t,\kin}}{ \F_{t,0}} \\ &  + \frac{2}{\alpha} L^2
\pas^2  \sum_{k=1}^{\kin}  \CPE{ \|H_{t,k}\|^2}{ \F_{t,0}} + 2
\alpha \frac{L^2 \kin}{\lbatch} \sum_{k=1}^{\kin}  \pas^2
\CPE{\|H_{t,k}\|^2 }{ \F_{t,0} }\eqsp\\
& \leq 
\CPE{G_{t,0}}{ \F_{t,0}} - 	\CPE{G_{t,\kin}}{ \F_{t,0}} \\ &  + \frac{2}{\alpha} L^2
\pas^2 \left(1 + \frac{\alpha^2 \kin }{\lbatch}\right )  \sum_{k=1}^{\kin}  \CPE{ \|H_{t,k}\|^2}{ \F_{t,0}} \eqsp.
	\end{align*}
	Summing from $t=1$ to $t = \kouter$, dividing by $ \kouter
        \kin $, and taking expectation we get:
		\begin{align*}
                  \frac{1}{\kouter \kin }	 \sum_{t=1}^{\kouter}  \sum_{k=0}^{\kin-1} 
                  \PE[G_{t,k}]
                  & \leq 
                    \frac{2}{\alpha \kouter \kin}	\PE[G_{1,0}]\\ &  + \frac{4}{\alpha^2 \kouter \kin} L^2
                                                                         \pas^2 \left(1 + \frac{\alpha^2 \kin }{\lbatch}\right )  \sum_{t=1}^{\kouter}  \sum_{k=1}^{\kin}  \PE[ \|H_{t,k}\|^2] \eqsp.
	\end{align*}
        We used that $G_{t,\kin} = G_{t+1,0}$. By denoting $(\tau, K)$
        a uniform random variable on $[\kouter]^\star \times [\kin
          -1]$ -- independent of the path $\{\hatS_{t,k} , t \in
        [\kouter]^\star, k \in [\kin] \}$, we have
\[
\PE[G_{\tau,K}] \leq \frac{2}{\alpha \kouter \kin} \PE[G_{1,0}] +
\frac{4}{\alpha^2} L^2 \pas^2 \left(1 + \frac{\alpha^2 \kin
  }{\lbatch}\right ) \PE[ \|H_{\tau,K+1}\|^2] \eqsp.
\]
From \Cref{theo:DS}, this yields (note that $\alpha = (1+\omega)^{-1}$
and $\lbatch \geq \kin / (1+\omega)^2$)
\[
\PE[G_{\tau,K}] \leq \frac{2(1+\omega)}{ \kouter \kin} \PE[G_{1,0}] +
\pas \frac{16 (1+\omega)^2L^2}{v_{\min} \kin \kouter} \left(
  \lyap(\hatS_{1,0}) - \min \lyap\right) \eqsp.
\] 
\end{proof}

\section{Supplement to the numerical section}
\label{app:numerical}

This section gathers additional details concerning the models used in our numerical experiments. Namely, \Cref{app:subsec:num:GMM} presents the full derivations for the \FEDEM~algorithm for finite Gaussian Mixture Models, and \Cref{app:fedmiss} provides the detailed pseudo-code for the \texttt{FedMissEM} algorithm for federated missing values imputation introduced in \Cref{sec:numerical} and provides the necessary information to request access to the data we used on the eBird platform~\cite{eBird}.

\subsection{Gaussian Mixture Model}
\label{app:subsec:num:GMM}
Let $y_1, \ldots, y_N$ be $N$ $\rset^p$-valued observations; they are
modeled as the realization of a vector $(Y_1, \ldots, Y_N)$ with
distribution defined as follows:
\begin{itemize}
\item conditionally to a $\{1, \ldots, L\}$-valued vector of random
  variables $(Z_1, \ldots, Z_N)$, $(Y_1, \ldots, Y_N)$ are
  independent; and the conditional distribution of $Y_i$ is
  $\mathcal{N}_p(\mu_{Z_i}, \Sigma)$.
 \item the r.v.  $(Z_1, \ldots, Z_n)$ are i.i.d. with multinomial distribution
  of size $1$ and with probabilities $\pi_1, \ldots, \pi_L$.
\end{itemize}
Equivalently, the random variables $(Y_1, \ldots, Y_N)$ are
independent with distribution $\sum_{\ell=1}^L \pi_\ell \,
\mathcal{N}_p(\mu_\ell, \Sigma)$. For $1\leq i\leq N$, the negative
log-likelihood of the observation $Y_i$ is given up to an additive
constant term by
$$\param \mapsto \frac{1}{2} \ln \mathrm{det} \Sigma + \frac{1}{2} \pscal{Y_i Y_i^\top}{\Sigma^{-1}}- \ln \sum_{z=1}^L \exp\left(\pscal{s(Y_i,z)}{\phi(\param)} \right)$$
where, denoting $\mathsf{1}_{\{l\}}(z)$ the indicator function equal to $1$ if $z=l$ and $0$ otherwise:
\begin{equation}
  \label{supp:eq:si}
  s(y,z) \eqdef  \left(\begin{array}{c}
  \mathsf{1}_{\{1\}}(z) \\
  \vdots \\
  \mathsf{1}_{\{L\}}(z) \\
y\mathsf{1}_{\{1\}}(z) \\
  \vdots \\
y\mathsf{1}_{\{L\}}(z)
  \end{array} \right) \eqsp, \qquad
  \phi(\theta) \eqdef \left(\begin{array}{c} \log(\pi_1) - \frac{1}{2}
    \mu_1^\top\Sigma^{-1}\mu_1 \\ \vdots \\ \log(\pi_L) -
    \frac{1}{2}\mu_L^\top\Sigma^{-1}\mu_L \\ \Sigma^{-1}\mu_1
    \\ \vdots \\ \Sigma^{-1}\mu_L
  \end{array} \right) \eqsp.
  \end{equation}
  The goal is to estimate the parameter $\param \eqdef
  (\pi_1,\ldots,\pi_L,\mu_1,\ldots,\mu_L,\Sigma)$ by minimizing the
  normalized negative log-likelihood:
\begin{equation}
\label{eq:GGM-optim}
 F(\theta) \eqdef \frac{1}{2} \ln \mathrm{det} \Sigma + \frac{1}{2}
 \pscal{\frac{1}{N} \sum_{i=1}^N Y_i Y_i^\top}{\Sigma^{-1}}-
 \frac{1}{N} \sum_{i=1}^N \ln \int
 \exp\left(\pscal{s(Y_i,z)}{\phi(\param)} \right) \nu(\rmd z)
\end{equation}
where $\nu$ is the counting measure on $\{1, \ldots L\}$.

\paragraph{Classical EM algorithm}
We use the EM algorithm: in the Expectation (E) step, using the current value of the iterate $\theta_{\text{curr}}$, we compute a majorizing function $\theta \mapsto \mathsf{Q}(\theta, \theta_{\text{curr}})$ given up to an additive constant by
$$\mathsf{Q}(\theta, \theta_{\text{curr}}) = -\langle \bars ( \theta_{\text{curr}}), \phi(\theta) \rangle + \psi(\theta),$$
where
\[
\psi(\param) \eqdef \frac{1}{2} \ln \mathrm{det} \Sigma +
\frac{1}{2} \pscal{\frac{1}{N} \sum_{i=1}^N Y_i
  Y_i^\top}{\Sigma^{-1}} \eqsp,
\]
$\bars(\theta_{\text{curr}}) \eqdef
\frac{1}{N}\sum_{i=1}^N\bars_{i}(\theta)$, and for any $i \in
     [N]^{\star}$, $\bars_{i}(\theta)$ is the conditional
     expectation of the complete data sufficient statistics:
\begin{equation}
\label{eq:bars_i}
\bars_{i}(\theta) = \left(\begin{array}{c} \bar{\rho}_{i,1}(\theta)
  \\ \vdots \\ \bar{\rho}_{i,L}(\theta) \\ Y_i\bar{\rho}_{i,1}(\theta)
  \\ \vdots \\ Y_i\bar{\rho}_{i,L}(\theta)
  \end{array} \right) \text{, where for } \ell \in [L]^{\star} \text{, } \bar{\rho}_{i,l}(\theta) \eqdef \frac{\pi_\ell \ \mathcal{N}_p(\mu_\ell,
  \Sigma)[Y_i]}{\sum_{u=1}^L \pi_u\
  \mathcal{N}_p(\mu_u, \Sigma)[Y_i]} \eqsp.
\end{equation}
In~\eqref{eq:bars_i}, $\mathcal{N}_p(\mu, \Sigma)[y]$ is the density
function of the distribution $\mathcal{N}_p(\mu, \Sigma)$ evaluated at
$y$.

In the optimization step (M-step), a new value of
$\theta_{\text{curr}}$ is computed as a minimizer of
$\theta\mapsto\mathsf{Q}(\theta,\theta_{\text{curr}})$. Let us now
detail this step.

  \begin{algorithm}[htbp]
    \caption{Classical EM algorithm for mixture of Gaussians}
  \begin{algorithmic}[1]
   \STATE {\bfseries Input:} $\kmax \in \nset$, $X$, $\hat{S}_0$, $\hat{\theta}_0$
   \STATE {\bfseries Output:} The sequence of statistics: $\{\hat{S}_{k}, k \in [\kmax]\}$; the sequence of parameters $\{\hat{\theta}_{k}, k \in [\kmax]\}$
   \FOR{$k=0, \ldots,
     \kmax-1$}
      \STATE { \textit{Expectation step}: compute conditional expectations given current parameter $\hat{\theta}_k$: Set $\hat{S}_{k+1} = \frac{1}{N}\sum_{i=1}^N\bar{s}_{i}(\hat\theta^k)$}
   \STATE {\textit{Maximization step}: update parameter $\hat{\theta}_{k+1}$ based on current statistics $\hat{S}_{k+1}$ according to update rule~\eqref{eq:update}}
   \ENDFOR
       \end{algorithmic}
  \end{algorithm}
  
\paragraph{The M step: the map $\map$.} Let
\[
s = (s^{(1)}, s^{(2)}) = (s^{(1),1}, \ldots, s^{(1),L}, s^{(2),1},
\ldots, s^{(2),L}) \in \rset^L \times \rset^{pL}\eqsp;
\]
we write $\pscal{s}{\phi(\param)} = \sum_{j=1}^2
\pscal{s^{(j)}}{\phi^{(j)}(\param)}$ where the functions $\phi^{(j)}$
are defined by
\begin{equation}
  \label{supp:eq:phi}
  \phi^{(1)}(\theta) \eqdef \left(\begin{array}{c}
  \log(\pi_1) - \frac{1}{2}\mu_1^\top\Sigma^{-1}\mu_1 \\
  \vdots \\
  \log(\pi_L)  - \frac{1}{2} \mu_L^\top\Sigma^{-1}\mu_L
  \end{array} \right) \eqsp, \qquad 
  \phi^{(2)}(\theta) \eqdef \left(\begin{array}{c} \Sigma^{-1}\mu_1
    \\ \vdots \\ \Sigma^{-1}\mu_L
  \end{array} \right) \eqsp.
  \end{equation}

By definition, $\map(s) = \mathrm{argmin}_{\param \in \Param}
-\pscal{s}{\phi(\param)}+\psi(\param)$. Here, this optimum is unique
and defined by $\map(s) = \{\pi_\ell(s), \mu_\ell(s), \ell=1, \ldots, L;
\Sigma\}$ with
\begin{align}
\label{eq:update}
  \pi_\ell(s) &\eqdef \frac{s^{(1),\ell}}{ \sum_{u=1}^L s^{(1),u}}
  \eqsp, \\ \mu_\ell(s) & \eqdef \frac{s^{(2),\ell}}{s^{(1),\ell}}
  \eqsp, \\ \Sigma(s) & \eqdef \frac{1}{N} \sum_{i=1}^N Y_i Y_i^\top -
  \sum_{\ell=1}^L s^{(1),\ell} \mu_\ell(s) \, \mu_\ell^\top(s) \eqsp.
\end{align}
The expressions of $\pi_\ell(s)$ and $\mu_\ell(s)$ are easily obtained. We
provide details for the covariance matrix. We have for any symmetric
matrix $H$
  \begin{align*}
  \ln \frac{\mathrm{det}(\Gamma+H)}{\mathrm{det}(\Gamma)} & = \ln
  \mathrm{det}(I + \Gamma^{-1}H) = \ln (1+ \mathrm{Tr}(\Gamma^{-1} H)
  +o(\|H\|)) \\ & = \mathrm{Tr}(\Gamma^{-1} H) +o(\|H\|) =
  \pscal{H}{\Gamma^{-1}} + o(\|H\|)
  \end{align*}
  thus showing that the derivative of $\Gamma \mapsto \ln
  \mathrm{det}\Gamma$ is $\Gamma^{-1}$.  $\map(s)$ depends on
  $\Sigma^{-1}$ through the function
  \[
\Sigma^{-1} \mapsto -\frac{1}{2} \ln \mathrm{det}(\Sigma^{-1}) +
\frac{1}{2} \pscal{\Sigma^{-1}}{ \frac{1}{N}\sum_{i=1}^N Y_i Y_i^\top}
+ \pscal{\Sigma^{-1}}{ \frac{1}{2} \sum_{\ell=1}^L s^{(1),\ell}
  \mu_\ell \mu_\ell^\top - \sum_{\ell=1}^L \mu_\ell \,
  (s^{(2),\ell})^\top}\eqsp.
\]
The optimum solves
\[\Sigma  =  \frac{1}{N}\sum_{i=1}^N Y_i Y_i^\top + \sum_{\ell=1}^L s^{(1),\ell} \mu_\ell \mu_\ell^\top - 2 \sum_{\ell=1}^L \mu_\ell \,
  (s^{(2),\ell})^\top
\]
Hence, $\Sigma(s)$ is this solution when $\mu_\ell \leftarrow
\mu_\ell(s)$ which yields the expression since $s^{(2),\ell} =
s^{(1),\ell} \mu_\ell(s)$.

\paragraph{In the federated setting.}
In the federated setting, the data is distributed across $n$ local
servers. For all $c\in[n]^{\star}$, the $c$-th server possesses a
local data set of size $N_c$; $N_c\geq 1$ and $\sum_{c=1}^nN_c =
N$. We write
\[
\bigcup_{i=1}^N \{Y_i\} = \bigcup_{c=1}^n \bigcup_{j=1}^{N_c}
\{Y_{cj}\} \eqsp,
\]
thus meaning that each local worker $\# c$ processes the data set
$\{Y_{c1}, \ldots, Y_{c N_c} \}$.

The computation of the map $\map$ requires the knowledge of a
statistic of the full data set, namely $N^{-1} \sum_{i=1}^N Y_i
Y_i^\top$. For this reason, we want the map $\map$ to be available at
the central server only. Since
\[
\sum_{i=1}^N Y_i = \sum_{c=1}^n \sum_{j=1}^{N_c} Y_{cj}
\]
this full sum can be computed during the initialization of the
algorithm by the central server, by using the $n$ local summaries
$\sum_{j=1}^{N_c} Y_{cj}$ sent by the local workers.

In the FL setting, we write the objective function as follows
\begin{align*}
\param & \mapsto \psi(\param) - \frac{1}{N} \sum_{c=1}^n
\sum_{j=1}^{N_c} \ln \int \exp\left(\pscal{s(Y_{cj},z)}{\phi(\param)}
\right) \nu(\rmd z) \\ & = - \frac{1}{N} \sum_{c=1}^n \ln
\prod_{j=1}^{N_c} \int \exp\left(\pscal{s(Y_{cj},z)}{\phi(\param)} -
\frac{N}{n N_c} \psi(\param) \right) \nu(\rmd z) \\ & \propto -
\frac{1}{n} \sum_{c=1}^n \ln \prod_{j=1}^{N_c} \int
\exp\left(\pscal{s(Y_{cj},z)}{ \phi(\param)} - \frac{N}{n N_c}
\psi(\param) \right) \nu(\rmd z) \eqsp.
\end{align*}
It is of the form \eqref{eq:problem} with $\R(\param) = 0$ and
\[
\loss{c}(\param) \eqdef - \ln \prod_{j=1}^{N_c} \int
\exp\left(\pscal{s(Y_{cj},z)}{ \phi(\param)} - \frac{N}{n N_c}
\psi(\param) \right) \nu(\rmd z) \eqsp.
\]
In the case $n N_c = N$ for any $c \in [n]^\star$, we have
\[
\loss{c}(\param) = - \sum_{j=1}^{N/n} \ln p(Y_{cj}; \param) \eqsp, \qquad 
\]
with
\[
p(y; \param) \eqdef \int p(y,z; \param) \, \nu(\rmd z) \qquad
p(y,z;\param) \eqdef \exp\left(\pscal{s(y,z)}{ \phi(\param)} -
\psi(\param) \right) \nu(\rmd z) \eqsp.
\]
$p(y,z;\param)$ is of the form \eqref{eq:complete-likelihood}; this
yields
\[
\bars_{cj}(\param) \eqdef \sum_{z=1}^L s(Y_{cj},z)
\bar{\rho}_{cj,z}(\theta) \eqsp, \qquad \bars_c(\param) \eqdef
\frac{n}{N} \sum_{j=1}^{N/n} \bars_{cj} \eqsp,
\]
where $\bar{\rho}_{cj,z}(\param)$ is defined by \eqref{eq:bars_i}.

 The pseudo code for the \FEDEM~algorithm is given in
 Algorithm~\ref{algo:fed-GMM-nocompress}.
\begin{algorithm}[htbp]
\caption{Federated EM algorithm for distributed GMM without compression}
\label{algo:fed-GMM-nocompress}
\label{algo:fedImput}
\begin{algorithmic}[1]
\STATE {\bfseries Input:} $\kmax \in \nset$; for $c \in [n]^\star$,
$V_{0,c} \in \rset^{L+pL}$; $\hatS_0 \in \rset^{L+pL}$; $\hat{\theta}_0 \in \rset^{L}\times(\rset^p)^L\times\rset^{p\times p}$; a positive sequence
   $\{\pas_{k+1}, k \in [\kmax-1]\}$; $\alpha$
\STATE {\bfseries Output:} The \FEDEM
   sequence: $\{\hatS_{k}, k \in [\kmax]\}$
\FOR{$k=0, \ldots,
     \kmax-1$}
     		\FOR{$c=1, \ldots, n$}
     				\STATE {\em (agent $\# i$, locally)}
  					 \STATE Sample a batch $\mathcal{I}_{k,c}\subset [N_c]$
   					\STATE Set $\Smem_{k+1,c} = \frac{1}{|\mathcal{I}_{k,c}|}\sum_{i\in\mathcal{I}_{k,c}}\bar{s}_i(\hat{\theta}_k)$, where $\bar{s}_i$ is defined in~\eqref{eq:bars_i}
   					\STATE Set $\Delta_{k+1,c} =\Smem_{k+1,c}  - \hatS_{k} - V_{k,c}$
   					\STATE Update $V_{k+1,c} = V_{k,c} + \alpha\, \Q{(\Delta_{k+1,c})}$
   					\STATE Send $\Q{(\Delta_{k+1,c})}$ to the controller
		\ENDFOR
 			\STATE {\em (the controller)}
		 \STATE Compute
   $H_{k+1} = V_k  + \frac{1}{n}\sum_{c=1}^n
   \Q{(\Delta_{k+1,c})}$
   			\STATE Set $\hatS_{k+1} = \hatS_k + \pas_{k+1}
   H_{k+1}$
  			 \STATE Set $V_{k+1} = V_k + \alpha n^{-1} \sum_{c=1}^n
   \Q(\Delta_{k+1,c})$.
   			\STATE Send $\hatS_{k+1}$ and $\hat{\theta}_{k+1} = \map
   (\hatS_{k+1})$ to the agents, where $\map(\hat{S}_{k+1})$ is given by the update rule~\eqref{eq:update}
   \ENDFOR
       \end{algorithmic}
\end{algorithm}

\subsection{Federated missing values imputation}
\label{app:fedmiss}

{\em $\bullet$ Model and the \texttt{FedMissEM} algorithm.}  $I$
observers participate in
the programme, there are $J$ ecological sites and $L$ time
stamps. Each observer $\# i$ provides a $J \times L$ matrix $X^i$ and
a subset of indices $\Omega^i \subseteq [J]^\star \times
[L]^\star$. For $j \in [J]^{\star}$ and $\ell \in [L]^{\star}$, the
variable $X_{j\ell}^i$ encodes the observation that would be collected
by observer $\# i$ if the site $\# j$ were visited at time stamp $\#
\ell$; since there are unvisited sites, we denote by $Y^i \eqdef
\{X^i_{j\ell}, (j,\ell)\in\Omega^i\}$ the set of observed values and
$Z^i \eqdef \{X^i_{j\ell}, (j,\ell)\notin\Omega^i\}$ the set of
unobserved values.  The statistical model is parameterized by a matrix
$\theta\in\rset^{J \times L}$, where $\theta_{j\ell}$ is a scalar
parameter characterizing the distribution of species individuals at
site $j$ and time stamp $\ell$. For instance, $\theta_{j\ell}$ is the
log-intensity of a Poisson distribution when the observations are
count data or the log-odd of a binomial model when the observations
are presence-absence data. This model could be extended to the case
observers $\# i$ and $\# i'$ count different number of specimens on
average at the same location and time stamp, because they
do not have access to the same material or do not have the same level
of expertise: heterogeneity between observers could be modeled by
using different parameters for each individual $\# i$ say
$\theta^{i}\in\rset^{J \times L}$.  Here, we consider the case when $
\theta_{j\ell}^{i} = \theta_{j\ell}$ for all $(j,\ell)\in
      [J]^\star\times[L]^\star$ and $i \in [I]^\star$.
      
We further assume that the entries $\{ X_{j\ell}^i, i \in [I]^\star,
  j \in [J]^\star, \ell \in [L]^\star\}$ are independent with a
  distribution from an exponential family with respect to some
  reference measure $\nu$ on $\rset$ of the form: $x
  \mapsto \rho(x) \exp\{ x \theta_{j\ell} - \psi(\theta_{j\ell})\}.$
  The function $\psi$ is for instance defined by $\psi(\tau) =
  -\frac{1}{2} \tau^2$ for a Gaussian model with
    expectation $\tau$ and variance $1$, $\psi(\tau) =
  \log(1+\mathrm{e}^{\tau})$ for a Bernoulli model
with success probability $\tau$, and $\psi(\tau) =
  \mathrm{e}^{\tau}$ for a Poisson model with
    intensity $\tau$. Therefore, the joint distribution of $(Y^i,
  Z^i)$ is given by 
  $
p_i(y^i,z^i;\theta) \eqdef \Big(\prod_{(j,\ell) \in \Omega^i}
\rho(y^i_{j\ell}) \Big) \ \Big( \prod_{(j,\ell) \notin \Omega^i}
\rho(z^i_{j\ell}) \Big) \exp\Big( \pscal{s_i(y^i,z^i)}{\theta} -
\sum_{j\ell} \psi(\param_{j\ell})\Big) \eqsp;
$
where $s_i(Y^i,Z^i)$ is a $J \times L$ matrix with entry $\# (j,\ell)$
given by $Y^i_{j\ell}$ if $(j,\ell) \in \Omega^i$ and $Z^i_{j,\ell}$
otherwise.

In order to estimate the unknown matrix $\theta\in\rset^{J\times L}$,
we assume that $\theta$ is low-rank; we use the parameterization
$\theta = U V^{\top}$, where $U\in\rset^{J\times r}$ and
$V\in\rset^{L\times r}$ with $\operatorname{rank}(\theta) = r$ and
$r<\min(J,L)$.  The estimator is defined as a minimizer of the
negative penalized log-likelihood:
$\min_{U\in \rset^{J\times r}, V\in \rset^{L\times r}} F(U,V)$,  with $ F(U,V) \eqdef \frac{1}{n}\sum_{i=1}^n \mathcal{L}^i(UV^{\top})
+ \frac{\lambda}{2}\left(\|U\|_F^2 +\|V\|_F^2\right),$ where for
$\theta\in\rset^{J\times L}$, $\mathcal{L}^i(\theta) \eqdef -\log \int
p_{i}(Y^i,z^i;\theta) \ \prod_{(j,\ell) \notin \Omega^i} \nu(\rmd
z^i_{j\ell})$.

\paragraph{\texttt{FedMissEM} algorithm.}
Algorithm~\ref{algo:fedImput} provides the pseudo-code for the Federated EM algorithm for mising values imputation.

\begin{algorithm}[htbp]
    \caption{Federated EM algorithm for distributed missing data imputation}
    \label{algo:fedImput}
  \begin{algorithmic}[1]
   \STATE {\bfseries Input:} $\kmax \in \nset$; for $c \in [n]^\star$,
   $V_{0}^c \in \rset^{I\times J}$; $\hatS_0 \in \rset^{I\times J}$; a positive sequence
   $\{\pas_{k+1}, k \in [\kmax-1]\}$; $\alpha$; the
   quantization function $\Q$ \STATE {\bfseries Output:} The \FEDEM~
   sequence: $\{\hatS_{k}, k \in [\kmax]\}$ \FOR{$k=0, \ldots,
     \kmax-1$} \FOR{$c=1, \ldots, n$} \STATE {\em (agent $\# i$, locally)}
     \STATE Initialize  $\Smem_{k+1,c}=0$ and $\Delta_{k+1,c}=0$ everywhere.
   \STATE Sample a minibatch $(\mathcal{I}_{k}^c, \mathcal{J}_{k}^c)\subset [I]^{\star}\times [J]^{\star}$
   		\FOR{$i \in \mathcal{I}_{k}^c$}
   		\FOR{$j \in \mathcal{J}_{k}^c$}
   		\STATE Set $(\Smem_{k+1}^c)_{i,j} =\mathsf{1}_{i,j\in\Omega^c}Y^c_{i,j} + (1-\mathsf{1}_{i,j\in\Omega^c})( \hat\theta_k)_{i,j}$
   		\STATE Set $(\Delta_{k+1}^c)_{i,j} =(\Smem_{k+1}^c)_{i,j}  - \hatS_{i,j} - (V_{k}^c)_{i,j}$
   		\ENDFOR
   		\ENDFOR
   		\STATE Update $V_{k+1}^c = V_{k}^c + \alpha\, \Q(\Delta_{k+1,c})$
   		\STATE Send $\Q(\Delta_{k+1}^c)$ to the controller
\ENDFOR
 \STATE {\em (the controller)}
 \STATE Compute
   $H_{k+1} = V_k + n^{-1}\sum_{c=1}^n
   \Q(\Delta_{k+1}^c)$
   \STATE Set $\hatS_{k+1} = \hatS_k + \pas_{k+1}
   H_{k+1}$
   \STATE Set $V_{k+1} = V_k + \alpha  n^{-1}\sum_{c=1}^n
   \Q(\Delta_{k+1}^c)$.
   \STATE Send $\hatS_{k+1}$ and $\hat{\theta}_{k+1} = \map
   (\hatS_{k+1})$ to the agents
   \STATE {\em (Note: thresholded SVD for Gaussian model or computed iteratively for a general exponential family model)}
   \ENDFOR
       \end{algorithmic}
\end{algorithm}

\paragraph{eBird data information.} In our experiments, we used a sample of the eBird data set~\cite{eBird}, provided upon request by the Cornell Lab of Ornithology. We are not allowed to disclose the data itself, but we provide here the details to reproduce our experiments on the same data set, after requesting acess on the eBird platform (https://ebird.org/data/request). We selected the counts recorded anywhere in France, between January 2000 and September 2020, for two different species: the Mallard and the Common Buzzard. These two species were analyzed independently (see \Cref{sec:numerical}); the corresponding code is also available as supplementary material.

\end{document}